\documentclass[12 pt]{amsart}

\usepackage{hyperref}
\usepackage{etex}
\usepackage[shortlabels]{enumitem}
\usepackage{amsmath}
\usepackage{amsxtra}
\usepackage{amscd}
\usepackage{amsthm}
\usepackage{adjustbox}
\usepackage{amsfonts}
\usepackage{amssymb}
\usepackage{eucal}
\usepackage[all]{xy}
\usepackage{graphicx}
\usepackage{tikz-cd}
\usepackage{mathrsfs}
\usepackage{subfiles}
\usepackage{mathpazo}
\usepackage[colorinlistoftodos, textsize=tiny]{todonotes}
\usepackage{morefloats}
\usepackage{pdfpages}
\usepackage{thm-restate}
\usepackage[percent]{overpic}
\usepackage[utf8]{inputenc}
\usepackage{epigraph}
\usepackage{csquotes}
\usepackage[margin=1in]{geometry}
\usepackage{adjustbox}
\usepackage{microtype}
\usepackage{verbatim}
\usepackage{stmaryrd}
\usepackage{scalerel}
\usepackage{stackengine}
\stackMath
\newcommand\reallywidehat[1]{%
\savestack{\tmpbox}{\stretchto{%
  \scaleto{%
    \scalerel*[\widthof{\ensuremath{#1}}]{\kern-.6pt\bigwedge\kern-.6pt}%
    {\rule[-\textheight/2]{1ex}{\textheight}}
  }{\textheight}%
}{0.5ex}}%
\stackon[1pt]{#1}{\tmpbox}%
}
\graphicspath{ {images/} }

\RequirePackage{color}
\definecolor{myred}{rgb}{0.75,0,0}
\definecolor{mygreen}{rgb}{0,0.5,0}
\definecolor{myblue}{rgb}{0,0,0.65}

\usepackage{color}

\usepackage{hyperref}
\hypersetup{citecolor=blue}
\usepackage{tikz}
\usetikzlibrary{matrix,arrows,decorations.pathmorphing}

\theoremstyle{plain}
\newtheorem{theorem}[subsubsection]{Theorem}
\newtheorem{proposition}[subsubsection]{Proposition}
\newtheorem{lemma}[subsubsection]{Lemma}
\newtheorem{corollary}[subsubsection]{Corollary}

\theoremstyle{definition}
\newtheorem{definition}[subsubsection]{Definition}
\newtheorem{remark}[subsubsection]{Remark}
\newtheorem{example}[subsubsection]{Example}

\newtheorem{question}[subsubsection]{Question}
\newtheorem{conjecture}[subsubsection]{Conjecture}

\theoremstyle{remark}
\newtheorem{notation}[subsubsection]{Notation}

\numberwithin{equation}{section}
\newcommand\nc{\newcommand}
\nc\on{\operatorname}
\nc\renc{\renewcommand}
\DeclareMathOperator\rk{rk}

\DeclareMathOperator\Mod{Mod}

\DeclareMathOperator\PW{PW}

\DeclareMathOperator\id{id}
\DeclareMathOperator\Hom{Hom}

\title{Canonical representations of surface groups}
\author{Aaron Landesman, Daniel Litt}
\date{\today}

\begin{document}

\begin{abstract}
Let $\Sigma_{g,n}$ be an orientable surface of genus $g$ with $n$ punctures. We study actions of the mapping class group $\on{Mod}_{g,n}$ of 
$\Sigma_{g,n}$ via Hodge-theoretic and arithmetic techniques.  We show that if
$$\rho: \pi_1(\Sigma_{g,n})\to \on{GL}_r(\mathbb{C})$$ is a representation whose
conjugacy class has finite orbit under $\on{Mod}_{g,n}$, and $r<\sqrt{g+1}$,
then $\rho$ has finite image. 
This answers questions of Junho Peter Whang and Mark Kisin.  We give applications of our methods to 
the Putman-Wieland conjecture,
the Fontaine-Mazur conjecture, 
and a question of Esnault-Kerz. 

The proofs rely on non-abelian Hodge theory, our earlier work on semistability
of isomonodromic deformations, and recent work of Esnault-Groechenig and Klevdal-Patrikis
on Simpson's integrality conjecture for cohomologically rigid local systems.
\end{abstract}

\maketitle
\setcounter{tocdepth}{1}
\tableofcontents

\section{Introduction}
\label{section:introduction}
\subsection{Overview}
\label{subsection:overview}
Let $\Sigma_{g,n}$ be an orientable topological surface of genus $g$ with $n$ punctures, and let $x\in\Sigma_{g,n}$ be a basepoint.
We denote the mapping class group of $\Sigma_{g,n}$ by $\on{Mod}_{g,n}$. 
There is a natural outer action of $\on{Mod}_{g,n}$ on $\pi_1(\Sigma_{g,n},x)$, induced by the action of $\on{Homeo}^+(\Sigma_{g,n})$ on $\Sigma_{g,n}$. 
Hence, $\on{Mod}_{g,n}$ acts on the set of conjugacy classes of
representations of $\pi_1(\Sigma_{g,n},x)$ into $\on{GL}_r(\mathbb{C})$. Our goal is to study the finite orbits of this action
via Hodge-theoretic and arithmetic techniques. 

\begin{definition}
	\label{definition:mcg-finite}
	For $g,n,r \geq 0$,
	a representation $\rho: \pi_1(\Sigma_{g,n},x) \to \on{GL}_r(\mathbb C)$
	is {\em MCG-finite} if the conjugacy class of $\rho$ has finite orbit under the action of
	$\on{Mod}_{g,n}$.
\end{definition}
Note here that we are studying \emph{conjugacy classes} of representations. For
example, semisimple MCG-finite representations correspond to finite
$\on{Mod}_{g,n}$-orbits in the character variety of $\pi_1(\Sigma_{g,n})$, not
its representation variety. Indeed, $\on{Mod}_{g,n}$ does not naturally act on
the representation variety of $\pi_1(\Sigma_{g,n})$.

Any representation of $\pi_1(\Sigma_{g,n})$ constructed without making choices
(e.g.~without naming a specific curve in $\Sigma_{g,n}$),
or by making choices with only a finite amount of indeterminacy, is MCG-finite.
So we view MCG-finite representations 
as the \emph{canonical} representations of $\pi_1(\Sigma_{g,n})$. 

The study of MCG-finite
representations has a long history.
It has a close
connection to isomonodromy and the Painlev\'e VI equation,
originally introduced in 1902 by Painlev\'e 
\cite{painleve:on-the-differential-equations-of-sectond-order}
and Gambier
\cite{gambier:on-the-differential-equations-of-second-order}.
For each $g,n,r$, conjugacy classes of semisimple MCG-finite representations 
$$\pi_1(\Sigma_{g,n})\to \on{GL}_r(\mathbb{C})$$ correspond to algebraic solutions
to a certain \emph{isomonodromy} differential equation, see \cite[Theorem A]{CH:isomonodromic}. 
Algebraic solutions to the Painlev\'e VI equation in an appropriate choice of coordinates 
correspond to MCG-finite representations with trivial determinant in
the particular case $g=0, n=4, r=2$ (see
\cite[\S1-\S3]{mahoux:itroduction-to-the-theory-of-isomonodromic-deformations}
and \cite[\S2]{doran:isomonodromic}). 

There is substantial literature on the
search for algebraic solutions 
to the Painlev\'e VI equation,
\cite{andreev2002transformations, boalch2005klein, boalch2006fifty, boalchsome, boalch2007higher, cantat2009holomorphic, dubrovin1996geometry, dubrovin2000monodromy, hitchin1996poncelet, hitchin2003lecture, kitaev2006grothendieck, kitaev2005remarks},
culminating in the classification given in \cite{lisovyy2014algebraic}. Boalch \cite{boalch2007towards} pitches the problem of classifying algebraic solutions to Painlev\'e VI as a natural generalization of Schwarz's list of hypergeometric equations with finite monodromy \cite{schwarz1873ueber}.

We will later see that the classification of MCG-finite representations is
closely connected to several major open questions in low-dimensional topology.
In particular, it has connections to the Putman-Wieland conjecture
\cite[Conjecture 1.2]{putmanW:abelian-quotients} and Ivanov's conjecture that mapping class groups do not virtually surject onto $\mathbb{Z}$ \cite[\S7]{Ivanov:problems}.
These conjectures in turn have applications to algebraic geometry, as they would
determine the first rational homology of finite covers of $\mathscr M_{g,n}$.

Many interesting classes of
representations are MCG-finite.
Examples include 
the rigid local systems studied by Katz \cite{katz2016rigid}, 
representations constructed via TQFT techniques ~(e.g.~\cite{kuperberg2011denseness, koberda2016quotients, koberda2018irreducibility, koberda2018representation}),
and representations of algebro-geometric interest, such as those constructed via the Kodaira-Parshin trick (e.g.~\cite[Example 3.3.1]{lawrence2019representations}).
We will see later that MCG-finite representations of $\pi_1(\Sigma_{g,n})$ are closely related to representations of finite-index subgroups of $\on{Mod}_{g,n+1}$. 

\subsection{Results for surface groups}
\label{subsection:surface-groups}
Given the difficulty of classifying MCG-finite representations in the $g = 0, n = 4, r = 2$ case (encompassing Painlev\'e VI), it may be surprising that
 we are able to obtain a complete and simple characterization of MCG-finite representations when 
$g \gg r$.

\begin{theorem}\label{theorem:finite-image}
	For $g, n, r\geq 0$, let
	$$\rho: \pi_1(\Sigma_{g,n})\to \on{GL}_r(\mathbb{C})$$ be a MCG-finite
representation. If $r<\sqrt{g+1}$, then $\rho$ has finite image. 
\end{theorem}
	This answers a question of Junho Peter Whang \cite[Question
	1.5.3]{lawrence2019representations}. We will prove
	\autoref{theorem:finite-image} in
\autoref{subsection:main-proof}.
	
	While the statement appears to be purely topological or even group-theoretic in nature, the proof relies on (non-abelian and mixed) Hodge theory, and also takes input from the Langlands program, through the work of Esnault-Groechenig \cite{esnault2018cohomologically} 
and Klevdal-Patrikis \cite{klevdalP:g-rigid-local-systems-are-integral}
on integrality of cohomologically rigid local systems.
	
\begin{remark}
	\label{remark:}
	Some bound on the rank as in \autoref{theorem:finite-image} is
	necessary. Indeed, there are a number of interesting MCG-finite representations of $\pi_1(\Sigma_{g,n})$ of rank $r\gg g$ with infinite image, as explained in \autoref{question:bounds}. 
	We have no reason to think that the precise bound of
	$\sqrt{g+1}$ is optimal in general. 
	That said the bound cannot be improved too much. It is sharp when $g=1$,
	by \autoref{remark:sharp-g-1}, but not for $g=2,3$, by \autoref{remark:not-sharp-g-2-3}.
	For general $g$ there are MCG-finite representations of $\pi_1(\Sigma_g)$ of rank  
	$2g+1$ with infinite image (\autoref{non-rigid-examples}\eqref{h1-example}), so the bound cannot be improved to be better than linear in $g$.
	We summarize the situation in \autoref{figure:MCG-geography}.
\end{remark}

\begin{figure}
	\centering
	\includegraphics[scale=.7, trim={2cm 18.0cm 0 1cm}]{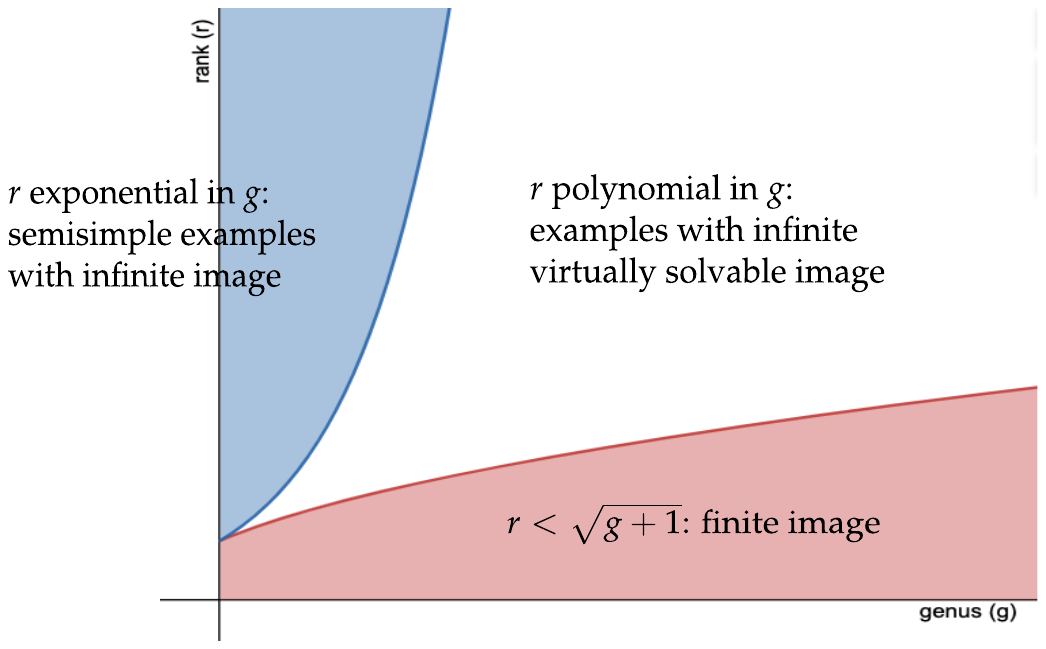}
\caption{The geography of known MCG-finite representations.
\autoref{theorem:finite-image} shows that the region $r<\sqrt{g+1}$ contains
only MCG-finite representations with finite image. As explained in
\autoref{question:bounds}, there are examples of MCG-finite rank $r$ representations with infinite image for $r\gg g$. For $r$ polynomial in $g$, all known examples have virtually solvable image. For $r$ exponential in $g$, there are interesting MCG-finite irreps with infinite image.}\label{figure:MCG-geography}
\end{figure}

As explained in ~\cite[Theorem A]{CH:isomonodromic}, 
\autoref{theorem:finite-image} yields a classification result for algebraic solutions to isomonodromy differential equations:
\begin{corollary}
	Let $C$ be a smooth projective curve of genus $g$, and let $D\subset C$
	be a reduced effective divisor. Let $({E},\nabla)$ be a semisimple flat
	vector bundle on $C$ with regular singularities along $D$, with $\rk
	{E}<\sqrt{g+1}$. Suppose that the eigenvalues of the residue matrices of
	$({E},\nabla)$ have real parts in $[0,1)$. Then $({E},\nabla)$ has an algebraic universal isomonodromic deformation if and only if $({E},\nabla)$ has finite monodromy.
\end{corollary}
Note that the condition on residue matrices can always be achieved after a birational gauge transformation, by replacing $({E}, \nabla)$ with the Deligne canonical extension of $({E}, \nabla)|_{C\setminus D}$.

This algebro-geometric reformulation of our main result is not just window
dressing. It is the perspective we will take in the proof of \autoref{theorem:finite-image}. 

\subsection{An application to low rank local systems on $\mathscr M_{g,n}$}

We next record a consequence of our main result for local systems of low rank on families of curves.
As in \autoref{notation:versal-family}, we say a family $\pi:\mathscr C \to \mathscr M$ of smooth proper genus
$g$ curves with geometrically connected fibers, equipped with $n$ disjoint sections $s_i: \mathscr M \to
\mathscr C$, is {\em versal} if the corresponding map $\mathscr M \to \mathscr M_{g,n}$
is dominant. We call 
$\pi^\circ: \mathscr C^\circ := \mathscr C \setminus \bigcup_i s_i(\mathscr M)\to \mathscr
M$ a {\em punctured versal family} of genus $g$ curves.

The following is immediate from \autoref{theorem:finite-image}, by \autoref{proposition:MCG-finite-family-of-curves} below.
\begin{corollary}
	\label{corollary:}
	Suppose $\pi^\circ: \mathscr C^\circ \to \mathscr M$
	is a punctured versal family of genus $g$ curves,
	and $\mathbb V$ is a local system on $\mathscr C^\circ$ of rank $<
	\sqrt{g+1}$.
	For any fiber $C^\circ$ of $\pi^\circ$, the local system $\mathbb
	V|_{C^\circ}$ has finite monodromy.
\end{corollary}

\subsection{The Putman-Wieland conjecture}
\label{subsection:intro-putman-wieland}

Among the main ingredients in the proof of \autoref{theorem:finite-image} are some new results toward the Putman-Wieland conjecture, sketched below and
described in detail in \autoref{section:Putman-Wieland}.

We now set some notation to state the Putman-Wieland conjecture.
Let $\Sigma_{g,n}$ be a surface and let $\phi: \pi_1(\Sigma_{g,n},
x)\twoheadrightarrow H$ be a surjection onto a finite group, corresponding to
some finite $H$-cover $\Sigma_{g'}\to \Sigma_{g}$ branched over $n$ points. Let $\Gamma$ be the  stabilizer of $\phi$ up to conjugacy in the pure mapping class group of $\Sigma_{g,n+1}$. Then $\Gamma$ acts naturally acts on $H_1(\Sigma_{g'})$. 

We say that {\em the Putman-Wieland conjecture for $(g,n,H)$ holds} if
for every such $H$-cover $\Sigma_{g'} \to \Sigma_g$, 
all nonzero vectors in $H_1(\Sigma_{g'})$ have infinite orbit under $\Gamma$.
The original Putman-Wieland conjecture 
\cite[Conjecture 1.2]{putmanW:abelian-quotients}
asserts that for any fixed $g \geq 2, n \geq 0$,
the Putman-Wieland conjecture for $(g,n,H)$
holds for every $H$.
We prove 
Putman-Wieland conjecture for $(g,n,H)$ for any fixed $H$ once $g$
is sufficiently large. 
In this way, our results are an ``asymptotic" version of the Putman-Wieland conjecture. 
\begin{theorem}\label{theorem:intro-PW}
	For $g \geq 2, n \geq 0$, and any finite group $H$ with
	$\#H<g^2$,
	the Putman-Wieland conjecture for $(g,n,H)$ holds.
\end{theorem}
\autoref{theorem:intro-PW} is proven below as a case of \autoref{corollary:asymptotic-putman-wieland}.
More generally, for arbitrary $H$ we prove that the Putman-Wieland conjecture holds for the subspace of $H_1(\Sigma_{g'})$ spanned by irreducible representations of $H$ of small rank.
\begin{theorem}\label{theorem:rho-isotypic-PW-intro}
	Let $\rho$ be an irreducible representation of $H$ with $\dim\rho<g$, and let $H_1(\Sigma_{g'})^{\rho}$ be the $\rho$-isotypic component of $H_1(\Sigma_{g'})$. Then no nonzero vector in $H_1(\Sigma_{g'})^{\rho}$ has finite orbit under $\Gamma$.
\end{theorem}
\autoref{theorem:rho-isotypic-PW-intro} is proven below as a special case of
\autoref{theorem:asymptotic-putman-wieland}.

As we shall see in the course of the proof of \autoref{theorem:finite-image},
especially \autoref{lemma:semisimplicity},
the Putman-Wieland conjecture is connected to the classification of MCG-finite representations with virtually solvable image.

\subsection{Arithmetic applications}
\label{subsection:arithmetic}
In \autoref{section:arithmetic-applications}, we give a number of applications
to questions in arithmetic geometry. For example, we show in
\autoref{theorem:arithmetic-consequence} that low rank representations of the
arithmetic \'etale fundamental group of a generic curve of genus $g$ have finite
image when restricted to the geometric fundamental group. This verifies a
prediction of the Fontaine-Mazur conjecture, see
\autoref{remark:fontaine-mazur}. 
These results also answer a question of Esnault-Kerz, see \autoref{remark:esnault-kerz}.

As a consequence, we construct many residual
representations of the geometric fundamental group of a generic curve of genus
$g$ which have no lifts to representations of geometric origin. These examples are related to a conjecture of de Jong \cite{de2001conjecture} (proven by Gaitsgory \cite{gaitsgory2007jong}) and a question of Flach, see \autoref{remark:complete-intersection} and \autoref{remark:flach}.

\subsection{A consequence for free groups}
\label{subsection:free-groups}

From \autoref{theorem:finite-image}, we deduce an analogous result for free groups.
\begin{corollary}\label{corollary:free-groups}
	Let $F_N$ be a free group on $N$ generators, with $N=2g$ or $N=2g+1$.
	Let $$\rho: F_N\to \on{GL}_r(\mathbb{C})$$ be a representation whose conjugacy class has finite orbit under $\on{Out}(F_N)$. If $r<\sqrt{g+1}$, then $\rho$ has finite image.
\end{corollary}
\begin{proof}
	If $N$ is even, choose an isomorphism $F_N\simeq \pi_1(\Sigma_{g,1})$
	and if $N$ is odd choose $F_N\simeq \pi_1(\Sigma_{g,2})$. Any
	representation of $F_N$ with finite orbit under $\on{Out}(F_N)$ yields a
	representation of a genus $g$ surface group with finite orbit under the
	mapping class group. Such a representation has finite image by \autoref{theorem:finite-image}.
\end{proof}
\begin{remark}
It is natural to ask if the bound on $r$ in \autoref{corollary:free-groups} is sharp. 
There exist (non-semisimple) representations of $F_N$ of fairly low rank ($r =N + 1$) with infinite image and finite orbit under $\on{Out}(F_N)$, see 
\autoref{example:free-group-bounds}, so the bound cannot be improved too much. 
However, in contrast with the case of MCG-finite representations, we do not know any examples of \emph{semisimple} representations of $F_N$ with infinite image and finite orbit under $\on{Out}(F_N)$ as soon as $N\geq 3$.
Indeed, a conjecture of Grunewald and Lubotzky \cite[Conjecture in
\S9.2]{grunewald2009linear}, combined with the main result of
\cite{farb2017moving}, implies that no such examples exist when $N \geq 3$. In the interests of provocation, we conjecture 
(\autoref{conjecture:infinite-image-semisimple})
that such examples \emph{do} exist. 
As evidence, we offer several examples of semisimple MCG-finite representations of surface groups with infinite image, see \autoref{question:bounds}. Note that there are interesting semisimple representations of $F_2$ with finite orbit under $\on{Out}(F_2)$, see \autoref{example:free-group-bounds}.
\end{remark}

\begin{remark}
	\label{remark:explicit-free}
One appeal of \autoref{corollary:free-groups} lies in the fact that it admits a completely elementary reformulation, using Nielsen's description of $\on{Aut}(F_N)$,
see \cite{nielsen:die-isomorphismengruppe} 
and also \cite[Theorem 3.2, p. 131]{magnusKS:combinatorial-group-theory}.
Let $(A_1, \cdots, A_N)$ be an $N$-tuple of invertible $r\times r$ complex matrices, with $N=2g$ or $N=2g+1$. Consider the following operations on $N$-tuples of invertible matrices:
\begin{enumerate}
	\item The cyclic permutation $$c: (A_1, A_2, \cdots, A_N)\mapsto (A_2, A_3, \cdots, A_N, A_1)$$
	\item The transposition $$\tau: (A_1, A_2, A_3, \cdots, A_N)\mapsto (A_2, A_1, A_3, \cdots, A_N)$$
	\item The inversion map  $$\epsilon: (A_1, A_2, \cdots, A_N)\mapsto (A_1^{-1},A_2, \cdots, A_N)$$
	\item The Dehn twist $$d: (A_1, A_2, \cdots, A_N)\mapsto (A_1A_2, A_2, \cdots, A_N)$$
\end{enumerate}
	We say $(A_1, \cdots, A_N)$ and $(A_1', \cdots, A_N')$ are \emph{conjugate} if there exists some $B$ such that $A_i=BA_i'B^{-1}$ for all $i$. 
	
	Now suppose that the set of $N$-tuples obtained from $(A_1, \cdots, A_N)$ by repeatedly applying $c, \tau, \epsilon$, and $d$ only intersect finitely many conjugacy classes of $N$-tuples
	$(A_1, \cdots, A_N)$. 
	If $r<\sqrt{g+1}$,
	\autoref{corollary:free-groups} implies that $A_1, \ldots, A_N$ generate a finite subgroup of $\on{GL}_r(\mathbb{C})$. 
\end{remark}
\begin{remark}
	\label{remark:}
	Similarly to \autoref{remark:explicit-free}, one could make
	\autoref{theorem:finite-image} explicit, using any one of the known
	generating sets of $\on{Mod}_{g,n}$, but the formulas are a bit more
	involved, as in \cite[\S6]{CH:isomonodromic}. It would in our view be of great interest to find a \emph{proof} of \autoref{theorem:finite-image} or \autoref{corollary:free-groups} of a similarly explicit nature.
\end{remark}
	
\begin{remark}
A result analogous to \autoref{corollary:free-groups} follows immediately for
characteristic quotients of free or surface groups. For example, let $G$ be a
characteristic quotient of a free group $F_N$ on $N=2g$ or $N=2g+1$ generators,
and $\rho: G\to \on{GL}_r(\mathbb{C})$ a representation whose conjugacy class has finite orbit under $\on{Out}(G)$. If $r<\sqrt{g+1}$, then $\rho$ has finite image by \autoref{corollary:free-groups}.
\end{remark}

\subsection{Cohomological results}
\label{subsection:cohomological}
One of the new technical inputs we use to achieve \autoref{theorem:finite-image} is
an analysis of the cohomology of unitary local systems on families of curves.
Our main result, which follows from an analysis of the derivative of the period
map associated to the mixed Hodge structure on the cohomology of unitary local
systems, is the following.
We adapt notation described later in \autoref{notation:versal-family}.
\begin{theorem}\label{theorem:rank-bound-subsystem}
Let $\pi:\mathscr{C}\to \mathscr{M}$ be a smooth proper family of $n$-pointed
curves of genus $g$ with geometrically connected fibers, so that the associated
map $\mathscr{M}\to \mathscr{M}_{g,n}$ is dominant \'etale. Let $\pi^\circ:
\mathscr{C}^\circ\to \mathscr{M}$ be the associated family of punctured curves.
Let $\mathbb{V}$ be a unitary local system on $\mathscr{C}^\circ$. Then any
nonzero sub-local system of $R^1\pi_*^\circ \mathbb{V}$ has rank at least $2g-2\rk\mathbb{V}$. 	
\end{theorem}
We prove this in \autoref{subsection:rank-bound-proof}, and give a related cohomological vanishing theorem with milder unitarity hypotheses in \autoref{theorem:unitary-rigid}.
Note that \autoref{theorem:rank-bound-subsystem} shows that $H^0(\mathscr{M}, R^1\pi_*^\circ \mathbb{V})=0$ if
$g>\on{rk}\mathbb{V}$. 
This result is used several times in the proof of
\autoref{theorem:finite-image}, as described in
\autoref{subsection:outline-of-proof}.
It is especially interesting when $\mathbb{V}$ has finite monodromy, where it is closely related to the Putman-Wieland conjecture, as described
above in \autoref{subsection:intro-putman-wieland} and in more detail in \autoref{section:Putman-Wieland}.

\subsection{Prior work}
\label{subsection:prior-work}
\subsubsection{MCG-finite representations and Painlev\'e VI}
\label{subsubsection:prior-mcg-finite}
As mentioned in \autoref{subsection:overview}, there has been a huge amount of effort put forth towards classifying algebraic solutions to the Painlev\'e VI equation, arguably beginning with work of Riemann \cite{riemann1857beitrage} and Schwarz \cite{schwarz1873ueber}
and culminating in the classification \cite{lisovyy2014algebraic}. 
Calligaris-Mazzocco \cite{calligaris2018finite} classified algebraic
solutions satisfying several conditions in the case $g=0, n=5, r=2$, but the complete classification in
this case remains open. See also \cite{michel-hurwitz} for a related result in the genus zero case, when the local monodromy matrices are given by reflections; this appears to be the first result showing that MCG-finite representations (satisfying certain conditions) have finite image.

For all the work invested in the algebraic solutions to Painlev\'e VI, this only
addresses the case of classifying  MCG-finite representations with $g = 0, n = 4,
r = 2$, and trivial determinant, as explained in \cite{lisovyy2014algebraic}.
In higher genus, all work that we know of has been on the case of
$2$-dimensional representations.
The beautiful paper \cite{biswas2017surface} handles the $2$-dimensional case ($r = 2, g \geq 1, n \geq 0$) with trivial
determinant.
The non-semisimple case $r = 2, g \geq 1, n \geq 0$ with arbitrary determinant
is explained in \cite[Theorem B]{CH:isomonodromic}.
These results appear to depend crucially on the assumption $r=2$.

\begin{remark}
	\label{remark:}
	The above-mentioned results in rank $2$ led Junho Peter Whang to ask whether, for $g\gg r$, all MCG-finite rank $r$ representations of $\Sigma_{g,n}$
have finite image \cite[Question 1.5.3]{lawrence2019representations}.
More generally, motivated by the $p$-curvature conjecture, Mark Kisin asked whether MCG-finite representations
necessarily have finite image,
see \cite[p. 3]{biswas2017surface} and \cite[p. 1]{sinz:thesis}. (Note that
counterexamples are known in general, see ~\cite[Theorem
5.1]{biswas2018representations} for counterexamples coming from TQFT techniques or 
\cite[Example 3.3.1]{lawrence2019representations} for counterexamples coming
from the Kodaira-Parshin trick, as well as
\autoref{question:bounds} of this paper.)
In this way, \autoref{theorem:finite-image} answers Whang's question
affirmatively, and provides a positive answer to Kisin's question in the regime
$g > r^2-1$.
\end{remark}

There has also been much work put towards finding interesting \emph{examples} of MCG-finite representations, notably \cite{doran:isomonodromic, diarra2013construction, girand:a-new-two-parameter-family, diarra2015ramified, girand2016algebraic}. 

In a more arithmetic direction, \cite{bourgainGS:announcement} 
announced striking results on strong approximation for Markoff triples, obtained by analyzing the
arithmetic properties of subvarieties of the character variety parametrizing $2$-dimensional
representations of $\pi_1(\Sigma_{1,1})$.
The first step in their analysis is to determine the finite orbits of the
mapping class group action on this character variety \cite[last paragraph of p.
132]{bourgainGS:announcement}.
In this way, our \autoref{theorem:finite-image} can be viewed as a necessary
first step 
toward attempting to generalize their approach to 
higher rank, higher genus character
varieties.

\subsubsection{MCG-finite representations and geometric topology}

In the world of low-dimensional topology a number of authors have studied the
dynamics of mapping class group actions on character varieties; finite orbits
are the same as semisimple MCG-finite representations. Kasahara
\cite{kasahara2015visualization} relates fixed points of this action
corresponding to faithful representations to the well-known question of
linearity of the mapping class group. Goldman and many other authors have
studied ergodicity of these actions, see \cite{goldmanmapping} and the
references therein. Previte-Xia \cite{previte2000topological,
previte2002topological} study the relationship between density of the image of
an $SU(2)$-representation and density of its mapping class group orbit. In some
sense our main result is a partial answer to \cite[Question
2.7]{goldmanmapping}, which asks for necessary and sufficient conditions for a
representation to have dense orbit under the mapping class group. We
characterize the most extreme possible failure of density, namely the case of
finite orbit.

There are also a number of related results on representations of the mapping class group.
Farb, Lubotzky, and Minsky \cite[Theorem 1.6]{farbLM:rank-1-phenomena} show there are no faithful linear
representations of finite index subgroups of $\on{Mod}_{g,0}$ of dimension $< 2 \sqrt{g-1}$. See \autoref{remark:FLM-rank-1-result} for a comparison of our results to theirs.
See also
\cite{franksH:triviality,korkmaz:low-dimensional-linear-representations,funar:two-questions, kielak-pierro}
for bounds on the dimension of representations of 
$\on{Mod}_{g,n}$, although these results only address representations of the full mapping
class group, as opposed to representations of finite index subgroups.

\subsubsection{The Putman-Wieland conjecture} The other main contribution of
\label{subsubsection:prior-putman-wieland}
this paper, towards the Putman-Wieland conjecture, has a number of precursors,
notably \cite{looijenga:prym-representations,grunewald2015arithmetic, looijenga2021arithmetic}.
Our Hodge-theoretic approach is related to an approach suggested by Looijenga \cite{looiejenga:AG}. 
The strongest previous result towards the Putman-Wieland conjecture is perhaps 
the main result of \cite{grunewald2015arithmetic}, which says that for certain quotients of $\pi_1(\Sigma_g)$, 
called \emph{$\phi$-redundant} quotients,
much more than the Putman-Wieland conjecture is true --- the monodromy representations considered by the conjecture have very large image, commensurable with an arithmetic group. There has also been interesting recent work related to the Putman-Wieland conjecture by Markovi\'c and Markovi\'c-To\u{s}i\'c \cite{markovic, markovic2}.

The analogue of the Putman-Wieland conjecture for graphs, as opposed to surfaces, has been proven by Farb and Hensel \cite{farb2017moving}.

\subsection{Outline of the proof}\label{subsection:outline-of-proof}
We now sketch the proof of \autoref{theorem:finite-image}, which is loosely inspired by Katz's proof of the $p$-curvature conjecture for the Gauss-Manin connection \cite{katz:pcurvature-and-hodge}.
Following this, we sketch the proof of the vanishing results we use (\autoref{theorem:rank-bound-subsystem} and its consequence \autoref{theorem:unitary-rigid}) in
\autoref{subsubsection:rigidity-idea}, which are used in different ways in each of the three steps listed below of the proof of \autoref{theorem:finite-image}.
A schematic diagram outlining the main elements of the proof is depicted in
\autoref{figure:proof-schematic}.

\begin{figure}
	\centering
\adjustbox{scale=.75,center}{
   \begin{tikzcd}[column sep = 2.3em]
	\qquad &&
\text{Lem.}~\ref{lemma:gl-defined-over-number-field} \ar{d}&
  \text{Lem.}~\ref{lemma:defined-over-number-field} \ar{l} \\
  \qquad &&
  \underbrace{\text{Prop.}~\ref{proposition:unitary-implies-finite}}_{\text{Unitary
  case}}
  \ar{dl}{\text{\cite[Non-abelian Hodge theory]{mochizuki2006kobayashi}}} 
  \ar[swap]{dl}{\text{\cite[Isomonodromy]{LL:geometric-local-systems}}}  &
  \text{Prop.}~\ref{proposition:cohomologically-rigid} 
  \ar{l}{\text{\cite{mochizuki2006kobayashi}}} 
  \ar[swap]{l}{\text{\cite{LL:geometric-local-systems}}}
  \ar[swap]{u}{\text{\cite[Langlands]{esnault2018cohomologically, klevdalP:g-rigid-local-systems-are-integral}}} &
  \text{Prop.}~\ref{proposition:generic-parabolic-global-generation} \ar{d} \\
  \text{Thm.}~\ref{theorem:finite-image} &
  \underbrace{\text{Thm.}~\ref{theorem:finite-image-semisimple}}_{\text{Semisimple
  case}}\ar{l}&
  \text{Lem.}~\ref{lemma:global-rigidity} \ar{l} &
  \underbrace{\text{Thm.}~\ref{theorem:unitary-rigid}}_{\text{Cohomological
  vanishing}} \ar{l} \ar{u}&
  \text{Prop.}~\ref{proposition:pairing-result} \ar{dl}{\text{Mixed Hodge
  theory and \cite{mehta1980moduli}}} \\
  \qquad & \text{Lem.}~\ref{lemma:semisimplicity} \ar{ul}&
  \underbrace{\text{Thm.}~\ref{theorem:asymptotic-putman-wieland}}_{\text{Putman-Wieland}} \ar{l}&
	  \underbrace{\text{Thm.}~\ref{theorem:rank-bound-subsystem}}_{\text{Cohomological
	  rank bound}}\ar{l} \ar{u} 
	  &
	  \underbrace{\text{Thm.}~\ref{theorem:period-map-computation}}_{\text{Period map}}
	  \ar{l}
	  \end{tikzcd}
}
\caption{
A diagram depicting the structure of the proof of
the main result,
\autoref{theorem:finite-image}, in the shape of a boat.}
\label{figure:proof-schematic}
\end{figure}
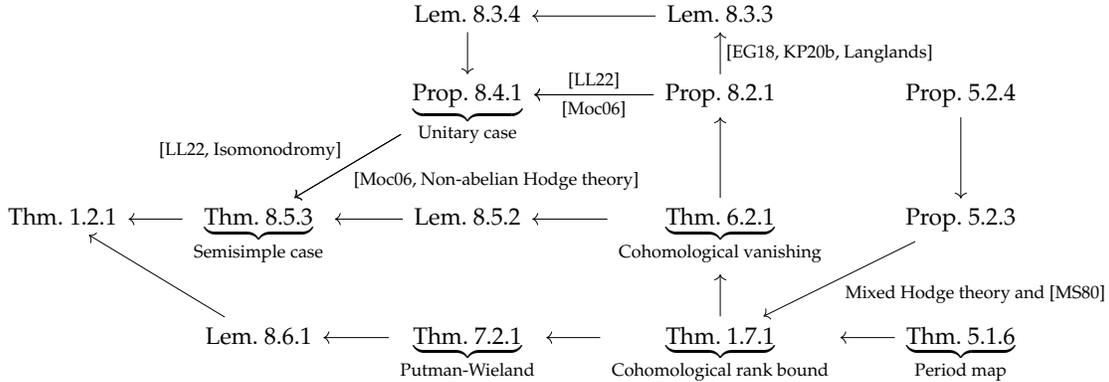

\subsubsection{Step 1. The unitary case.} 
\label{subsubsection:step-1}
Every unitary representation is a direct sum of irreducible unitary representations.
Therefore, we suppose $\rho$ is unitary, irreducible, and MCG-finite, 
with $\rk\rho<\sqrt{g+1}$. We construct from $\rho$ a finite
\'etale cover $\mathscr{M}$ of $\mathscr{M}_{g,n}$, with associated family of
punctured curves $\pi^\circ:\mathscr{C}^\circ\to \mathscr{M}$, and a projective unitary local
system $\mathbb{V}$ on $\mathscr{C}^\circ$ whose restriction to a fiber
$C^\circ$ of
$\pi^\circ$ has monodromy given by $\rho$. 
Applying \autoref{proposition:cohomologically-rigid}, (which is a fairly 
straightforward consequence of our cohomological vanishing result,
\autoref{theorem:unitary-rigid}, applied to $\on{ad}\mathbb{V}$), 
shows that $\mathbb{V}$ is cohomologically rigid. 
The main result of \cite{klevdalP:g-rigid-local-systems-are-integral} then gives that $\rho$ is defined over the ring of integers $\mathscr{O}_K$ of some number field $K$.
Moreover, using \autoref{proposition:gl-lift},
by replacing $\mathscr M$ by a dominant \'etale scheme over it, along which
certain cohomological lifting obstructions vanish, we can assume
$\mathbb V$ lifts from a projective local system to a bona fide local system.

By compactness of the unitary group and discreteness of $\mathscr{O}_K$, it
suffices to show that for each embedding $\iota:
\mathscr{O}_K\hookrightarrow\mathbb{C}$, $\rho\otimes_{\mathscr{O}_K, \iota}
\mathbb{C}$ is unitary. (We know this unitarity for one such $\iota$ by assumption, but not
the others.) The rigidity of $\mathbb{V}$ implies by non-abelian Hodge theory
that these $\rho\otimes_{\mathscr{O}_K, \iota} \mathbb{C}$ underlie complex polarizable variations of Hodge structure for any complex structure on $\Sigma_{g,n}$. 
Hence, given their low rank, these local systems are unitary by \cite[Theorem
1.2.12]{LL:geometric-local-systems}.

\subsubsection{Step 2. The semisimple case.} 
\label{subsubsection:step-2}
Now suppose $\rho$ is an arbitrary
semisimple, MCG-finite representation with $\rk\rho<\sqrt{g+1}$. Again, we
associate to $\rho$ a local system $\mathbb{V}$ on a family of curves
$\pi^\circ:\mathscr{C}^\circ\to \mathscr{M}$, so that $\mathscr M$ has a
dominant \'etale map to $\mathscr M_{g,n}$ and whose fibral monodromy is given by $\rho$.
By non-abelian Hodge theory, we may deform $\mathbb{V}$ to a local system
$\mathbb V_0$ underlying a complex polarizable
variation of Hodge structure. By
\cite[Theorem 1.2.12]{LL:geometric-local-systems}, $\mathbb{V}_0$ has unitary monodromy when
restricted to a fiber of $\pi^\circ$. By the unitary case,
\autoref{subsubsection:step-1},
$\mathbb{V}_0$ thus has finite monodromy when restricted to a fiber of $\pi^{\circ}$. Note that even if $n=0$, i.e.~$\pi^\circ$ is proper, we here need to use non-abelian Hodge theory for non-proper varieties, as the total space $\mathscr{C}^\circ$ will not be proper.

It remains to argue that the restriction $\mathbb V_0|_{C^\circ}$ of
$\mathbb V_0$ to a fiber $C^\circ$ of
$\pi^\circ$ agrees with the restriction $\mathbb V|_{C^\circ}$,
corresponding to $\rho$.
Recall that 
$\mathbb V_0$ and $\mathbb V$ were only deformation equivalent, so it may be
surprising that they necessarily restrict to the same local system on fibers.
We verify this agreement 
in \autoref{lemma:global-rigidity}
through another application of
\autoref{theorem:unitary-rigid}, which 
which tells us that since the fibral monodromy $\mathbb{V}_0|_{C^\circ}$ 
is unitary of low rank, it
does not admit non-trivial MCG-finite deformations. 

Note that $\mathbb{V}_0$ may not a priori be unitary; we only know that $\mathbb V_0|_{C^\circ}$ is unitary. In particular, it is not clear whether $\mathbb{V}_0$ is necessarily cohomologically rigid.

\subsubsection{Step 3. The general case} 
\label{subsubsection:step-3}
The crucial input for dealing with
non-semisimple representations is our work towards the
Putman-Wieland conjecture. By the above it is enough to show that low-rank
MCG-finite representations are semisimple, i.e.~we wish to verify that certain
extensions split. 
For simplicity, let's suppose for the purpose of this sketch that $\rho$ is an extension of two irreducible MCG-finite representations
$\rho_1$ and $\rho_2$. 
By the previous step, $\rho_1$ and $\rho_2$ have finite monodromy, so after
passing to a finite cover $\Sigma_{g',n'}$ of $\Sigma_{g,n}$, we may assume
$\rho_1$ and $\rho_2$ have trivial monodromy.
The splitting of this extension of $\rho_2$ by $\rho_1$ corresponds to the vanishing of a certain
element in $\on{Ext}^1_{\pi_1(\Sigma_{g',n'})}(\rho_1, \rho_2)$.
Because we arranged that $\rho_1$ and $\rho_2$ have trivial monodromy
on $\pi_1(\Sigma_{g',n'})$, the above extension class corresponds to a map
$\pi_1(\Sigma_{g',n'}) \to \rho_1^\vee \otimes \rho_2$, with unipotent abelian
image. Hence, it factors through $H_1(\Sigma_{g',n'})$,
and so defines a low rank subspace of $H^1(\Sigma_{g',n'})$, stable under a
finite index subgroup of $\on{Mod}_{g,n+1}$.
We verify in \autoref{lemma:semisimplicity}, 
using our main result toward the Putman-Wieland conjecture,
\autoref{theorem:asymptotic-putman-wieland},
that such a subspace cannot exist. Theorem \autoref{theorem:asymptotic-putman-wieland} itself follows more or less immediately from \autoref{theorem:rank-bound-subsystem}.


\subsubsection{The proof of our cohomological vanishing results}
\label{subsubsection:rigidity-idea}
Note that we used \autoref{theorem:rank-bound-subsystem}, or its immediate consequence,
\autoref{theorem:unitary-rigid},
in every step above.
Before explaining the idea of the proof of these results, let us recall the
setting. We begin with a scheme $\mathscr M$ with a dominant \'etale map
$\mathscr M \to \mathscr M_{g,n}$.  We denote the associated family of punctured curves by $\mathscr
C^\circ \to \mathscr M$. We are given a local system $\mathbb V$ on $\mathscr C^\circ$.
\autoref{theorem:rank-bound-subsystem} says that if $\mathbb{V}$ is unitary,
there are no sub-local systems of $R^1 \pi^\circ_* \mathbb V$ of low rank (that is, rank less
than $2g - 2\rk \mathbb V$).
Let $C^\circ$ be a fiber of $\pi^\circ$.
We deduce \autoref{theorem:unitary-rigid}, which says that when $\mathbb
V|_{C^\circ}$
is unitary and has low rank (less than $g$), then $R^1 \pi^\circ_* \mathbb V$ has no global sections.
The benefit of \autoref{theorem:unitary-rigid} is that we only need check
unitarity on fibers, but the cost is that we are only able to rule out global
sections, i.e. trivial sub-local systems, instead of arbitrary low rank
sub-local systems.

First, we deduce \autoref{theorem:unitary-rigid} from
\autoref{theorem:rank-bound-subsystem}.
The idea is to use the assumption that
$\mathbb V$ is unitary on fibers to
reduce via \autoref{lemma:unitary-direct-sum} to the case that $\mathbb V$ is a
tensor product $\mathbb U \otimes (\pi^\circ)^* \mathbb W$, with $\mathbb{U}$ unitary. 
A nonzero global section of $R^1\pi^\circ_*(\mathbb U \otimes (\pi^\circ)^* \mathbb W)=(R^1\pi^\circ_*\mathbb{U})\otimes \mathbb{W}$ yields a nonzero map $\mathbb W^\vee \to R^1 \pi^\circ_* \mathbb U$, and hence a
low rank sub-local system of $R^1 \pi^\circ_* \mathbb U$, contradicting 
\autoref{theorem:rank-bound-subsystem}.

The proof of
\autoref{theorem:rank-bound-subsystem} boils down to an analysis of the derivative of the period map associated to the complex variation of mixed Hodge structure on $R^1\pi^\circ_*\mathbb{V}$, for $\mathbb{V}$ unitary.
We identify this derivative with a natural multiplication map $$H^0(E\otimes \omega_C(D))\otimes H^0(E^\vee\otimes \omega_C)\to H^0(\omega_C^{\otimes 2}(D))$$ in
\autoref{theorem:period-map-computation}, where $C$ is the smooth projective compactification of $C^\circ$ with boundary $D$, and $E$ is the vector bundle on $C$ associated to $\mathbb{V}|_{C^\circ}$ by the Mehta-Seshadri correspondence \cite{mehta1980moduli}. By the theorem of the fixed part, the existence of a low rank sub-local system of $R^1\pi_*^\circ\mathbb{V}$ places restrictions on this map, which we rule out
 using vector bundle methods,
as developed in \cite[\S6]{LL:geometric-local-systems},
ultimately relying on Clifford's theorem for vector bundles.


\subsection{Notation}
\label{subsection:notation}
Unless otherwise stated, we will work over the field of complex numbers $\mathbb{C}$. 
In particular, we will lift this restriction in
\autoref{section:arithmetic-applications}.

\begin{notation}
	\label{notation:versal-family}
	We fix non-negative integers $(g,n)$ so that $n \geq 1$ if $g = 1$ and
	$n \geq 3$ if $g = 0$, i.e., $\Sigma_{g,n}$ is hyperbolic.
	Let $\mathscr{M}$ be a connected complex variety. A \emph{family
	of $n$-pointed curves of genus $g$ over $\mathscr{M}$} is a smooth
	proper morphism $\pi: \mathscr{C}\to \mathscr{M}$ of relative dimension one, with geometrically
	connected genus $g$ fibers, equipped with $n$ sections $s_1,\cdots,
	s_n: \mathscr{M}\to \mathscr{C}$ with disjoint images. Call such a family is \emph{versal} if the
	induced map $\mathscr{M}\to \mathscr{M}_{g,n}$ is dominant and \'etale.
	Here $\mathscr M_{g,n}$ denotes the Deligne-Mumford moduli stack of $n$-pointed genus
	$g$ smooth proper curves with geometrically connected fibers.

	If $\pi: \mathscr{C}\to \mathscr{M}$ is a family of $n$-pointed curves, we let
	$\mathscr D := \coprod_{i=1}^n \on{im}(s_i)$ denote the images of
	the sections, which is finite \'etale of degree $n$ over $\mathscr M$.
	Also let
	$\mathscr{C}^\circ:=\mathscr{C}\setminus \bigcup_i \on{im}(s_i)$,
	let $j: \mathscr C^\circ \hookrightarrow \mathscr{C}$ be the natural
	inclusion, and let $\pi^\circ := \pi \circ j : \mathscr C^\circ \to \mathscr	M$ denote the composition. We will refer to $\pi^\circ: \mathscr{C}^\circ\to \mathscr{M}$ as the \emph{associated family of punctured curves}. If $\pi^\circ$ arises as the family of punctured curves associated to a versal family of $n$-pointed curves, we will call it a \emph{punctured versal family}.
	We will frequently use $m \in \mathscr M$ as a basepoint, and $c \in
	\mathscr C^\circ$ as a basepoint with $\pi^\circ(c) = m$.
	We use $C^\circ$ as notation to denote the fiber $(\pi^\circ)^{-1}(m) = \mathscr
	C^\circ_m$.
	\begin{equation}
		\label{equation:}
		\begin{tikzcd}
			\mathscr C^\circ \ar {r}{j} \ar {rd}{\pi^\circ} &
			\mathscr C \ar {d}{\pi} & \mathscr D \ar{l} \\
			& \mathscr M \ar{ur}{s_1} \ar[bend
			right=70,swap,ur,"s_n"]
			\arrow[bend right = 60, ur, draw=none, "\ddots"]
		 &
		\end{tikzcd}
	\end{equation}
\end{notation}


\begin{notation}
	\label{notation:adjoint}
	For $G$ an algebraic group with derived subgroup $G^{\on{der}}$ and
corresponding Lie algebra $\mathfrak g^{\on{der}}$, we use
$\on{Ad} : G \to \on{GL}(\mathfrak g^{\on{der}})$
to denote the natural action of $G$ on $\mathfrak{g}^{\on{der}}$ by conjugation.
Given a representation $\rho: \pi_1(X, x) \to G(\mathbb C)$, let $\on{ad}(\rho)
:= \rho \circ \on{Ad}: \pi_1(X, x) \to \on{GL}(\mathfrak g^{\on{der}})$. 
In particular, given
$\rho : \pi_1(X, x) \to \on{GL}_r(\mathbb C)$ or 
$\rho : \pi_1(X, x) \to \on{PGL}_r(\mathbb C)$
we use $\on{ad}(\rho)$ to denote the composite map
$\on{ad}(\rho) : \pi_1(X, x) \to \on{GL}(\mathfrak {pgl}_r(\mathbb C))$.
Under the identification between local systems on a connected space and representations of the
fundamental group, if $\mathbb W$ is a local system on $X$ corresponding to some representation $\rho$, 
we use $\on{ad}\mathbb W$ to denote the local system corresponding to
$\on{ad}\rho$.

A representation $\pi_1(X,x)\to \on{GL}_r(\mathbb{C})$ is \emph{unitary} if its image has
compact closure. 
A complex local system is unitary if its monodromy representation is unitary. By
an averaging argument, unitary representations are exactly those that preserve a positive-definite Hermitian form on $\mathbb{C}^n$, i.e.~they are conjugate to a representation that factors through the unitary group $\on{U}(n)$.
\end{notation}

\begin{notation}
For a pointed finite-type scheme or Deligne-Mumford stack $(X, x)$ over
$\mathbb{C}$, we will use $\pi_1(X, x)$ to denote the topological fundamental group 
 of the associated complex-analytic space or analytic stack.
Similarly, for a local system $\Lambda$ on $X$, we use
$H^i(X, \Lambda)$ to denote the singular cohomology of the associated complex-analytic space or stack, unless otherwise stated. We will use $\on{Mod}_{g,n}$ to denote the mapping class group of an orientable surface of genus $g$ with $n$ punctures, and $\on{PMod}_{g,n}=\pi_1(\mathscr{M}_{g,n})$ 
(see \autoref{lemma:mgn-fundamental-group})
to denote pure mapping class group, i.e.~the subgroup of $\on{Mod}_{g,n}$ of index $n!$ preserving the punctures pointwise.
\end{notation}

\subsection{Acknowledgements}

We would like to thank H\'el\`ene Esnault
for extremely helpful comments, and especially for finding and helping us correct a substantial gap in the original proof of \autoref{theorem:finite-image-semisimple}. We thank Andrew Putman for encouraging us in this project, and Ian Agol and Dawid Kielak for bringing the Putman-Wieland conjecture to our attention.
We also thank an anonymous referee,
Philip Boalch,
Anand Deopurkar,
Charles Doran,
Benson Farb,
Matt Kerr,
Mark Kisin,
Christian Klevdal,
Joshua Lam,
Brian Lawrence,
Alexander Lubotzky,
Vladimir Markovi\'c,
Takuro Mochizuki,
Stefan Patrikis,
Alexander Petrov,
Mihnea Popa,
Will Sawin,
Carlos Simpson,
Andrew Snowden,
Junho Peter Whang,
Ben Wieland,
Melanie Matchett Wood, 
and Kang Zuo.
Landesman was supported by the National Science Foundation under Award No. DMS-2102955 
and Litt was supported by NSF grant DMS-2001196.
This material is based upon work supported by the Swedish Research Council under grant no. 2016-06596 while the authors were in residence at Institut Mittag-Leffler in Djursholm, Sweden during the fall of 2021.

\section{Representation-theoretic preliminaries}
\label{section:representation-preliminaries}

In this section, we give group-theoretic constructions which we will use to
analyze representations with finite mapping class group orbit.
In \autoref{subsection:basic-properties},
we verify basic properties of MCG-finite representations. 
In \autoref{subsection:mcg-finite-construction}, given an irreducible representation
$\rho: \pi_1(\Sigma_{g,n},x) \to \on{GL}_r(\mathbb C)$ whose conjugacy class is fixed by a finite-index subgroup $\Gamma\subset\on{PMod}_{g,n+1}$, we construct a representation
$\widetilde{\rho}:  \Gamma \to \on{PGL}_r(\mathbb C)$,
which will be analyzed throughout this paper.
In \autoref{subsection:lifting-projective},
we show how to lift certain projective representations of the fundamental groups of families of curves to honest representations into
$\on{GL}_r(\mathbb C)$,
after passing to a suitable cover. Finally, in
\autoref{subsection:unitary-on-fibers}, we prove some structural results about the representations we've constructed.

We will use the following lemma throughout, to connect properties of mapping class groups to geometry.
\begin{lemma}
	\label{lemma:mgn-fundamental-group}
	For $m \in \mathscr M_{g,n}$ a basepoint and $m'\in \mathscr M_{g, n+1}$ a lift of $m$, there are isomorphisms $\pi_1(\mathscr M_{g,n}, m) \simeq \on{PMod}_{g,n}$, $\pi_1(\mathscr M_{g,n+1}, m') \simeq \on{PMod}_{g,n}$ such that the diagram 
	$$\xymatrix{
	1 \ar[r] & \pi_1(\Sigma_{g,n}) \ar[r] \ar@{=}[d]& \pi_1(\mathscr{M}_{g,n+1}, m') \ar[r] \ar[d]^\sim & \pi_1(\mathscr{M}_{g,n}, m) \ar[r] \ar[d]^\sim & 1 \\
	1 \ar[r] & \pi_1(\Sigma_{g,n}) \ar[r] & \on{PMod}_{g,n+1} \ar[r] & \on{PMod}_{g,n}\ar[r] & 1
	}$$
	commutes, where the vertical maps are given by these isomorphisms, the top row is the exact sequence of homotopy groups induced by the forgetful map $\mathscr{M}_{g,n+1}\to \mathscr{M}_{g,n}$,  and the bottom row is the Birman exact sequence \cite[p.~98]{farbM:a-primer}. 
\end{lemma}
\begin{proof}
This follows from the contractibility of the universal cover of $\mathscr{M}_{g,
n}$, and the fact that $\mathscr{M}_{g,n}$ is the quotient of its universal
cover by the properly discontinuous action of $\on{PMod}_{g,n}$. See
\cite[\S10.6.3 and p. 353]{farbM:a-primer}. There is a choice involved here, namely: if $C$ is the Riemann surface associated to the point $m\in \mathscr{M}_{g,n}$, one must choose a homeomorphism between $C$ and our reference surface $\Sigma_{g,n}$, and similarly with $m'$. Changing this homeomorphism replaces the vertical isomorphisms by conjugate isomorphisms.
\end{proof}
The subgroup $ \pi_1(\Sigma_{g,n})\subset \on{PMod}_{g,n+1} $ is often referred to as the \emph{point-pushing subgroup}.

\subsection{Basic properties of MCG-finite representations}\label{subsection:basic-properties}

To acquaint the reader with MCG-finite representation, we now spell out some of
their basic properties.

\begin{proposition}\label{proposition:basic-properties}
	The direct sum and tensor product of two MCG-finite representations of $\pi_1(\Sigma_{g,n})$ is MCG-finite. Any semisimple subquotient of an MCG-finite representation is MCG-finite.
\end{proposition}
\begin{proof}
The first statement, about direct sums and tensor products, is clear. The second is immediate from \cite[Lemma 2.2.1]{lawrence2019representations}.
\end{proof}

We next show that restrictions of representations of $\on{Mod}_{g,n+1}$ to the
point-pushing subgroup are MCG-finite. See also \cite[Lemma
2.2]{kasahara2015visualization} for a related result with analogous proof. 

\begin{proposition}\label{proposition:MCG-finite-index-rep}
	Let $\Gamma\subset \on{Mod}_{g,n+1}$ be a finite index subgroup
	containing the point-pushing subgroup $\pi_1(\Sigma_{g,n})\subset
	\on{Mod}_{g,n+1}$. Let $$\rho:\Gamma\to \on{GL}_r(\mathbb{C})$$ be a representation. Then $\rho|_{\pi_1(\Sigma_{g,n})}$ is MCG-finite.
\end{proposition}
\begin{proof}
	By replacing $\Gamma$ with 
	$\Gamma \cap \on{PMod}_{g,n+1},$
	we may assume $\Gamma \subset \on{PMod}_{g,n+1}$,
	the pure mapping class group.  We claim that the image of $\Gamma$ in
	$\on{PMod}_{g,n}$ (under the map $\on{PMod}_{g,n+1}\to \on{PMod}_{g,n}$
	arising from the Birman exact sequence) stabilizes the conjugacy class
	of $\rho$.
	It suffices to show that for each $\gamma\in \Gamma$, 
\begin{align*}
	\rho^\gamma:  \pi_1(\Sigma_{g,n})& \rightarrow \on{GL}_r(\mathbb C) \\
	g & \mapsto \rho(\gamma g \gamma^{-1})
\end{align*}
is conjugate to $\rho$. Indeed, $\rho^\gamma(h) = \rho(\gamma) \rho(h)
\rho(\gamma)^{-1}$.
\end{proof}

We later give a partial converse to \autoref{proposition:MCG-finite-index-rep},
for irreducible representations, in \autoref{corollary:converse-to-MCG-rep}.
We now give a geometric counterpart to
\autoref{proposition:MCG-finite-index-rep},
which is closely related to \cite[Theorem A1]{CH:isomonodromic}.
\begin{proposition}
\label{proposition:MCG-finite-family-of-curves}
Let $\pi^\circ: \mathscr{C}^\circ\to \mathscr{M}$ be a punctured versal family
of $n$-pointed genus $g$ curves, as in \autoref{notation:versal-family}. In
particular, $\mathscr M \to \mathscr M_{g,n}$ is dominant and \'etale.
	Let $C^\circ$ be a fiber of $\pi^\circ$. If
	$$\rho:\pi_1(\mathscr{C}^\circ)\to \on{GL}_r(\mathbb{C})$$ is a representation, $\rho|_{\pi_1(C^\circ)}$ is MCG-finite.
\end{proposition}
We prove \autoref{proposition:MCG-finite-family-of-curves}, in
\autoref{subsubsection:proof-family-of-curves},
but we first require some well-known lemmas about families of curves. 
\begin{lemma}
	\label{lemma:finite-index-from-dominant-map}
	For $\mathscr M$ a scheme and any dominant \'etale map
	$\mathscr M \to \mathscr M_{g,n}$,
	$\on{im}(\pi_1(\mathscr M) \to \pi_1(\mathscr M_{g,n}))$
	has finite index in 
$\pi_1(\mathscr M_{g,n})\simeq \on{PMod}_{g,n}$.
\end{lemma}
\begin{proof}
	After pulling back to a finite \'etale cover $\mathscr
	N \to \mathscr M_{g,n}$ which is a scheme (such exists by e.g.~\cite[Proposition 2.3.4]{pikaart-dejong}),
	we may assume there is a dominant \'etale map of schemes
	$\alpha: \mathscr M \to \mathscr N$. It is enough to show $\on{im}(\pi_1(\mathscr M) \to \pi_1(\mathscr N))$
	has finite index in $\pi_1(\mathscr N)$, which follows from \cite[Lemma
	4.19]{debarre:higher-dimensional}.
	The basic idea here is to pass to opens over which the map is finite
	flat, and use that passing to opens will preserve the property that the
	image of the map of fundamental groups has finite index.
%
\end{proof}

\begin{lemma}\label{lemma:SES-pi1-versal-curves}
If $\mathscr{C} \to \mathscr{M}$ is a versal family of $n$-pointed curves of
genus $g$, with $3g - 3 + n \geq 0$, as in \autoref{notation:versal-family},
so that $C^\circ$ is a fiber of $\pi^\circ$,
the sequence
\begin{equation}
	\label{equation:fibration-sequence}
	\begin{tikzcd}
		1 \ar {r} & \pi_1({C}^\circ)  \ar {r} &
		\pi_1(\mathscr{C}^\circ) \ar {r} & \pi_1(\mathscr{M}) \ar {r} &
		1
\end{tikzcd}\end{equation}
is exact.
\end{lemma}
\begin{proof}
Except for injectivity, the result is immediate from the long exact sequence in
homotopy groups and the fact that ${C}^\circ$ is connected. Injectivity holds in
the universal case where $\mathscr{M}=\mathscr{M}_{g,n}$ is the Deligne-Mumford
moduli stack of curves and $\mathscr{C}^\circ = \mathscr{C}^\circ_{g,n}$ is the
universal punctured curve, as the universal cover of $\mathscr{M}_{g,n}$ is contractible. 
We next show the map $\pi_1(C^\circ) \to
\pi_1(\mathscr C^\circ)$ is injective in the general case via pullback.
The composite map $C^\circ \to \mathscr C \to \mathscr{C}^\circ_{g,n}$
is the natural inclusion of $C^\circ$ as a fiber of $\mathscr{C}^\circ_{g,n}$
over $\mathscr M_{g,n}$.
Hence, the composition $\pi_1(C^\circ) \to \pi_1(\mathscr C^\circ) \to \pi_1(\mathscr
C^\circ_{g,n})$ is injective, and so the first map is injective.
\end{proof}

\subsubsection{}
\label{subsubsection:proof-family-of-curves}
\begin{proof}[Proof of {\autoref{proposition:MCG-finite-family-of-curves}}]
	The exact sequence \eqref{equation:fibration-sequence}
induces an outer action of $\pi_1(\mathscr M)$ on
$\pi_1(C^\circ)$, which factors through the outer action of $\on{PMod}_{g,n}$ on
$\pi_1(C^\circ)$ induced by the Birman exact sequence.
Specifically, this outer action sends $\gamma \in \on{PMod}_{g,n}$
to the action of conjugation by $\widetilde{\gamma}$ on $\pi_1(C^\circ)$, for
$\widetilde{\gamma}\in \on{PMod}_{g,n+1}$ 
a lift of $\gamma$.
By construction, this outer action is compatible with the outer action of $\on{Mod}_{g,n}$ on
$\pi_1(C^\circ)$ of \autoref{definition:mcg-finite}.

To show $\rho|_{\pi_1(C^\circ)}$ is MCG-finite,
it therefore suffices to show the conjugacy class of $\rho|_{\pi_1(C^\circ)}$ is stable under
the image $\Gamma$ of $$\pi_1(\mathscr{C}^\circ)\twoheadrightarrow \pi_1(\mathscr{M})\to \pi_1(\mathscr{M}_{g,n})\simeq \on{PMod}_{g,n}.$$
This image has finite index by \autoref{lemma:finite-index-from-dominant-map}. 
The conjugacy class of $\rho|_{\pi_1(C^\circ)}$ is stable under this image because, 
if $\widetilde{\gamma} \in \pi_1(\mathscr{C}^\circ)$ is any lift of  
$\gamma \in \Gamma$, then $\rho^{\widetilde{\gamma}}$ is conjugate to $\rho$ by $\rho(\widetilde{\gamma})$, as in the proof of \autoref{proposition:MCG-finite-index-rep}.
\end{proof}
\subsection{Projective local systems associated to MCG-finite representations}
\label{subsection:mcg-finite-construction}

We next give a construction lifting an irreducible MCG-finite representation
of $\pi_1(\Sigma_{g,n})$ to a projective representation of a finite index subgroup of $\Mod_{g,n+1}$.

\begin{notation}
	\label{notation:finite-orbit}
	Suppose $\rho$ is an irreducible MCG-finite representation of $\pi_1(\Sigma_{g,n},x)$.
	Let $\Gamma \subset \on{PMod}_{g,n}$ denote the finite index stabilizer
	of the conjugacy class of $\rho$ and let $\widetilde{\Gamma} \subset
	\on{PMod}_{g,n+1}$ denote the preimage of $\Gamma$ under the surjection
	$\on{PMod}_{g,n+1} \to \on{PMod}_{g,n}$ coming from the Birman exact
	sequence.
\end{notation}
	
\begin{lemma}
	\label{lemma:mcg-rep-construction}
	In the setup of \autoref{notation:finite-orbit}, there exists a unique representation $\widetilde{\rho}:
	\widetilde{\Gamma} \to \on{PGL}_r(\mathbb C)$ so that
	\begin{equation}
		\label{equation:mcg-rho-and-tilde-rho-compatibility}
		\begin{tikzcd} 
			\pi_1(\Sigma_{g,n},x)  \ar[r, "\rho"] \ar {d} &
			 \on{GL}_r(\mathbb C) \ar {d} \\
			 \widetilde{\Gamma}  \ar [r, "\widetilde \rho"] & \on{PGL}_r(\mathbb C)
	\end{tikzcd}\end{equation}
	commutes,
	where $\pi_1(\Sigma_{g,n}, x) \subset\widetilde{\Gamma}\subset \on{PMod}_{g,n+1}$ is the inclusion
	of the point-pushing subgroup from the Birman exact sequence.
\end{lemma}
\begin{proof}
	First, we construct the representation $\widetilde{\rho}$. 
	By identifying 
	$\on{PMod}_{g,n+1} \subset \on{Aut}(\pi_1(\Sigma_{g,n},x))$, 
	we obtain an action of 
	$\on{PMod}_{g,n+1}$ on $r$-dimensional representations of 
	$\pi_1(\Sigma_{g,n},x)$.
	Since $\widetilde{\Gamma}$ preserves $\rho$ up to
	conjugacy, for every $\gamma \in \widetilde{\Gamma}$, there is some matrix $M_\gamma \in \on{GL}_r(\mathbb C)$
	so that $M_\gamma \rho M_\gamma^{-1} = \gamma(\rho)$.
	Since $\rho$ is irreducible, $M_\gamma$ is unique up to scaling by
	Schur's lemma.
We let $\overline{M}_\gamma$ denote the image of $M_\gamma$ in
	$\on{PGL}_r(\mathbb C)$; $\overline{M}_\gamma$ is well-defined by the previous sentence.
	Define $\widetilde{\rho}$ by  
	\begin{align*}
		\widetilde{\rho}:  \widetilde{\Gamma} &\rightarrow \on{PGL}_r(\mathbb C) \\
		\gamma &\mapsto \overline{M}_\gamma.
	\end{align*}

	Uniqueness of 
	$\overline{M}_\gamma \in \on{PGL}_r(\mathbb C)$
	implies that $\widetilde{\rho}$ is a representation, as $\overline{M}_\gamma\overline{M}_{\gamma'}=\overline{M}_{\gamma\gamma'}$. 

	Commutativity of 	
	\eqref{equation:mcg-rho-and-tilde-rho-compatibility}
	follows from the definition of the inclusion $\iota: \pi_1(\Sigma_{g,n}, x) \hookrightarrow
	\on{PMod}_{g,n+1}\subset \on{Aut}(\pi_1(\Sigma_{g,n}, x))$ as the
	point-pushing subgroup, realizing the set of inner automorphisms.
  Indeed, for $\eta\in \pi_1(\Sigma_{g,n}, x)$, $\iota$ sends
	$\eta$ to the automorphism of $\pi_1(\Sigma_{g,n}, x)$ given by $\beta\mapsto \eta \beta
	\eta^{-1}.$
	Hence we may take $M_{\iota(\eta)}=\rho(\eta)$, implying 
	\eqref{equation:mcg-rho-and-tilde-rho-compatibility} commutes.
	
	Finally, uniqueness of $\widetilde{\rho}$ follows from commutativity of \eqref{equation:mcg-rho-and-tilde-rho-compatibility} and uniqueness of $\overline{M}_\gamma$. Indeed, for any $\gamma\in \widetilde{\Gamma}$ and for all $\eta\in \pi_1(\Sigma_{g,n}, x)$, we must have $$\widetilde{\rho}(\gamma\eta\gamma^{-1})=\widetilde{\rho}(\gamma)\widetilde{\rho}(\eta)\widetilde{\rho}(\gamma)^{-1},$$ and hence $\widetilde{\rho}(\gamma)$ must equal $\overline{M}_\gamma$.
\end{proof}
See e.g.~the proof of \cite[Theorem 4]{simpson1992higgs} for a similar argument.

We will also need a geometric variant of \autoref{lemma:mcg-rep-construction}.
We give the geometric rephrasing of \autoref{lemma:mcg-rep-construction},
implicitly using \autoref{lemma:mgn-fundamental-group}. Below $\mathscr{M}$ is
the finite \'etale cover of $\mathscr{M}_{g,n}$ corresponding to a finite-index subgroup of $\Gamma\subset \on{PMod}_{g,n}$, chosen so that $\mathscr{M}$ is a scheme (as opposed to a Deligne-Mumford stack).

\begin{lemma}
	\label{lemma:geometric-mcg-rep-construction}
	Let
	$\rho: \pi_1(\Sigma_{g,n},x) \to \on{GL}_r(\mathbb C)$ be an irreducible MCG-finite representation. 
	There exists a scheme $\mathscr M$ with a finite \'etale map $\mathscr M \to \mathscr
	M_{g,n}$, with associated family of $n$-times punctured curves $\pi^\circ: \mathscr{C}^\circ\to \mathscr{M}$, a point $c\in \mathscr{C}^\circ$ with $m:=\pi^\circ(c)$,
	and a representation $\widetilde{\rho}: \pi_1(\mathscr C^\circ, c) \to
	\on{PGL}_r(\mathbb C)$ so that the following holds. Upon identifying
	$\pi_1(\Sigma_{g,n},x)  \simeq \pi_1(\mathscr C^\circ_m,c)$, the diagram
	\begin{equation}
		\label{equation:geometric-rho-and-tilde-rho}
		\begin{tikzcd} 
			\pi_1(\mathscr C^\circ_m,c)  \ar[r, "\rho"] \ar {d} &
			 \on{GL}_r(\mathbb C) \ar {d} \\
			 \pi_1(\mathscr C^\circ, c) \ar[r, "\widetilde\rho"]& \on{PGL}_r(\mathbb C)
	\end{tikzcd}\end{equation}
	commutes.
\end{lemma}

\subsection{Lifting Projective representations}
\label{subsection:lifting-projective}

The main goal of this section is to prove \autoref{proposition:gl-lift},
which says that, \'etale locally on the base $\mathscr{M}$, we can lift the representation $\widetilde{\rho}$ into
$\on{PGL}_r(\mathbb C)$
constructed in \autoref{lemma:geometric-mcg-rep-construction}
to a representation $\rho'$ into $\on{GL}_r(\mathbb C)$.

Because the obstruction to lifting is related to $\mathbb C^\times =
\ker(\on{GL}_r(\mathbb C) \to
\on{PGL}_r(\mathbb C))$, we will use the following classification when
$r = 1$.

\begin{proposition}[\protect{~\cite[Lemma 3.2]{biswas2017surface}}]
	\label{proposition:1-dim-finite-image}
	Let $g\geq 1, n\geq 0$. Then any MCG-finite representation $$\rho: \pi_1(\Sigma_{g,n})\to \mathbb{C}^\times$$ has finite image.
\end{proposition}
For the reader's benefit, we recall the idea of the proof of
\autoref{proposition:1-dim-finite-image}.
The idea is to show
that if any standard generator of $\pi_1(\Sigma_{g,n})$ corresponding to a loop
$\gamma$ in $\Sigma_{g,n}$ has infinite order under $\rho$, then acting on $\rho$ by powers of the
Dehn twist about $\gamma$ gives an infinite collection of distinct
representations.

In order to lift 
$\on{PGL}_r(\mathbb C)$-reps to
$\on{GL}_r(\mathbb C)$-reps, 
we will need to know 
that the image of $\rho$ has finite
intersection with the center of $\on{GL}_r(\mathbb C)$.
\begin{lemma}
	\label{lemma:finite-central-part}
	Let $g\geq 1$.
	Suppose $\rho: \pi_1(\Sigma_{g,n}) \to \on{GL}_r(\mathbb C)$ is an MCG-finite representation.
	Then $\rho$ has finite determinant.
	Moreover, $\rho$ factors through a subgroup $G$ with $\on{SL}_r(\mathbb C) \subset G
	\subset \on{GL}_r(\mathbb C)$ and $\mu := G \cap \ker(\on{GL}_r(\mathbb
	C) \to \on{PGL}_r(\mathbb C))$ a finite group.
\end{lemma}
\begin{proof}
	By \autoref{proposition:1-dim-finite-image}, the composition
	$\pi_1(\Sigma_{g,n}) \xrightarrow{\rho} \on{GL}_r(\mathbb C) \xrightarrow{\on{det}}
	\mathbb C^\times$ has finite image $H$.
	We can take $G := \det^{-1}(H) \subset \on{GL}_r(\mathbb{C})$,
	which indeed has finite intersection with the center of
	$\on{GL}_r(\mathbb C)$, and hence finite intersection with
	$\ker(\on{GL}_r(\mathbb C) \to \on{PGL}_r(\mathbb C))$
\end{proof}

In order to lift representations from
$\on{PGL}_r(\mathbb C)$ to
$\on{GL}_r(\mathbb C)$,
we will need to kill certain cohomological obstructions using dominant \'etale maps.
\begin{lemma}
	\label{lemma:vanish-on-etale}
	Let $i > 0$ and $\mathscr{M}$ be a complex variety. Suppose $\mu$ is a finite abelian group and
	$\sigma \in H^i(\pi_1(\mathscr M, m), \mu)$
	a cohomology class.
	Then there is a dominant \'etale map $\mathscr M' \to \mathscr M$
	so that $\sigma|_{\mathscr M'} =0 \in H^i(\pi_1(\mathscr M'), \mu)$.
\end{lemma}
\begin{proof}
%

First, we claim that for any class
$\alpha \in H^i(\mathscr M, \mu)$,
there is a dominant \'etale map $\mathscr N \to \mathscr M$ so that
$\alpha|_{\mathscr N} = 0$.
Since $\mu$ is finite, we can use the comparison between singular
and \'etale cohomology to reduce to showing that for any
$\alpha \in H^i_{\mathrm{\acute et}}(\mathscr M, \mu)$ there is a dominant \'etale map
$\mathscr N \to \mathscr M$ for which $\alpha|_{\mathscr N}= 0$.
Choose an injective resolution $I^\bullet$ of $\mu$ and represent
$\alpha$ by some $\bar\alpha \in \ker(I^i(\mathscr M) \to I^{i+1}(\mathscr M))$.
Because $I^\bullet$ is exact, there exists for each $x\in \mathscr{M}$ an \'etale neighborhood $U_x$ of $x$  
so that $\alpha|_{U_x}$ is in the image of $I^{i-1}(U_x)\to I^{i}(U_x)$. Taking $x$ to be the generic point of $\mathscr{M}$ gives the claim.

Now, let $\alpha$ as above denote the image of $\sigma$ under the natural map
$H^i(\pi_1(\mathscr M, m), \mu) \to H^i(\mathscr M, \mu),$
and take $\mathscr N$ as above so that $\alpha|_\mathscr N = 0$.
By
\cite[Expos\'e XI, 4.6]{SGA4}
we can find a Zariski open $\mathscr M'
\subset \mathscr N$ which is a 
$K(\pi_1(\mathscr M', m'), 1)$ Eilenberg-Maclane space.
This implies the natural map $H^i(\pi_1(\mathscr M', m'), \mu) \to H^i(\mathscr M', \mu)$
is an isomorphism.
Therefore, since $\alpha|_{\mathscr N} = 0$, we also have $\alpha|_{\mathscr M'}
= 0$, and hence $\sigma|_{\pi_1(\mathscr M')} = 0$, as desired.
\end{proof}

We now combine the above results to show that the projective representation $\widetilde{\rho}$ constructed in \autoref{lemma:geometric-mcg-rep-construction}
can be lifted to a linear representation.

\begin{proposition}
	\label{proposition:gl-lift}
	Suppose $\mathscr C \to \mathscr M$ is a versal family of $n$-pointed
	curves of genus $g\geq 1$, with associated family of punctured curves 
	$\pi^\circ: \mathscr{C}^\circ\to \mathscr{M}$. Let $c\in \mathscr{C}^\circ$ be a point and $m=\pi^\circ(c)$.
	Suppose we have representations $\rho: \pi_1(\mathscr C_m^\circ, c) \to
	\on{GL}_r(\mathbb C)$
	and $\widetilde{\rho}: \pi_1(\mathscr C^\circ, c) \to \on{PGL}_r(\mathbb
	C)$ so that the diagram
	\begin{equation}
		\label{equation:rho-and-tilde-rho-compatibility}
		\begin{tikzcd} 
			 \pi_1(\mathscr C_m^\circ, c) \ar[r, "\rho"] \ar {d} &
			 \on{GL}_r(\mathbb C) \ar {d} \\
			 \pi_1(\mathscr C^\circ, c) \ar[r, "\widetilde\rho"] & \on{PGL}_r(\mathbb C)
	\end{tikzcd}\end{equation}
	commutes.
	Moreover, assume $\rho$ is MCG-finite.
	Then, there exists 
	\begin{enumerate}
		\item a dominant \'etale map $\mathscr M' \to \mathscr M$,
		\item corresponding relative curve $\mathscr C' := \mathscr M'
	\times_{\mathscr M} \mathscr C$ with associated family of punctured curves ${\mathscr{C}'}^\circ$,
	and 
\item a representation $\rho' : \pi_1(\mathscr C'^\circ, c') \to
	\on{GL}_r(\mathbb C)$ for some basepoint $c' \in \mathscr C'^\circ$ lying over $c$
	\end{enumerate}
so that, 
	upon choosing some $m'$ over $m$ with $\mathscr C_{m'} \simeq \mathscr C_m$,
	\begin{equation}
		\label{equation:cover-rep-compatibility}
		\begin{tikzcd} 
			\pi_1(\mathscr C_{m'}^\circ, c') \ar{r} \ar[bend left]{rr}{\rho} &
			\pi_1(\mathscr C'^\circ, c') \ar {r}{\rho'} \ar {d} &
			 \on{GL}_r(\mathbb C) \ar {d} \\
			 & \pi_1(\mathscr C^\circ, c) \ar
			 {r}{\widetilde{\rho}} & \on{PGL}_r(\mathbb C)
	\end{tikzcd}\end{equation}
	commutes.	
	Moreover, $\on{im} \rho'$ is contained a subgroup $G$ with $\on{SL}_r(\mathbb
	C) \subset G \subset
	\on{GL}_r(\mathbb C)$, such that $\on{SL}_r(\mathbb C) \subset G$ has
	finite index.
\end{proposition}
\begin{proof}
	Taking $G$ and $\mu$ as in \autoref{lemma:finite-central-part},
	there is a commuting diagram
\begin{equation}
	\label{equation:pgl-and-birman}
	\begin{tikzcd} 
H^1(\pi_1(\mathscr C^\circ,c), \mu) \ar
		{r}{\lambda} \ar
		{d}{\nu} &  H^1(\pi_1(\mathscr
				C_m^\circ,c),
		\mu) \ar {d}{\xi} \\
		\Hom(\pi_1(\mathscr C^\circ,c), G) \ar
		{r}{\alpha} \ar
		{d}{\beta} & \Hom(\pi_1(\mathscr
				C_m^\circ,c),
			G) \ar {d}{\gamma} \\
		\Hom(\pi_1(\mathscr C^\circ,c), \on{PGL}_r(\mathbb C)) \ar
		{r}{\delta}\ar{d}{\varepsilon}
		&\Hom(\pi_1(\mathscr C_m^\circ,c),
		\on{PGL}_r(\mathbb C))\ar{d}{\zeta} \\
			H^2(\pi_1(\mathscr C^\circ,c), \mu) \ar {r}{\eta}
		& H^2(\pi_1(\mathscr C_m^\circ,c),
		\mu)
\end{tikzcd}\end{equation}
where the columns are exact sequences of sets.
We start with a representation $\widetilde{\rho} \in \on{Hom}(\pi_1(\mathscr C^\circ,c),
\on{PGL}_r(\mathbb C))$ 
and $\rho \in \on{Hom}(\pi_1(\mathscr C_m^\circ, c), G)$.
The commutativity of \eqref{equation:rho-and-tilde-rho-compatibility} can be equivalently
expressed as the statement that $\delta(\widetilde{\rho}) = \gamma(\rho)$. 
We seek to construct $\rho'$ such that \eqref{equation:cover-rep-compatibility} commutes,
meaning that, after replacing $\mathscr M$ with a dominant \'etale $\mathscr M'$ and $\mathscr{C}$ with $\mathscr{C}'$,
there is some $\rho'$ with $\beta(\rho') = \widetilde{\rho}$ and $\alpha(\rho')
= \rho$.

First, we argue we can pass to some $\mathscr M'$ so that $\widetilde{\rho}$
lies in the image of $\beta$. 
It is equivalent to arrange that $\varepsilon(\widetilde{\rho}) = 0$.
As a first step, we claim $\eta(\varepsilon(\widetilde{\rho})) = 0$.
This holds because $\eta(\varepsilon(\widetilde{\rho})) =
\zeta(\delta(\widetilde{\rho}))$, and
$\delta(\widetilde{\rho}) = \gamma(\rho)$ lies in the image of $\gamma$.
We therefore have that $\varepsilon(\widetilde{\rho}) \in \ker \eta$. 
The Hochschild-Serre spectral sequence associated to the short exact sequence  $$1\to \pi_1(\mathscr{C}_m^\circ)\to \pi_1(\mathscr{C}^\circ)\to \pi_1(\mathscr{M})\to 1$$ of
\autoref{lemma:SES-pi1-versal-curves} yields a short exact sequence $$0\to H^{2,0}\to \ker \eta \overset{\iota}{\to} H^{1,1}\to 0,$$ where $H^{i,j}$ is a subquotient of $H^i(\pi_1(\mathscr M, m), H^j(\pi_1(\mathscr C^\circ_m,c),
\mu))$.

By passing to a finite \'etale cover $\mathscr{M}_1$ of $\mathscr{M}$, we can assume $\pi_1(\mathscr{M}, m)$ acts trivially on the finite group $H^1(\pi_1(\mathscr{C}_m^\circ, c), \mu)$.  Now as $H^i(\pi_1(\mathscr M, m), H^j(\pi_1(\mathscr C^\circ_m,c),
\mu))$ is finite, we may by \autoref{lemma:vanish-on-etale} pass to a further
dominant \'etale scheme over $\mathscr{M}$ killing every element of $H^{i,j}$,
$(i,j)=(2,0) \text{ or } (1,1)$, and hence every element of $\ker \eta$. Thus
after passing to any component $\mathscr{M}_2$ of this dominant \'etale scheme
over $\mathscr M_2$, we have $\epsilon(\widetilde\rho)=0$. Let $\mathscr{C}_2^\circ$ be the pullback of $\mathscr{C}^\circ$ to $\mathscr{M}_2$. 

Let $\overline{\rho} \in H^1(\pi_1(\mathscr C_2^\circ, c), G)$
be an element with $\beta(\overline{\rho}) = \widetilde{\rho}$; such a $\overline{\rho}$ exists as we have arranged $\epsilon(\widetilde{\rho})=0$. 
We wish to arrange $\alpha(\overline{\rho}) =\rho$. 
This may not be the case, but it is enough to show that after passing to a further dominant
\'etale $\mathscr M'$,
that we can modify $\overline{\rho}$ by an element of $\on{im}\nu$ so that this
does hold.
More precisely, note that $\alpha(\overline{\rho})$ differs from $\rho$ by an
element of $\sigma\in H^1(\pi_1(\mathscr{C}^\circ_m, c), \mu)$, as by
construction $\gamma(\alpha(\overline{\rho}))=\gamma(\rho)$. It suffices to show
that, after passing to a dominant \'etale scheme over $\mathscr{M}_2$, $\sigma$ is in the image of $\lambda$. After passing to a finite \'etale cover of $\mathscr{M}_2$, we may assume that $H^1(\pi_1(\mathscr{C}^\circ_m, c), \mu)$ is fixed by $\pi_1(\mathscr{M}, m)$.
The Hochschild-Serre spectral sequence gives a map
\begin{align*}
	H^0(\pi_1(\mathscr M, m), H^1(\pi_1({\mathscr C}_{m}^\circ,c), \mu))
\xrightarrow{\chi}
H^2(\pi_1(\mathscr M, m), H^0(\pi_1({\mathscr C}^\circ_{m}, c), \mu))
\end{align*}
such that $\chi(\sigma) = 0$ if an only if $\sigma \in \on{im}\lambda$.
Since $H^0(\pi_1({\mathscr C}^\circ_{m}, c), \mu)=\mu$
is finite, we can apply
\autoref{lemma:vanish-on-etale} so as to assume, after replacing $\mathscr M$ by
a dominant \'etale $\mathscr M'$
that $\chi(\sigma) = 0$.
This implies that $\sigma \in \on{im}\lambda$, as desired.
\end{proof}

We can now give a partial converse to \autoref{proposition:MCG-finite-family-of-curves}.
\begin{corollary}\label{corollary:converse-to-MCG-rep}
Let $\rho:\pi_1(\Sigma_{g,n})\to \on{GL}_r(\mathbb{C})$ be an irreducible MCG-finite
representation, with $g\geq 1$. 
There is a punctured versal family of curves
$\mathscr{C}^\circ\to \mathscr{M}$, and a representation $\rho'$ of
$\pi_1(\mathscr{C}^\circ)$ whose determinant has finite order, with the following property: For $C^\circ$ a fiber of
$\pi^\circ$, there is an identification $\pi_1(\Sigma_{g,n})\simeq
\pi_1(C^\circ)$ such that, under this identification, $\rho'|_{\pi_1(C^\circ)}=\rho$.
\end{corollary}
\begin{proof}
	This follows by combining \autoref{lemma:geometric-mcg-rep-construction} and \autoref{proposition:gl-lift}.
\end{proof}

\begin{remark}
	\label{remark:cousin-heu}
	Combining \autoref{corollary:converse-to-MCG-rep} with
	\autoref{proposition:MCG-finite-family-of-curves} gives a proof of
	\cite[Theorem A]{CH:isomonodromic}
	in the case that $\rho$ is semisimple and the flat vector bundle $(E, \nabla_0)$
	of \cite[Theorem A]{CH:isomonodromic} 
	is taken to be the Deligne canonical extension of
	$(\rho|_{C^\circ}\otimes \mathscr{O}, \on{id}\otimes d)$ to $C$.
\end{remark}

\subsection{Representations that are unitary on fibers}
\label{subsection:unitary-on-fibers}

The main goal of this section is to verify
\autoref{lemma:unitary-direct-sum}, which allows us to write local systems on a family of punctured curves $\mathscr{C}^\circ\to \mathscr{M}$ in terms of unitary local systems and local systems pulled back from $\mathscr{M}$, after passing to a dominant \'etale
map.

The following criterion for unitarity and finiteness will be useful when analyzing the representations produced by \autoref{lemma:mcg-rep-construction} and \autoref{proposition:gl-lift}.
\begin{lemma}
	\label{lemma:finite-rho-implies-finite-lift}
Let $G$ be a group and $H\trianglelefteq G$ a normal subgroup. Let $$\rho: G\to \on{GL}_r(\mathbb{C})$$ be a representation such that $\det\rho$ has finite order, and suppose $\rho|_H$ is irreducible.
\begin{enumerate}
\item If $\rho|_H$ has finite image, then $\rho$ has finite image.
\item if $\rho|_H$ is unitary, then $\rho$ is unitary.
\end{enumerate}
\end{lemma}
\begin{proof}
	We first prove $(1)$.
		As $\det\rho$ has finite order, it suffices to show that the
	projectivization $$\mathbb{P}\rho: G\overset{\rho}{\longrightarrow}
	\on{GL}_r(\mathbb{C})\to \on{PGL}_r(\mathbb{C})$$ has finite image. Let
$t=\#\on{im}(\rho|_H)$. For each $g\in G$, $\mathbb{P}\rho(g)$ is the unique
(by Schur's lemma) element of $\on{PGL}_r(\mathbb{C})$ such that
$$\mathbb{P}\rho(g)\mathbb{P}\rho(h)\mathbb{P}\rho(g)^{-1}=\mathbb{P}\rho(ghg^{-1})$$ for all $h\in H$.
Since $\rho(g)$ acts by conjugation on the order $t$ set $\on{im}(\rho|_H)$, its action 
has order dividing $t!$, so $\rho(h)=\rho(g^{t!})\rho(h)\rho(g^{-t!})$. Hence we have $\mathbb{P}\rho(g^{t!})=\on{id}$ by
uniqueness. Thus the image of $\mathbb{P}\rho$ has exponent dividing $t!$. But a linear group with finite exponent is finite \cite{burnside1905criteria}.

We now turn to the unitary case and verify $(2)$.
Let $h$ be a positive-definite Hermitian inner product preserved by $\rho|_H$. 
We first check $\on{im}\rho$ lies in the general unitary group, i.e.~it preserves $h$ up to scaling.
	Let $V$ be the underlying vector space of our
	representation $\rho$. The invariant inner product $h: V \times V \to \mathbb C$ corresponds to a $\mathbb{C}$-linear isomorphism of $H$-representations $\phi: V \to \overline{V}^\vee$.
	The linear map $\phi$, and hence $h$, is unique up to scaling by Schur's lemma. 
	For any $g\in G$, $$h^g: (v, w)\mapsto h(\rho(g)v, \rho(g)w)$$ is
	another $H$-invariant Hermitian form, and hence $h^g=c\cdot h$ for some
	$c\in\mathbb{C}^\times.$ 
	Equivalently, $\on{im}(\rho)\subset GU(h)$. 
	Since $\rho$ has finite determinant, we moreover obtain
	$\on{im}(\rho)\subset U(h)$.\qedhere
\end{proof}
We conclude the section by giving a convenient description of local systems
$\mathbb V$ on a family of punctured curves $\pi^\circ: \mathscr C^\circ \to
\mathscr M$, with unitary fibral monodromy. It may be useful for the reader to consider the case $A=\mathbb{C}$.

As in \autoref{notation:versal-family}, let $\pi: \mathscr{C}\to \mathscr{M}$ be a versal family of $n$-pointed curves, and let $\mathscr{C}^\circ\to \mathscr{M}$ be the associated family of punctured curves. Let $m\in\mathscr{M}$ be a point.

\begin{lemma}
	\label{lemma:unitary-direct-sum}
	Let $A$ be an Artin local $\mathbb{C}$-algebra with residue field $\mathbb{C}$.
	Suppose $\mathbb V$ is a local system of free $A$-modules on $\mathscr C^\circ$
	such that $\mathbb V|_{\mathscr C^\circ_m}$ is a constant deformation of a unitary local system, i.e.~there exists a unitary $\mathbb{C}$-local system $\mathbb{V}_0$ on $\mathscr C^\circ_m$ such that $\mathbb V|_{\mathscr C^\circ_m}\simeq \mathbb{V}_0\otimes_\mathbb{C} A$.  
	
	There is a dominant \'etale map $\mathscr M' \to \mathscr M$, with $\mathscr{C}'=\mathscr{M}'\times_\mathscr{M}\mathscr{C}$ and ${\pi'}^\circ: {\mathscr{C}'}^\circ\to \mathscr{M}'$  the associated family of punctured curves, over
	which 
	\begin{align*}
	\mathbb V|_{ {\mathscr{C}'}^\circ} \simeq \oplus_{i=1}^s \mathbb U_i
\otimes ({\pi'}^\circ)^* \mathbb
W_i,
	\end{align*}
where the $\mathbb W_i$ are locally constant sheaves of free $A$-modules on $\mathscr M'$ and
	the $\mathbb U_i$ are unitary local systems on ${\mathscr{C}'}^\circ$.
	Moreover, for $C'^\circ$ a fiber of $\mathscr{C}'^\circ \to
	\mathscr{M}'$, each $\mathbb{U}_i|_{C'^\circ}$ is irreducible, the
	$\mathbb{U}_i|_{C'^\circ}$ are pairwise non-isomorphic,
	and $\mathbb W_i \simeq {\pi'}^\circ_*\Hom(\mathbb U_i, \mathbb V|_{ {\mathscr{C}'}^\circ} )$.
\end{lemma}
\begin{proof}
Since $\mathbb V_0$ is unitary, we can express it as
as a sum of irreducible unitary local systems, 
$\mathbb V_0 \simeq \oplus_{i=1}^s \mathbb
S_i^{\oplus n_i}$, with the $\mathbb{S}_i$ pairwise non-isomorphic.
Let $\rho:\pi_1(\mathscr C_m^\circ) \to \on{GL}_r(\mathbb C)$ be the monodromy representation associated to $\mathbb{V}_0$, and let
$\rho_i$ be the (irreducible) representation associated to $\mathbb{S}_i$.
By \autoref{proposition:MCG-finite-family-of-curves}, $\rho$ is MCG-finite. Hence each
$\rho_i$ is MCG-finite by \autoref{proposition:basic-properties}.

By repeatedly applying \autoref{corollary:converse-to-MCG-rep},
there is a dominant \'etale map $\mathscr M' \to \mathscr M$ and representations
$\rho'_i: \pi_1({\mathscr {C}'}^\circ) \to \on{GL}_r(\mathbb C)$ with finite determinant so that for any $m' \in \mathscr
M'$, $\rho'_i$ restricts to a representation
$\pi_1({\mathscr{C}'}^\circ_{m'}) \to \on{GL}_r(\mathbb C)$
identified with $\rho_i$.
Let $\mathbb U_i$ denote the local system on ${\mathscr{C}'}^\circ$ corresponding to
$\rho_i'$.
Each $\mathbb U_i$ is unitary by \autoref{lemma:finite-rho-implies-finite-lift}.
The $\mathbb{U}_i$ 
are irreducible and pairwise non-isomorphic
because the same holds for the $\rho_i$.

Let $\mathbb W_i = \pi'^\circ_*\Hom(\mathbb U_i, \mathbb V|_{\mathscr
C'^\circ})$, as in the statement. The $\mathbb{W}_i$ are locally constant sheaves of free $A$-modules by the hypothesis that $\mathbb V|_{\mathscr C^\circ_m}\simeq \mathbb{V}_0\otimes_\mathbb{C} A$. There is a natural evaluation map
\begin{align*}
	\psi: \oplus_{i=1}^s (\pi'^\circ)^* \mathbb W_i \otimes \mathbb U_i \to \mathbb
	V|_{{\mathscr{C}'}^\circ}.
\end{align*}
It only remains to show $\psi$ is an isomorphism.
We do so after restriction to any fiber
of $\mathscr C'^\circ \to \mathscr M'$, where $\psi$ is identified with the
isomorphism
\begin{align}
	\nonumber
	\oplus_{i=1}^s \Hom(\rho_i, \rho\otimes A) \otimes \rho_i  & \rightarrow \rho\otimes A\\
	\sum_{i=1}^s \alpha_1 \otimes v_i & \mapsto \sum_{i=1}^s \alpha_i(v_i).
	\qedhere
\end{align}
\end{proof}

\section{Deformation-theoretic preliminaries}\label{section:deformation-theory}
In this section we introduce some deformation theory of group representations. In \autoref{section:main-proof}, we will use the theory developed here in conjunction with the vanishing results proven in \autoref{section:vanishing} to show that MCG-finite representations have certain rigidity properties.
\subsection{Deformations of representations}\label{subsection:deformation-basics}
We begin by recalling a standard description of deformations of representations in terms of cohomological data.

Recall 
\autoref{notation:adjoint}:
for $R$ a ring and a representation $\rho: G\to \on{GL}_r(R)$, we use $\on{ad}(\rho)$
to denote the adjoint $G$-representation obtained by composing $\rho$ with
$\on{GL}_r(R)\to \on{GL}(\mathfrak{pgl}_r(R)).$ 

Fix a group $G$ and a representation $\rho_0: G\to GL_r(\mathbb{C})$. Let $\on{Art}$ be the category of local Artin $\mathbb{C}$-algebras with residue field $\mathbb{C}$. If $A\in\on{Art}$, with
maximal ideal $\mathfrak{m}_A$, we say that a representation $\rho_A: G\to
\on{GL}_r(A)$ is a deformation of $\rho_0$ if $\rho_A=\rho_0 \bmod \mathfrak{m}_A$. The \emph{constant} deformation is the deformation obtained via the composition $$G\overset{\rho_0}{\to}GL_r(\mathbb{C})\to GL_r(A).$$ We say a representation $\rho: G\to GL_r(A)$ is \emph{conjugate to a constant representation} if there exists a matrix $M\in GL_r(A)$ such that $M\rho M^{-1}$ factors through $GL_r(\mathbb{C})$.
Say a deformation $\rho_A$ of $\rho_0$ has \emph{constant determinant} if $\det\rho_A=\det\rho_0$, regarded as a map $G\to A^\times$ via the natural inclusion $\mathbb{C}^\times\hookrightarrow A^\times$.

We now define deformations of representations over square-zero extensions in
$\on{Art}$.
  Given a square-zero extension $$0\to I\to B\to A\to 0$$ with $A, B\in\on{Art}$, and a representation $\rho_A: G\to GL_r(A)$, a deformation of $\rho_A$ is a representation $\rho_B: G\to \on{GL}_r(B)$ with $\rho_B=\rho_A\bmod I$. Two such deformations $\rho_B, \rho_B'$ are \emph{equivalent} if there exists a matrix $M\in \on{GL}_r(B)$ with $\rho_B=M\rho_B'M^{-1}$, such that $M=\on{id}\bmod I$. 
  \begin{lemma}
	\label{lemma:tangent-space}
	Fix a group $G$ and a representation $\rho_0 : G \to \on{GL}_r(\mathbb C)$. Consider $$0\to \epsilon A\to A[\epsilon]/\epsilon^2\to A\to 0$$ with $A\in\on{Art}$, and let $\rho_A: G\to \on{GL}_r(A)$ be a deformation of $\rho_0$ over $A$ with constant determinant.
	The set of equivalence classes of deformations of $\rho_A$ to an $A[\epsilon]/\epsilon^2$-representation with constant determinant is naturally in bijection with $H^1(G, \on{ad}(\rho_A))$.
	Moreover, this description is functorial in $G$.
\end{lemma}
\begin{proof}
See \cite[Chapter V, \S21, Proposition 1]{mazur1997introduction}.
\end{proof}

\begin{lemma}\label{lemma:split-constancy}
	Let $$1\to N\to G\to Q\to 1$$ be a short exact sequence of groups (i.e.~$N$ is normal in $G$ and $Q=G/N$). Let $A\in\on{Art}$ and $$\rho: G\to GL_r(A)$$ be a representation, and suppose that $H^0(Q, H^1(N, \on{ad}(\rho)|_N))=0.$ If $\rho_\epsilon$ is any deformation of $\rho$ to $A[\epsilon]/\epsilon^2$ with constant determinant, then $\rho_\epsilon|_N$ is equivalent to $\rho\otimes_A A[\epsilon]/\epsilon^2|_N$.
\end{lemma}
\begin{proof}
	By \autoref{lemma:tangent-space}, the set of equivalence classes of deformations of $\rho$ to $A[\epsilon]/\epsilon^2$ is naturally in bijection with $H^1(G, \on{ad}(\rho))$, and the set of equivalence classes of deformations of $\rho|_N$ is naturally in bijection with $H^1(N, \on{ad}(\rho)|_N)$. By functoriality, the restriction map is given by the natural map $$H^1(G, \on{ad}(\rho))\to H^1(N, \on{ad}(\rho)|_N),$$ which, by the $5$-term exact sequence from the Hochschild-Serre spectral sequence, factors through $H^0(Q, H^1(N, \on{ad}(\rho)|_N))=0.$ Hence $\rho_\epsilon|_N$ is equivalent to $\rho\otimes_A A[\epsilon]/\epsilon^2|_N$ as desired.
\end{proof}

\subsection{A criterion for constancy}
The next proposition gives a criterion for a deformation of a group representation to induce the trivial deformation on a normal subgroup.
This will be used in the proof of the semisimple version of our main theorem,
\autoref{theorem:finite-image-semisimple}.

We define  $$B_n=\mathbb{C}[t]/t^{n+1},$$ $$R_n=B_{n-1}[\epsilon]/\epsilon^2=\mathbb{C}[t, \epsilon]/(t^n,\epsilon^2).$$ Let $d_n: B_n\to R_{n}$ be the map given by $$d_n: f(t)\mapsto f(t)+\epsilon f'(t),$$ where $f'(t)$ is the derivative of $f(t)$. Note that $d_n$ is injective.

Below we say that an element of $GL_r(B_n)$ is constant if it lies in $GL_r(\mathbb{C})\subset GL_r(B_n)$. We say that an element of $GL_r(\mathbb{C}[[t]])$ is constant mod $t^{n+1}$ if its image in $GL_r(B_n)$ is constant. Likewise, a representation $\rho$ into $GL_r(B_n)$ is constant if it factors through $GL_r(\mathbb{C})$; it is conjugate to a constant representation if there exists $M\in GL_r(B_n)$ such that $M\rho M^{-1}$ is constant. Similarly, a representation into $GL_r(\mathbb{C}[[t]])$ is constant mod $t^{n+1}$ if its image in $GL_r(B_n)$ is constant.
\begin{proposition}\label{prop:stacky-split-constancy}
	Let $$1\to N\to G\to Q\to 1$$ be a short exact sequence of groups. Let $$\rho_0: G\to GL_r(\mathbb{C})$$ be a representation, and suppose that for all $n$ and all deformations $\gamma: G\to GL_r(B_n)$ of $\rho_0$ with $\gamma|_N$ constant, we have $H^0(Q, H^1(N, \on{ad}(\gamma)|_N))=0.$ If $$\rho_\infty: G\to GL_r(\mathbb{C}[[t]])$$ is any representation with $\rho_\infty=\rho_0\bmod t$ and constant determinant, then $\rho_\infty|_N$ is conjugate to a constant representation. That is, there exists $M_\infty\in GL_r(\mathbb{C}[[t]])$ such that $M_\infty\rho_\infty|_NM_\infty^{-1}$ factors through $GL_r(\mathbb{C})$.
\end{proposition}
\begin{proof}
We wish to construct $M_\infty\in GL_r(\mathbb{C}[[t]])$ such that $M_\infty\rho_\infty|_NM_\infty^{-1}$ is constant, i.e.~factors through $GL_r(\mathbb{C})$. We do so by successive approximation. For all $m\geq 0$, set $\rho_m: G\to GL_r(B_m)$ to be $\rho_\infty \bmod t^{m+1}$.

Set $S_1=\on{id}$. Suppose we have found $S_n\in GL_r(\mathbb{C}[[t]])$ such
that $S_n\rho_\infty|_N S_n^{-1}$ is constant modulo $t^n$.  We claim it
suffices to construct an element $M_n\in GL_r(B_n), M_n=\on{id}\bmod t^n$, such
that $M_nS_n\rho_n|_N S_n^{-1}M_n^{-1}$ is constant. Indeed, let
$M_1 = \id$ and for $n \geq 1$, let
$\widetilde{M_n}$ be an arbitrary lift of $M_n$ to $GL_r(\mathbb{C}[[t]])$. 
Then, the representation $\widetilde{M_n}S_n\rho_\infty|_NS_n^{-1}\widetilde{M_n}^{-1}$ is
constant mod $t^{n+1}$. Set $S_{n+1}=\widetilde{M_n}S_n$, so that $S_{n+1} =
\widetilde{M_n} \cdot \widetilde{M_{n-1}} \cdots \widetilde{M_1}$ by induction.
Then, setting $$M_\infty:=\lim_{n\to
\infty}\widetilde{M_n}\cdot\widetilde{M_{n-1}}\cdots
\widetilde{M_1}=\lim_{n\to\infty} S_n,$$ we claim that the limit converges as
$\widetilde{M_n}\to \on{id}$, and $M_\infty\rho_\infty|_N M_{\infty}^{-1}$ is
constant. This claim holds because $M_\infty=S_n\bmod t^n$ for all $n$, so we
have $$M_\infty \rho_\infty|_N M_\infty^{-1}\bmod t^n=S_n\rho_\infty|_N
S_n^{-1}\bmod t^n,$$ which is constant for all $n$ by construction. 

We now construct the desired matrix $M_n\in GL_r(B_n)$. Replacing $\rho_\infty$ by $S_n\rho_\infty S_n^{-1}$, we may assume $\rho_\infty|_N$ is constant modulo $t^n$, i.e.~$\rho_{n-1}|_N$ is constant. We wish to find $M_n\in GL_r(B_n)$ such that
\begin{enumerate}
\item $M_n=\on{id} \bmod t^n$, and
\item $M_n\rho_n|_NM_n^{-1}$ is constant.
\end{enumerate}

Given $g\in N$, we let $\rho_n(g)'\in \on{Mat}_{r\times r}(B_{n-1})$ be the matrix obtained by differentiating the entries of $\rho_n(g)$. The representation $d_n\circ \rho_n|_N$ is given by $$g\mapsto \rho_0(g)+\epsilon\rho_n(g)',$$ as $\rho_n|_N$ is constant mod $t^{n}$.

Note that $d_n\circ \rho_n|_N: N\to GL_r(R_n)$ is constant mod $\epsilon$ (as it is equal to $\rho_{n-1}|_N$ mod $\epsilon$). By hypothesis, $$H^0(Q, H^1(N, \on{ad}(\rho_{n-1})|_N))=0,$$ and so by \autoref{lemma:split-constancy}, $d_n\circ\rho_n|_N$ is conjugate to the constant representation $$\rho_{n-1}|_N\otimes_{B_{n-1}} R_n=\rho_0|_N\otimes_{\mathbb{C}} R_n$$ 
by some matrix $$\on{id}+\epsilon\sum_{i=0}^{n-1}C_it^i\in GL_r(R_n),$$ where
the $C_i\in \on{Mat}_{r\times r}(\mathbb{C})$. The idea of the remainder of the proof is to view the matrix above as a vector field which we may flow along to make $\rho_n|_N$ constant; we find the desired conjugating matrix $M_n$ by ``integrating" this vector field.

We compute that for $g\in N$,
\begin{align*}
	(\rho_0\otimes_{\mathbb{C}} R_n)(g) &= (\rho_{n-1}\otimes_{B_{n-1}}R_n)(g)\\
	&=(\on{id}+\epsilon\sum_{i=0}^{n-1}C_it^i)(d_n\circ \rho_n(g))(\on{id}-\epsilon\sum_{i=0}^{n-1}C_it^i)\\
	&=(\on{id}+\epsilon\sum_{i=0}^{n-1}C_it^i)(\rho_0(g)+\epsilon\rho_n(g)')(\on{id}-\epsilon\sum_{i=0}^{n-1}C_it^i)\\
&=\rho_0(g)+\epsilon\rho_n(g)'+\epsilon\sum_{i=0}^{n-1}[C_i, \rho_{0}(g)]t^i.
\end{align*}
As $\rho(g)$ is constant mod $t^n$, $\rho(g)'$ is $0$ mod $t^{n-1}$. Hence equating coefficients, we find $$[C_i, \rho_0(g)]=0 \text{ for $i<n-1$}$$ and  $$\rho_n(g)'=[\rho_0(g), C_{n-1}]t^{n-1}.$$
Now set $$M_n=\on{id}+\frac{C_{n-1}}{n}t^n\in GL_r(B_n).$$ We claim $M_n\rho_n|_N M_n^{-1}$is constant. Indeed,
\begin{align*}
M_n\rho_n(g) M_n^{-1} &= (\on{id}+\frac{C_{n-1}}{n}t^n)\rho_n(g)(\on{id}-\frac{C_{n-1}}{n}t^n)\\
&=\rho_n(g)+\frac{t}{n}[t^{n-1}C_{n-1}, \rho_n(g)]\\
&=\rho_n(g)+\frac{t}{n}[t^{n-1}C_{n-1}, \rho_0(g)]\\
&=\rho_n(g)-\frac{t}{n}\rho_n(g)'.
\end{align*}
But the only non-vanishing coefficient of $\rho(g)'$ is the coefficient of $t^{n-1}$, so the above expression is constant as desired.
\end{proof}

\section{Hodge-theoretic preliminaries}\label{section:Hodge-theoretic-preliminaries}
In this section we recall some preliminaries on polarizable complex variations
of Hodge structure (complex PVHS), mixed Hodge theory, and Simpson-Mochizuki's non-abelian Hodge theory.

\subsection{Variations of mixed Hodge structure}
\label{subsection:vmhs}
Our main goal in this section is \autoref{theorem:unitary-MHS},
which describes the variation of mixed Hodge structure on the cohomology of a
unitary local system on a family of punctured curves. 
We then describe the resulting bigrading explicitly in
\autoref{lemma:three-part-hodge-structure}.

We now review pertinent notation.
We refer the reader to \cite[\S5]{LL:geometric-local-systems} for basic definitions
related to complex PVHS.
Let $\mathbb V$ be a unitary local system on a smooth quasi-projective variety.
We say that \emph{$\mathbb{V}$ has a real structure} if there exists a real orthogonal local system $\mathbb{V}_\mathbb{R}$ such that 
$\mathbb{V}\simeq \mathbb{V}_\mathbb{R}\otimes_\mathbb{R}\mathbb{C}$. Note that
for any unitary local system $\mathbb{V}$, the local system
$\mathbb{V}\oplus\mathbb{V}^\vee$ has a natural real structure.

Let $\pi: \mathscr{C} \to \mathscr{M}$ be a family of $n$-pointed curves as in
\autoref{notation:versal-family}. 
Further assume that $\mathscr{M}$ is smooth. 
Let $\pi^\circ: \mathscr{C}^\circ\to
\mathscr{M}$ be the associated family of punctured curves, 
$j : \mathscr C^\circ \to \mathscr C$ the inclusion, 
$\mathscr{D}=\mathscr{C}\setminus\mathscr{C}^\circ$,
and $\mathbb V$ a unitary local system on $\mathscr C^\circ$.
Below we will use the notion of the Deligne canonical extension of a flat vector bundle from $\mathscr C^\circ$
to $\mathscr C$, as described in
\cite[Remarques 5.5(i)]{deligne:regular-singular} or \cite[Definition
4.1.2]{LL:geometric-local-systems}. 

See \cite[Definitions 14.44, 14.45, and 14.49]{peters2008mixed} for the definitions of variations of mixed Hodge structure, graded-polarizability
and 
admissibility in the rational case; the real case is analogous and discussed in \cite[Definition 3.4]{steenbrink1985variation}.

The following is well-known but perhaps difficult to extract from the
literature.
\begin{theorem}\label{theorem:unitary-MHS}
	Suppose $\mathbb{V}$ is a unitary local system on $\mathscr{C}^\circ$
	with real structure $\mathbb{V}_\mathbb{R}$. Then,
	$R^1\pi^\circ_*\mathbb{V}_\mathbb{R}$ underlies an admissible
	graded-polarizable real variation of mixed Hodge structure with weights
	in $[1,2]$, and weight filtration given by 
$$W^1R^1\pi^\circ_*\mathbb{V}_\mathbb{R}:=R^1\pi_*j_*\mathbb{V}_\mathbb{R}\subset R^1\pi^\circ_*\mathbb{V}_\mathbb{R}=:W^2R^1\pi^\circ_*\mathbb{V}_\mathbb{R}.$$ 
Let $(\mathscr{E}, {\nabla})$ denote the Deligne canonical extension of $(V\otimes \mathscr{O}_{\mathscr{C}^\circ}, \on{id}\otimes d)$ to $\mathscr{C}$. The Hodge filtration on $R^1\pi^\circ_*\mathbb{V}\otimes \mathscr{O}_{\mathscr{M}}$ is given, under the canonical identification $$R^1\pi^\circ_*\mathbb{V}\otimes \mathscr{O}_{\mathscr{M}}\overset{\sim}{\longrightarrow} R^1\pi_*(\mathscr{E}\overset{{\nabla}}{\longrightarrow}\mathscr{E}\otimes \Omega^1_{\mathscr{C}/\mathscr{M}}(\log \mathscr{D})),$$ by the filtration induced by the \emph{filtration b\^{e}te} on the relative de Rham complex $\mathscr{E}\overset{{\nabla}}{\longrightarrow}\mathscr{E}\otimes \Omega^1_{\mathscr{C}/\mathscr{M}}(\log \mathscr{D})$. 
That is, $$F^1R^1\pi^\circ_*\mathbb{V}\otimes \mathscr{O}_\mathscr{M}:=\on{im}(\pi_*(\mathscr{E}\otimes \Omega^1_{\mathscr{C}/\mathscr{M}}(\log \mathscr{D}))\to R^1\pi^\circ_*\mathbb{V}\otimes \mathscr{O}_\mathscr{M}).$$
\end{theorem}
\begin{proof}	
First, Griffiths transversality holds vacuously, as the Hodge filtration only has two steps. 
To verify the remaining properties of a variation of mixed Hodge structure
we may do so pointwise. 
Over a point, 
the verification that 
$R^1\pi^\circ_*\mathbb{V}_\mathbb{R}$ underlies a real variation of mixed Hodge
structure is carried out in 
\cite[Theorem, p. 152]{timmerscheidt:mixed-hodge-structure-for-unitary}.
The description of the weight filtration follows from \cite[Lemma
6.2]{timmerscheidt:mixed-hodge-structure-for-unitary}.
The description of the Hodge filtration follows from the definition of the
Hodge filtration and the degeneration of the Hodge-de Rham spectral sequence
\cite[Theorem 7.1(a)]{timmerscheidt:mixed-hodge-structure-for-unitary}.

It remains only to verify graded-polarizability and admissibility. This follows from Saito's theory \cite{saito1990mixed} (also see
\cite{schnell2014overview}) as we now explain. 
As $\mathbb{V}_\mathbb{R}$ is orthogonal, it underlies a real PVHS with trivial Hodge filtration and polarization arising from the orthogonal structure.
Hence, $Rj_*(\mathbb{V}_\mathbb{R})$ underlies an
object of the derived category of (real) mixed Hodge modules on $\mathscr{C}$. Hence
$R\pi^\circ_*\mathbb{V}_\mathbb{R}=R\pi_*Rj_*\mathbb{V}_\mathbb{R}$ 
underlies an object of the derived category of (real) mixed Hodge
modules on $\mathscr{M}$. As $\pi^\circ$ is a fiber bundle and hence $R^1\pi^\circ_*\mathbb{V}_\mathbb{R}$ is locally constant, this implies (by \cite[Theorem 21.1]{schnell2014overview}) that
$R^1\pi^\circ_*\mathbb{V}_\mathbb{R}$ in fact underlies a graded-polarizable
admissible variation of mixed Hodge structure. It remains only to compare it with
the structure described in the theorem. But this follows by compatibility
when $\mathscr{M}$ is a point, by base change.
\end{proof}
Let $\mathbb{V}$ be a unitary local system on $\mathscr{C}^\circ$, and let $\mathscr{H}_\mathbb{V}:=R^1\pi^\circ_*\mathbb{V}\otimes \mathscr{O}_\mathscr{M}$.   If $\mathbb{V}$ has a real structure $\mathbb{V}_\mathbb{R}$, as in \autoref{theorem:unitary-MHS}, then  $R^1\pi_*^\circ \mathbb{V}_\mathbb{R}$ underlies a natural polarized variation of $\mathbb{R}$-mixed Hodge structure. In general, we let $W^\bullet, F^\bullet$ be the weight and Hodge filtrations respectively, and $\overline{F}^\bullet$ the conjugate Hodge filtration induced by the natural isomorphism $$\mathscr{H}_{\mathbb{V}}\simeq \overline{\mathscr{H}_{\mathbb{V}^\vee}}.$$ Define $$\mathscr{H}_{\mathbb{V}}^{1,0}:= F^1\cap W^1,$$ $$\mathscr{H}^{0,1}_{\mathbb{V}} := \overline{F^1}\cap W^1,$$ and $$\mathscr{H}^{1,1}_{\mathbb{V}}=F^1\cap \overline{F^1}.$$ 
We give a schematic picture of the structures on $\mathscr{H}_{\mathbb{V}}$ in \autoref{figure:Hodge-diagram}.

\begin{figure}[h]
	\centering
	\includegraphics[scale=.3, trim={0cm 10.0cm 0 5cm}]{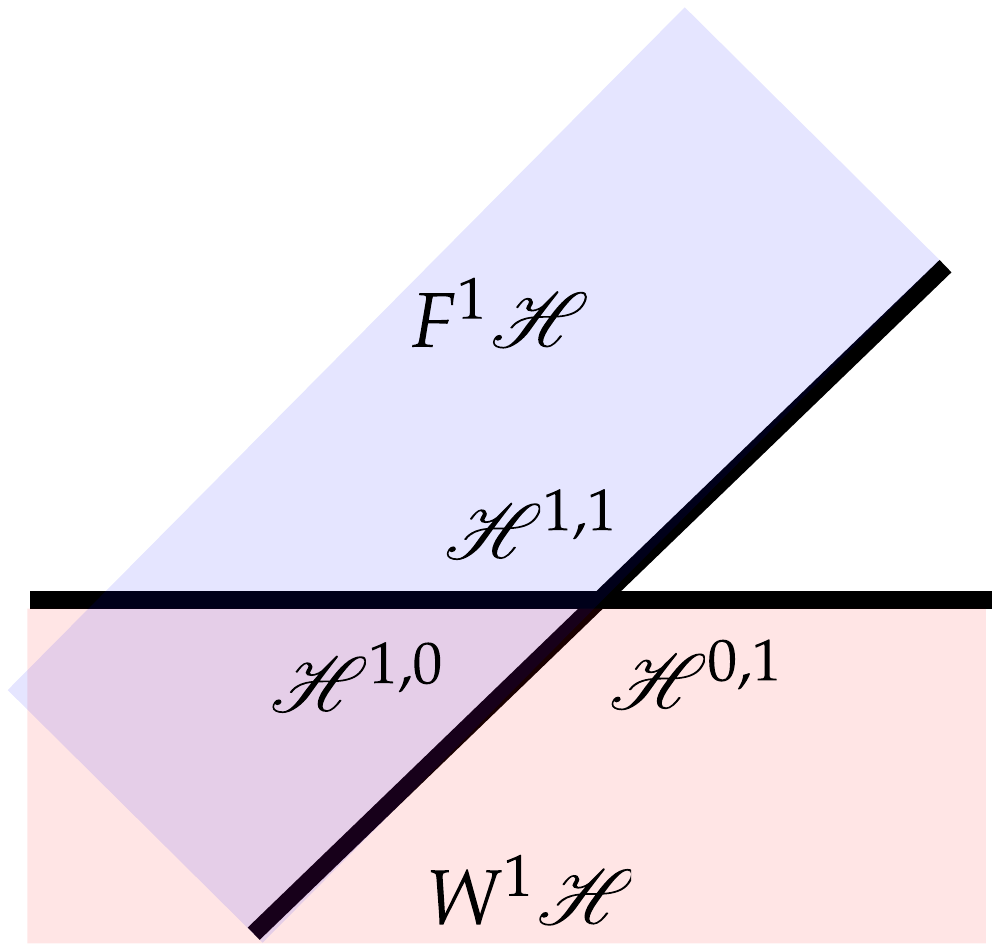}
\caption{A schematic diagram of the mixed Hodge structure on the cohomology of a unitary local system.}
\label{figure:Hodge-diagram}
\end{figure}
%
\begin{lemma}
	\label{lemma:three-part-hodge-structure}
	For $\mathbb V$ a unitary local system on $\mathscr C^\circ$, there is a natural
	direct sum decomposition
\begin{equation}
\label{eqn:hodge-decomposition-cpx}\mathscr{H}_\mathbb{V}=\mathscr{H}_{\mathbb{V}}^{1,0}\oplus
\mathscr{H}_{\mathbb{V}}^{0,1}\oplus
\mathscr{H}_{\mathbb{V}}^{1,1}\end{equation}
Further, under the natural isomorphism
$\overline{\mathscr{H}_\mathbb{V}}=\mathscr{H}_{\mathbb{V}^\vee}$, 
\begin{align}
	\label{equation:dual-conjugate}
\overline{\mathscr{H}_{\mathbb{V}}^{i,j}}=\mathscr{H}_{\mathbb{V}^\vee}^{j,i}.
\end{align}
\end{lemma}
\begin{proof}
	One may check both statements hold
	by doing so on fibers, in which case it is a straightforward
	verification using the definitions of the weight and Hodge filtrations.
\end{proof}

\subsection{Variations on the theorem of the fixed part}
\label{subsection:fixed-part}

Our main goal here is to prove
\autoref{proposition:subrep-real-variation}, which is a variant of the theorem of the fixed part.
\begin{theorem}[Theorem of the fixed part, {\cite[4.19]{steenbrink1985variation}}, {\cite[14.52]{peters2008mixed}}]\label{theorem:theorem-of-the-fixed-part}
	Let $\mathscr{X}$ be a smooth quasiprojective complex variety, and let $\mathbb{W}$ be an admissible graded-polarizable real variation of mixed Hodge structure on $\mathscr{X}$. Then there is a natural real Hodge structure on $H^0(\mathscr{X}, \mathbb{W})$ such that the natural inclusion $$H^0(\mathscr{X}, \mathbb{W})\otimes \underline{\mathbb{R}}\to \mathbb{W}$$ is a morphism of real variations of mixed Hodge structure.
\end{theorem}


For our main results, will need a variant of the theorem of the fixed part for irreducible representations appearing
the cohomology of a unitary local system on a family of curves.
Given a complex local system $\mathbb L$ on a variety, we define the local system
$\widetilde{\mathbb L}$ by
\begin{align}
	\label{equation:real-construction}
	\widetilde{\mathbb L} :=
	\begin{cases}
		\mathbb L & \text{ if $\mathbb L$ admits a real structure}  \\
		\mathbb L \oplus \overline{\mathbb L} & \text{ otherwise.}  \\
	\end{cases}
\end{align}
Note that $\widetilde{\mathbb L}$ admits a natural real structure. Note that if $\mathbb{L}$ is an irreducible local system underlying a complex variation of Hodge structure (necessarily unique up to reindexing), and if $\mathbb{L}$ admits a (necessarily unique) real structure, then $\mathbb{L}$ in fact underlies a real variation of Hodge structure. Indeed, the complex variation carried by $\mathbb{L}$ must be preserved (up to reindexing) by complex conjugation, as otherwise $\mathbb{L}$ would carry two distinct complex variations.
\begin{proposition}\label{proposition:subrep-real-variation}
	Let $\mathscr{X}$ be a smooth quasiprojective complex variety, and let $\mathbb{W}$ be an admissible graded-polarizable real variation of mixed Hodge structure on $\mathscr{X}$. Let $\mathbb{L}$ be an irreducible complex local system on $\mathscr{X}$ such that $\Hom(\mathbb{L}, \mathbb{W}_\mathbb{C})\not = 0.$ 
	Then, $\widetilde{\mathbb L}$ underlies a polarizable real variation of Hodge structure, and there exists a nonzero \emph{constant} real  mixed Hodge structure $Q$ and a nonzero map of variations of mixed Hodge structure $$Q\otimes \widetilde{\mathbb{L}}\to \mathbb{W}.$$
	\end{proposition}
	\begin{proof}
	First, let $i$ be an integer for which there exists a nonzero map $\mathbb{L}\to
	\on{gr}^i_W\mathbb{W}_\mathbb{C}$. Then $$\on{gr}^i_W
	\mathbb{W}_\mathbb{C}=\oplus_j \mathbb{V}_j\otimes W_j,$$ where the
	$\mathbb{V}_j$ are distinct irreducible local systems carrying
	polarizable complex variations of Hodge structure, and the $W_j$ are
	constant complex variations \cite[Proposition 4.1.4(2)]{LL:geometric-local-systems}. By assumption $\mathbb{L}$
	occurs among the $\mathbb{V}_j$, and hence carries a complex PVHS. 
	Hence
	$\widetilde{\mathbb{L}}$ carries a real PVHS.
	
		Let $Q:=\Hom(\widetilde{\mathbb{L}},
		\mathbb{W})=H^0(\mathscr{X}, \widetilde{\mathbb{L}}^\vee\otimes
		\mathbb{W})$. The real variation
		$\widetilde{\mathbb{L}}^\vee\otimes \mathbb{W}$ is a tensor
		product of a pure (hence admissible) variation with an
		admissible variation, hence is admissible itself. By the theorem
		of the fixed part, \autoref{theorem:theorem-of-the-fixed-part}, $Q$ carries a canonical real mixed Hodge structure such that the evaluation map $$Q\otimes \widetilde{\mathbb{L}}\to \mathbb{W}$$ is a map of variations, as desired. The map is nonzero by construction.
	\end{proof}


\subsection{Consequences of non-abelian Hodge theory}
The main result from non-abelian Hodge theory that we will use is the following result of Mochizuki, generalizing earlier work of Simpson in the projective case:
\begin{theorem}[~\protect{\cite[Theorem 10.5]{mochizuki2006kobayashi}}]
	\label{theorem:deformation-to-pvhs}
	Let $\overline{X}$ be a smooth projective variety and $D\subset
	\overline{X}$ a strict normal crossings divisor. Let
	$X=\overline{X}\setminus D$. 
	Any representation $$\rho: \pi_1(X)\to
	\on{GL}_r(\mathbb{C})$$ with finite determinant admits a
	deformation with constant determinant to a representation underlying a complex PVHS.
\end{theorem}
\begin{proof}
Aside from the statement about determinants, this is precisely	\cite[Theorem 10.5]{mochizuki2006kobayashi}; the statement about determinants follows by examining the proof of that theorem. 
\end{proof}

We also use the related result that cohomologically rigid representations underlie complex PVHS.
\begin{lemma}\label{lemma:rigid-reps}
Let $X, D, \rho$ be as in \autoref{theorem:deformation-to-pvhs}. Suppose in
addition that $\rho$ is semisimple with finite determinant and $H^1(X, \on{ad}\rho)=0$. Then $\rho$ underlies a complex PVHS.
\end{lemma}
\begin{proof}
	Let $G$ denote the Zariski-closure of the image of $\rho$.
	Note $G$ is reductive by semisimplicity of $\rho$. Let $\mathfrak{g}$ be
	the Lie algebra of $G$, viewed as a $\pi_1(X)$-representation. Note
	that $\mathfrak{g}$ is a subalgebra of $\on{ad}\rho$ as $\rho$ has
	finite determinant, and is moreover a summand of $\on{ad}\rho$ as
	$\on{ad}\rho$ is semisimple. Hence $\rho$ is cohomologically rigid as a
	$G$-representation, as $H^1(X, \mathfrak{g}) \subset H^1(X, \on{ad}
	\rho) =0$.
	
	We may apply \cite[Lemma 10.13]{mochizuki2006kobayashi},
	which yields that 
	$\rho$ admits a deformation as a $G(\mathbb C)$-representation to 
	$\rho_0: \pi_1(X) \to G(\mathbb C)$, which
	underlies a PVHS.
	But 
	$\rho$ is rigid and hence admits no
	non-trivial deformations.	This implies $\rho = \rho_0$, so $\rho$ underlies a PVHS.
\end{proof}

\section{The period map associated to a unitary representation}
\label{section:unitary-period-map}

In this section, we give an explicit description of the derivative of the period map
on the cohomology of a unitary vector bundle over a punctured curve.
In \autoref{subsection:period-map-derivative},
we set up notation to describe the derivative of the period map,
and state our description in
\autoref{theorem:period-map-computation}.
In \autoref{subsection:bilinear-pairing},
we study a natural bilinear pairing on vector bundles,
which is closely related to the derivative of the period map.

\subsection{The period map}
	\label{subsection:period-map-derivative}
	We now describe the period map associated to the 
first cohomology of a unitary local system over a punctured curve.
We aim to state \autoref{theorem:period-map-computation}, which is a
generalization of the classical statement that, for $C$ a curve, the multiplication map
$$H^0(C, \omega) \otimes H^0(C, \omega) \to H^0(C, \omega^{\otimes 2})$$ can be
identified via Serre duality with the derivative of the period map associated to the Hodge structure on
$H^1(C, \mathbb C)$ \cite[Lemma 10.22]{Voisin:hodgeTheory}.

\begin{notation}
	\label{notation:period-map}
	Suppose $\pi: \mathscr{C} \to \mathscr{B}$ is a relative smooth proper curve over a smooth
	contractible complex-analytic base $\mathscr{B}$. Suppose $s_1, \cdots,
	s_n: \mathscr{B}\to \mathscr{C}$ are sections to $\pi$ with disjoint
	images $\mathscr{D}_1, \cdots, \mathscr{D}_n$, and let $\mathscr{D}$ be
	the union of the $\mathscr{D}_i$. Let
	$\mathscr{C}^\circ=\mathscr{C}\setminus \mathscr{D}$,
	$j:\mathscr{C}^\circ\to \mathscr{C}$ be the natural inclusion, and let $\pi^\circ: \mathscr{C}^\circ \to \mathscr{B}$ be the composition $\pi^\circ=\pi\circ j$.
Let $\mathbb V$ be a unitary $\mathbb C$-local system on $\mathscr{C}^\circ$ and
let 
$\mathbb{W}:=R^1 \pi^\circ_* \mathbb V$ denote the higher direct image local system on $\mathscr{B}$.
Let $(\mathscr E, \nabla)$ denote the Deligne canonical extension 
\cite[Remarques 5.5(i)]{deligne:regular-singular}
of $(\mathbb{V}\otimes \mathscr{O}_{\mathscr{C}^\circ}, \on{id}\otimes d)$ to $\mathscr{C}$; this is a vector bundle with flat connection with regular singularities along $\mathscr{D}$. Note that for each $b\in \mathscr{B}$, $\mathscr{E}|_{\mathscr{C}_b}$ with the parabolic structure induced by $\nabla$ (as in \cite[Definition 3.3.1]{LL:geometric-local-systems}) is the semistable parabolic vector bundle associated to $\mathbb{V}|_{\mathscr{C}_b^\circ}$ by the Mehta-Seshadri correspondence \cite{mehta1980moduli, simpson1990harmonic}.

\end{notation}

\subsubsection{Constructing the period map}
\label{subsubsection:constructing-period-map}
We next construct the relevant period map in \eqref{equation:period-map}.
Set $\mathscr H:=\mathbb{W} \otimes_{\mathbb C}
\mathscr O_\mathscr{B} \simeq  \mathbb R^1
\pi_*( \mathscr E \overset{\nabla}{\to} \mathscr E \otimes \Omega^1_{\mathscr{C}/\mathscr{B}}(\log \mathscr{D}))$, the
right hand side being
relative hypercohomology \cite[Corollaire 6.10, Proposition 6.14]{deligne:regular-singular}.
As the vector bundle $\mathscr{H}$ arises from a local system, it carries a natural Gauss-Manin connection $\nabla_{GM}$.
Because we are assuming $\mathbb V$ is unitary, the Hodge-de Rham spectral sequence
degenerates \cite[Theorem 7.1(a)]{timmerscheidt:mixed-hodge-structure-for-unitary} (see also \cite[Th\'eor\`eme 5.3.1]{saito1988modules} for a much more general result), 
giving a $2$-step Hodge filtration of 
$\mathscr H = \mathbb W \otimes \mathscr O_\mathscr{B}$, satisfying
\begin{equation}
	\label{equation:h-filtration}
\begin{aligned}
F^1 \mathscr H &:= \pi_* (\mathscr E \otimes
\Omega^1_{\mathscr{C}/\mathscr{B}}(\log\mathscr{D})) \\
\mathscr H/ F^1 \mathscr H &\simeq R^1 \pi_* \mathscr E,
\end{aligned}
\end{equation}
as described in \autoref{theorem:unitary-MHS}.

Since $\mathscr{B}$ is contractible, we can globally trivialize the local system $\mathbb W$, yielding a flat trivialization of $\mathscr{H}$. 
We therefore obtain a period map
\begin{align}
	\label{equation:period-map}
	P: \mathscr{B} \to \on{Gr}(\rk F^1 \mathscr H, \rk \mathscr H),
\end{align}
for $\on{Gr}(s, r)$ the Grassmannian of $s$-dimensional subspaces of an $r$
dimensional vector space, which sends $b\in \mathscr{B}$ to the subspace $F^1\mathscr{H}_b\subset \mathscr{H}_b$. 

\begin{remark}
	\label{remark:}
	In the above \autoref{notation:period-map}, and in what follows,
	when $\mathscr C \to \mathscr B$ is a relative curve, we will tend to use
	$\Omega^1_{\mathscr C/\mathscr B}(\mathscr D)$ instead of the isomorphic but arguably more correct
	$\Omega^1_{\mathscr C/\mathscr B}(\log \mathscr D)$
	in order to simplify notation. 
	We will also often identify $\Omega^1_{\mathscr{C}/\mathscr{B}}$ with the relative dualizing sheaf $\omega_{\mathscr{C}/\mathscr{B}}$, and will pass between the two without comment.
\end{remark}

\subsubsection{The derivative of the period map}
\label{subsubsection:derivative-period-map-2}

Having set up notation for the period map, we now describe its derivative.
The derivative of the period map is an $\mathscr{O}_{\mathscr{B}}$-linear map
\begin{align*}
	dP : T_{\mathscr B} \to P^* T_{\on{Gr}(\rk F^1 \mathscr H, \rk \mathscr H)} \simeq (F^1
	\mathscr H)^\vee \otimes (\mathscr H/F^1 \mathscr H),
\end{align*}
where the latter canonical identification follows from \cite[Theorem 3.5]{eisenbudH:3264-&-all-that}
and the fact that the universal sub and quotient bundles on the Grassmannian
pull back under $P$ to $F^1 \mathscr H$ and 
$\mathscr H/F^1 \mathscr H$.
This is dual to a map
$$dP^\vee: (F^1\mathscr{H})\otimes
(\mathscr{H}/F^1\mathscr{H})^\vee\to \Omega^1_{\mathscr B}.$$ As $\mathscr{C},
s_1, \cdots, s_n$ is a family of $n$-pointed curves over $\mathscr{B}$, we also
obtain a classifying map $c: \mathscr{B}\to \mathscr{M}_{g,n}$, inducing a
pullback map $$c^*: c^*\Omega^1_{\mathscr{M}_{g,n}}\to \Omega^1_{\mathscr B}.$$

\subsubsection{The derivative of the period map at a point}
\label{subsubsection:derivative-period-map}

We now identify the derivative of the period map at a point $b \in \mathscr B$.
Given a point $b\in \mathscr{B}$, set $C:=\mathscr{C}_b$,
$D:=\mathscr{D}_b$, and $E:=\mathscr{E}|_C$. Via Serre duality, we identify $dP^\vee_b$ as a map $$dP^\vee_b: H^0(C, {E}\otimes \Omega^1_C(\log D))\otimes H^0(C, {E}^\vee \otimes \omega_C)\to \Omega^1_{\mathscr{B},b},$$ and $c^*_b$ as a map $$c^*_b: H^0(C, \omega_C^{\otimes 2}(D))\to \Omega^1_{\mathscr{B}, b}.$$
We also have a natural $\mathscr{O}_\mathscr{B}$-linear map $\overline{\nabla}: F^1 \mathscr H \to \mathscr H/F^1 \mathscr H \otimes
\Omega^1_{\mathscr B}$ given as the composition
\begin{align}
	\label{equation:shift-derivative}
	\overline{\nabla}: F^1 \mathscr H \to \mathscr H \xrightarrow{\nabla_{GM}} \mathscr H \otimes
	\Omega^1_{\mathscr B} \to \mathscr H/F^1 \mathscr H \otimes
	\Omega^1_{\mathscr B}.
\end{align}

A computation similar to
\cite[Lemma 10.19, Theorem 10.21, and Lemma 10.22]{Voisin:hodgeTheory},
which we carry out in 
\autoref{appendix} and complete in
\autoref{subsection:proof-period-map},
yields:
\begin{theorem}\label{theorem:period-map-computation}
	Under the identifications above, the map $dP^\vee_b$ factors as 
	\begin{equation}
	\xymatrix{
	H^0(C, {E}\otimes \Omega^1_C(\log D))\otimes H^0(C, {E}^\vee \otimes \omega_C) \ar[r]^-{dP^\vee_b}  \ar[d]^{\otimes} & \Omega^1_{\mathscr{B},b}\\
	H^0(C, {E}\otimes {E}^\vee \otimes \omega_C^{\otimes 2}(D)) \ar[r]^-{\on{tr}} & H^0(C, \omega_C^{\otimes 2}(D)) \ar[u]^{c^*_b}
	}
		\label{equation:multiplication-to-period}
	\end{equation}
where the left vertical map $\otimes$ is the tensor product of global sections, and the bottom map $\on{tr}$ is induced by the trace pairing ${E}\otimes {E}^\vee\to \mathscr{O}_C$.
	Moreover, $dP^\vee_b$ is adjoint to $\overline \nabla_b$.
\end{theorem}

\subsection{A bilinear pairing}\label{subsection:bilinear-pairing}

Let $C$ be a smooth proper connected curve of genus $g$, $D\subset C$ a reduced divisor, and $E$
a vector bundle on $C$. 
In this section, we study the natural perfect pairing
\begin{align}
	\label{equation:bilinear-pairing}
	B_E:
(E \otimes \omega_C(D))\times (E^\vee\otimes \omega_C) \to
\omega_C^{\otimes 2}(D)
\end{align}
given as the composition $$B_E: (E \otimes \omega_C(D))\times (E^\vee\otimes \omega_C)\overset{\otimes}{\longrightarrow} (E\otimes E^\vee)\otimes \omega_C^{\otimes 2}(D)\overset{\on{tr}\otimes \on{id}}{\longrightarrow} \omega_C^{\otimes 2}(D),$$
where $\on{tr}$ denotes the trace pairing ${E}\otimes {E}^\vee\to
\mathscr{O}_C.$ 
Our motivation for considering this pairing is the close relationship between
\eqref{equation:bilinear-pairing} and
\autoref{theorem:period-map-computation}: for $E$ as in \autoref{theorem:period-map-computation}, the pairing on global sections induced by $B_E$ computes the derivative of the period map.

The main result of this section is \autoref{proposition:pairing-result},
which gives a sufficient criterion for the rank of $E$ to be large in terms of
this bilinear pairing. It is this lower bound on the
rank of $E$ that ultimately leads to the bound of $\sqrt{g+1}$ in our main result,
\autoref{theorem:finite-image}.
In \autoref{section:vanishing} and \autoref{section:Putman-Wieland} we use
\autoref{proposition:pairing-result} to analyze mapping class group representations appearing in the
cohomology of unitary local systems. 

We next review background on parabolic bundles, in order to state
\autoref{proposition:pairing-result}.
For a more detailed and leisurely introduction to parabolic bundles, we suggest
the reader consult
\cite[\S2]{LL:geometric-local-systems}
and references therein.

\subsubsection{A lightning review of parabolic bundles}

Fix a smooth proper curve $C$ and let $D = x_1 + \cdots + x_n$ be a reduced divisor on
$C$. 
Recall that a {\em parabolic bundle} on $(C,D)$ is a vector bundle $E$ on $C$, a
decreasing filtration $E_{x_j} = E_j^1 \supsetneq E_j^2 \supsetneq \cdots
\supsetneq E_j^{n_j+1} =
0$ for each $1 \leq j \leq n$, and an increasing sequence of real numbers $0\leq
\alpha_j^1<\alpha_j^2<\cdots<\alpha_j^{n_j}<1$ for each $1 \leq j \leq n$.
referred to as {\em weights}. 
We use $E_\star = (E, \{E^i_j\}, \{\alpha^i_j\})$ to denote the data of a parabolic bundle.

Given a parabolic bundle
$E_\star = (E, \{E^i_j\}, \{\alpha^i_j\})$,
let $J \subset \{1, \ldots, n\}$ denote the set of 
integers $j \in \{1, \ldots, n\}$ for which $\alpha^1_j = 0$, and define
\begin{align*}
\widehat{E}_0 : = \ker( E \to \oplus_{j \in J} E_{x_j}/E_j^2).
\end{align*}
(This is a special case of more general notation used for coparabolic bundles
as in \cite[2.2.8]{LL:geometric-local-systems} or the equivalent
\cite[Definition 2.3]{bodenY:moduli-spaces-of-parabolic-higgs-bundles}, but is all we will need for this
paper.)
In particular, $\widehat{E}_0 \subset E$ is a subbundle.

Parabolic
bundles admit a notion of \emph{parabolic stability}, analogous to the usual
notion of stability for vector bundles, which we next recall. 
First, the {\em parabolic degree} of a parabolic bundle $E_\star$ is
\begin{align*}
	\on{par-deg}(E_\star) := \deg(E) + \sum_{j=1}^n \sum_{i=1}^{n_j}
	\alpha^i_j \dim(E_j^i/E_j^{i+1}).
\end{align*}
Then, the {\em parabolic slope} is defined by $\mu_\star(E_\star) :=
\on{par-deg}(E_\star)/\rk(E_\star)$.
Any subbundle $F \subset E$ has an induced parabolic structure $F_\star \subset
E_\star$ defined as
follows:
We take the filtration
over $x_j$ on $F$ is obtained
	from the filtration $$F_{x_j} = E^1_j \cap F_{x_j} \supset E^2_j \cap
	F_{x_j} \supset \cdots \supset E^{n_j+1}_j \cap F_{x_j} = 0$$
	by removing redundancies. For the weight associated to $F^i_j \subset
	F_{x_j}$ one takes $$\max_{\substack{k, 1 \leq k \leq n_j}} \{ \alpha^k_j : F^i_j = E^k_j \cap
	F_{x_j}\}.$$
A parabolic bundle $E_\star$ is {\em parabolically semi-stable} if for every
nonzero subbundle $F \subset E$ with
induced parabolic structure $F_\star$, we have $\mu_\star(F_\star) \leq
\mu_\star(E_\star)$.

Mehta and Seshadri 
\cite[Theorem 4.1, Remark 4.3]{mehta1980moduli},
give a correspondence between
\emph{parabolically stable parabolic bundles of parabolic degree zero} on
$(C,D)$ and irreducible unitary local systems on $C\setminus D$.
This bijection sends a local system $\mathbb{V}$ to the Deligne canonical extension of $(\mathbb{V}\otimes \mathscr{O}, \on{id}\otimes d)$ with the parabolic structure induced by the connection (as in \cite[Definition 3.3.1]{LL:geometric-local-systems}). 

\begin{remark}
It is not entirely trivial to extract from \cite{mehta1980moduli} that the parabolic vector bundle they construct from a unitary representation is in fact the parabolic vector bundle associated to the Deligne canonical extension of the associated flat vector bundle. What we will use is the fact that the parabolic vector bundle associated to the Deligne canonical extension is parabolically semistable, which follows from \cite[Theorem 5]{simpson1990harmonic}, for example. A generalization of this fact about unitary local systems (for complex PVHS) is nicely explained in \cite[\S5]{arapura2019vanishing}.
\end{remark}



The main technical result we will use regarding $B_E$ is:
\begin{proposition}\label{proposition:pairing-result}
	Let $E_\star$ be a semistable parabolic bundle on $(C,D)$ of parabolic
	degree zero, with underlying vector bundle $E$. Let $v\in H^0(C, E\otimes \omega_C(D))$ be a nonzero global section, and suppose the map 
\begin{align*}
	H^0(C, B_E(v,-)): H^0(C, E^\vee \otimes \omega_C) & \rightarrow H^0(C, \omega_C^{\otimes 2}(D)) \\
	u & \mapsto B_E(v,u)
\end{align*}
has rank $r$. Then $\on{rk}(E)\geq g-r.$
\end{proposition}

Before proceeding with the proof, we
recall \cite[Proposition 6.3.6]{LL:geometric-local-systems}, which the
reader may take as a black box.
For a less technical variant of this statement, see 
\cite[Proposition 6.3.1]{LL:geometric-local-systems}. 
\begin{proposition}[{ \cite[Proposition 6.3.6]{LL:geometric-local-systems}}]
	\label{proposition:generic-parabolic-global-generation}
	Suppose $C$ is a smooth proper connected genus $g$ curve and
	$E_\star = (E, \{E^i_j\}, \{\alpha^i_j\})$ is a nonzero parabolic bundle $C$
	with respect to $D = x_1 + \cdots + x_n$.
		Suppose $E_\star$ is parabolically semistable. 
	Let $U \subset \widehat E_0$ be a (non-parabolic) subbundle with
		$c := \rk E - \rk U$ and
		$\delta := h^0(C, \widehat{E}_0) - h^0(C, U)$. 
	\begin{enumerate}
		\item[(I)] If $\mu_\star({E}_\star)> 2g-2 + n$, 
		then $\rk E >
			gc - \delta$.
		\item[(II)] If $\mu_\star({E}_\star)= 2g-2+n$, then $\rk
			E \geq gc - \delta$.
	\end{enumerate}
\end{proposition}
\begin{remark}
\label{remark:}
The statement of \autoref{proposition:generic-parabolic-global-generation} is equivalent to \cite[Proposition
6.3.6]{LL:geometric-local-systems}, but differs slightly in that we write
``$E_\star$ is parabolically stable'' in place of ``$\widehat{E}_\star$ is
coparabolically stable'' and $\mu_\star(E_\star)$ in place of
$\mu_\star(\widehat{E}_\star)$.
However, by definition 
$E_\star$ is parabolically stable if and only if $\widehat{E}_\star$ is
coparabolically stable and $\mu_\star(E_\star) = \mu_\star(\widehat{E}_\star)$
\cite[Definitions 2.2.9 and 2.4.2]{LL:geometric-local-systems}.
\end{remark}

\begin{proof}[Proof of \autoref{proposition:pairing-result}]
The idea is to apply
\autoref{proposition:generic-parabolic-global-generation}(II) with $c =1$ and
$\delta = r$, and we now set up notation to do so.
Set $n=\deg D$. Let $f_v: E^\vee\otimes \omega_C\to \omega_C^{\otimes 2}(D)$ be the map of vector bundles induced by $B_E(v, -)$, and set $U=\ker(f_v)$. 
As $B_E$ is a perfect pairing and $v$ is nonzero by assumption, $U$ has corank one in $E^\vee\otimes \omega_C$.

We conclude by applying \autoref{proposition:generic-parabolic-global-generation}(II) 
to the parabolic bundle $(E_\star)^\vee \otimes
\omega_C(D)$.
This is parabolically stable by definition and
has slope $\mu((E_\star)^\vee \otimes
\omega_C(D)) = 2g - 2+n$,
since $E_\star$ has parabolic slope $0$.

By unwinding the definitions, one may verify $\reallywidehat{(E_\star)^\vee \otimes
\omega_C(D)}_0=E^\vee \otimes \omega_C.$
Therefore, the assumption of the proposition implies
\begin{align*}
	h^0(C,\reallywidehat{(E_\star)^\vee \otimes
\omega_C(D)}_0) - h^0(C, U) = 
h^0(C,E^\vee \otimes \omega_C) - h^0(C, U) = 
r.
\end{align*}
Applying \autoref{proposition:generic-parabolic-global-generation}(II) 
to the parabolic bundle
$(E_\star)^\vee \otimes
\omega_C(D)$
with $U = \ker(f_v)$ as defined above so that $c =
1$ and $\delta = r$ we find $\rk E=\rk \reallywidehat{(E_\star)^\vee \otimes
\omega_C(D)}_0 \geq g - r.$
\end{proof}

\section{The main cohomological results}\label{section:vanishing}

In \autoref{subsection:rank-bound-proof}, we prove our main cohomological result, showing that higher direct images of 
unitary systems on families of curves contain no low rank local systems.
From this we derive a related vanishing result in \autoref{subsection:no-sections},
which we will use later to establish cohomological rigidity of certain local systems.

Throughout this section, we continue to use \autoref{notation:versal-family}.
Namely, we use $\pi:\mathscr{C}\to \mathscr{M}$ for a versal family of
$n$-pointed curves of genus $g$, and $\pi^\circ: \mathscr{C}^\circ \to \mathscr{M}$ for the
associated punctured versal family of curves.

\subsection{A rank bound}
\label{subsection:rank-bound-proof}
We now prove our main result on the cohomology of unitary local systems on versal families of curves,
\autoref{theorem:rank-bound-subsystem}.
That is, for $\mathbb V$ a unitary local system on $\mathscr C^\circ$, we
will show that any nonzero
sub-local system of $R^1 \pi^\circ_* \mathbb V$ has rank at
least $2g-2\rk\mathbb{V}$.

\begin{proof}[Proof of \autoref{theorem:rank-bound-subsystem} assuming
	\autoref{lemma:low-rank-vector} (below)]
Let $m \in \mathscr{M}$, $C=\pi^{-1}(m)$, $C^\circ = (\pi^{\circ})^{-1}(m)$
and $D := C - C^\circ$.
Let $(E, \nabla)$ denote the
Deligne canonical extension
to $C$ of the vector bundle $\mathbb V|_{C^\circ} \otimes \mathscr O_{C^\circ}$ with its tautological connection; $(E,\nabla)$ is a flat vector bundle on $C$ with regular singularities along $D$. The residues of $\nabla$ endow $E$ with the structure of a parabolic bundle $E_\star$, as described in \cite[Definition 3.3.1]{LL:geometric-local-systems}.
By the Mehta-Seshadri correspondence \cite{mehta1980moduli},
the parabolic bundle $E_\star$ is a sum of parabolically stable bundles, hence
parabolically semistable. 

The connection $\nabla$ on $\mathscr{H} := R^1 \pi^\circ_*\mathbb V \otimes_{\mathbb C}
\mathscr O_{\mathscr M}$ induces an $\mathscr O_{\mathscr M}$-linear map
\begin{align*}
	F^1\mathscr{H} \xrightarrow{\overline{\nabla}} \mathscr{H}/F^1 \mathscr{H} \otimes \Omega_{\mathscr M}.
\end{align*}
as in \eqref{equation:shift-derivative}.
\autoref{theorem:period-map-computation} identifies the fiber of this map at the point $m \in \mathscr M$ with the map
\begin{align*}
	\mu: H^0(C, E \otimes \omega_C(D)) &\to \Hom(H^0(C, E^\vee \otimes \omega_C),
	H^0(C, \omega_C^{\otimes	2}(D))) \\
v&\mapsto (u\mapsto B_E(v, u)) \end{align*}
defined via $B_E$ from \eqref{equation:bilinear-pairing}. For $v\in H^0(C, E\otimes \omega_C(D))$, we denote by $\overline{\nabla}_m(v)$ the induced map 
\begin{align*}
 \overline{\nabla}_m(v): H^0(C, E^\vee \otimes \omega_C)& \rightarrow H^0(C, \omega_C^{\otimes	2}(D))\\
u & \mapsto B_E(v, u).
\end{align*}

Our strategy is to observe that a sub-local system of $R^1\pi_*^\circ \mathbb{V}$ of low rank would provide us with a vector $v$ in $F^1\mathscr{H}_m$ so that the linear transformation $\overline{\nabla}_m(v)$ has low rank. This is in tension with \autoref{proposition:pairing-result} above, which will provide us with the required dimension bound.

Now, suppose we have a nonzero irreducible sub-local system $\mathbb L \subset R^1 \pi^\circ_* \mathbb
V$.
We will show in \autoref{lemma:low-rank-vector} that, after possibly replacing $\mathbb{V}$ with its complex conjugate $\overline{\mathbb{V}}$, there is some nonzero $v \in
F^1\mathscr{H}_m \simeq
H^0(C, E \otimes \omega_C(D))$ for which
$\overline{\nabla}_m(v)$ has rank (as a linear map) at most $\frac{\rk \mathbb L}{2}$. 
Under the identification of
\autoref{theorem:period-map-computation},
we obtain that $\mu(v)$ has rank at most $\frac{\rk \mathbb L}{2}$. 
Using \autoref{proposition:pairing-result}, $\rk \mathbb{V}=\rk
E \geq g-\frac{\rk \mathbb L}{2}$. Upon rearranging, we find $\rk\mathbb{L}\geq 2g-2\rk\mathbb{V}$.
\end{proof}

We now complete the proof of \autoref{theorem:rank-bound-subsystem} by verifying
the following lemma, which is essentially an application of the theorem of the
fixed part.
\begin{lemma}\label{lemma:low-rank-vector}
Suppose $\pi^\circ:\mathscr{C}^\circ\to \mathscr{M}$ is a punctured versal family of genus
$g$ curves and $\mathbb{V}$ is a unitary local system on $\mathscr{C}^\circ$. 
For $m \in \mathscr M$, let $\overline{\nabla}_m$ denote the fiber over $m$ of
the map $F^1\mathscr{H} \xrightarrow{\overline{\nabla}} \mathscr{H}/F^1
\mathscr{H} \otimes \Omega_{\mathscr M}$
from \eqref{equation:shift-derivative}.
Suppose 
$\mathbb L \subset R^1 \pi^\circ_* \mathbb V$
is a nonzero irreducible sub-local system.
After possibly replacing $\mathbb{V}$ and $\mathbb L$ with their complex
conjugates
$\overline{\mathbb{V}}$ and $\overline{\mathbb L}$, there is some nonzero $v \in \mathbb L_m \cap F^1
\mathscr H_m$ so that
$\overline{\nabla}_m(v)$ has rank at most $\frac{\rk \mathbb L}{2}$.
\end{lemma}
\begin{proof}
For a local system $\mathbb W$, let $\widetilde{\mathbb W}$ denote the
corresponding local system with real structure as in
\eqref{equation:real-construction}. Note that $$\widetilde{R^1 \pi^\circ_* \mathbb{V}}=R^1 \pi^\circ_* \widetilde{\mathbb{V}}.$$

By \autoref{proposition:subrep-real-variation}, $\widetilde{\mathbb{L}}$
underlies a real polarizable variation of Hodge structure, and we have a
nonzero real mixed Hodge structure $Q_{\mathbb R}$ and a nonzero map of real
variations of mixed Hodge structures
$$\iota: Q_{\mathbb R} \otimes \widetilde{\mathbb L} \to \widetilde{R^1 \pi^\circ_* \mathbb V}.$$  

Let $Q := (Q_\mathbb R) \otimes_{\mathbb R} \mathbb C$.
As the grading on $(\widetilde{R^1 \pi^\circ_* \mathbb V})_{\mathbb C}$ is supported in degrees $(1,0), (0,1), (1,1)$, we may assume (after regrading and possibly replacing $Q$ with a subspace) that either
\begin{enumerate}
\item the bigrading on $\mathbb{L}$ is supported in degree $(0,0)$ and the bigrading on $Q$ has support contained in  $\{(1,0), (0,1), (1,1)\}$, or
\item the bigrading on $\mathbb{L}$ has support equal to $\{(1,0),(0,1)\}$
	and the bigrading on $Q$ is supported in degree $(0,0)$.
\end{enumerate}
In the two cases above, we choose $v$ as follows.

\emph{Case (1):} After possibly replacing ${\mathbb{V}}, \mathbb{L}$ with their complex conjugates, we may assume that  $$\iota_{i,j}: Q^{i,j}\otimes\mathbb{L}\to R^1\pi_*^\circ \mathbb{V}$$ is nonzero for some $(i,j)\in \{(1,0), (1,1)\}$. Choose $\ell\in \mathbb{L}_m$ and $q\in Q^{i,j}$, $(i,j)\in\{(1,0), (1,1)\}$ such that $v:=\iota_m(q\otimes\ell)$ is nonzero.

\emph{Case (2):} After possibly replacing $\mathbb{V}$ and $\mathbb{L}$ with
their complex conjugates, we may assume $\dim\mathbb{L}_m^{0,1}\leq \dim \mathbb{L}_m^{1,0}$. 
Choose arbitrary nonzero $q\in Q$ and $\ell\in \mathbb{L}_m^{1,0}$ such that $v:=\iota_m(q\otimes\ell)$ is nonzero.

It is enough to argue that for $v$ as above, $\overline{\nabla}(v)$ has rank at
most $\frac{\rk\mathbb{L}}{2}$. Since $\mathbb{L}$ carries a complex PVHS, we
may compute the rank of $\overline{\nabla}(v)$ by computing with the analogous
period map 
\begin{align*}
\overline{\nabla}_{\mathbb L,m}: F^1\mathbb{L}_m\to \mathbb{L}_m/F^1\mathbb{L}_m\otimes \Omega^1_{\mathscr{M},m}
\end{align*}
associated to $\mathbb{L}$. In case (1), this map is identically
zero, so we are done (i.e.~$\rk \overline{\nabla}(v)=0$). In case (2), 
from the definition of $\overline{\nabla}_{\mathbb L,m}$,
it
is enough to show
\begin{align*}
	\dim \mathbb{L}_m/F^1\mathbb{L}_m\leq \frac{\rk \mathbb L}{2}.
\end{align*}
This follows because we have 
$\dim \mathbb{L}_m=\dim F^1 \mathbb{L}_m+ \dim \mathbb{L}_m/F^1 \mathbb L_m$ and
we assumed
\begin{equation*}
	\dim F^1 \mathbb L_m = \dim \mathbb L^{1,0}_m \geq \dim \mathbb
	L^{0,1}_m = \dim \mathbb L_m/F^1 \mathbb L_m.\qedhere
\end{equation*}
\end{proof}

\subsection{A vanishing result}
\label{subsection:no-sections}

We apply \autoref{theorem:rank-bound-subsystem} to prove a vanishing result about local systems on $\mathscr{C}^\circ$ that are not necessarily unitary. We assume only that their restriction to each fiber of $\pi^\circ$ is unitary. The reader may find it useful to consider the case where $A=\mathbb{C}$ below, though we will need the full generality of the result below for our deformation-theoretic applications.
\begin{theorem}
	\label{theorem:unitary-rigid}
	With notation as in \autoref{notation:versal-family}, 
	let $\pi^\circ: \mathscr{C}^\circ\to \mathscr{M}$ be a punctured versal
	family of genus $g$ curves. Let $m\in \mathscr{M}$ be a point and set $C^\circ
	= \mathscr C^\circ_m$. Let $A$ be a local Artin $\mathbb{C}$-algebra with residue field $\mathbb{C}$, and let $\mathbb V$ be a locally constant sheaf of free $A$-modules on $\mathscr C^\circ$
	such that
	\begin{enumerate}
 \item $\mathbb V|_{C^\circ}$ is a constant deformation of a unitary local system, i.e. there exists a unitary complex local system $\mathbb{V}_0$ on $C^\circ$ and an isomorphism $\mathbb V|_{C^\circ}\simeq \mathbb{V}_0\otimes A$, and
\item $\rk_A \mathbb V < g$.
		\end{enumerate}
 Then
	$H^0(\mathscr M, R^1 \pi^\circ_* \mathbb V) = 0$.
\end{theorem}
\begin{proof}
	Given a dominant map $\mathscr M' \to \mathscr M$,
	we have an injection $$H^0(\mathscr M, R^1 \pi^\circ_* \mathbb
	V) \to H^0(\mathscr M', R^1 \pi^\circ_* \mathbb V|_{\mathscr M'}).$$
	Hence, it is enough to prove the claim after replacing $\mathscr M$ by
	some $\mathscr M'$ with a dominant map to  $\mathscr M$.
	Combining this with \autoref{lemma:unitary-direct-sum}, we may assume $\mathbb V \simeq \oplus_{i=1}^s \mathbb U_i
\otimes (\pi^\circ)^* \mathbb
	W_i$, where $\mathbb W_i$ are locally constant sheaves of free $A$-modules on $\mathscr M$ and
	$\mathbb U_i$ are unitary complex local systems on $\mathscr C^\circ$. 
	It is enough to show that 
	$$H^0(\mathscr{M}, R^1\pi^\circ_*(\mathbb{U}_i\otimes
	(\pi^\circ)^*\mathbb{W}_i))=H^0(\mathscr{M},
	(R^1\pi^\circ_*\mathbb{U}_i)
	\otimes \mathbb{W}_i )=0$$ for each $i$. 
	Letting $\mathfrak{m}_A$ be the maximal ideal of $A$, the $\mathfrak{m}_A$-adic filtration on $\mathbb{W}_i$ has associated-graded pieces isomorphic to direct sums of copies of $\mathbb{W}_i^0:=\mathbb{W}_i\otimes_A\mathbb{C}.$ 
	Hence $\mathbb W_i\otimes R^1\pi_*\mathbb{U}_i$ has a filtration with associated-graded pieces isomorphic to $\mathbb{W}_i^0\otimes R^1\pi_*\mathbb{U}_i$. 
	Thus it suffices to prove  $$H^0(\mathscr{M},
	R^1\pi^\circ_*(\mathbb{U}_i\otimes
	(\pi^\circ)^*\mathbb{W}^0_i))=H^0(\mathscr{M}, (R^1\pi^\circ_*\mathbb{U}_i) \otimes \mathbb W^0_i)=0.$$
	 Note that $\on{rk} \mathbb{U}_i\otimes (\pi^\circ)^*\mathbb{W}^0_i<g$.

	It is therefore enough to show that if $\mathbb{U}$ is a unitary complex local system on
	$\mathscr C^\circ$ and $\mathbb{W}$ is an arbitrary complex local system on $\mathscr M$
	with
\begin{align*}
	H^0(\mathscr{M}, R^1\pi^\circ_* (\mathbb{U}\otimes (\pi^\circ)^*\mathbb{W}))=
	H^0(\mathscr M, R^1 \pi^\circ_* \mathbb U \otimes \mathbb W) \neq 0,
\end{align*}
	then $\rk \mathbb U \otimes (\pi^\circ)^*\mathbb W \geq g$.
A nonzero element of $H^0(\mathscr M, R^1 \pi^\circ_* \mathbb U \otimes \mathbb W)$ yields a nonzero map of local systems $$\mathbb{W}^\vee\to R^1\pi^\circ_*\mathbb{U}.$$
Since $\mathbb U$ is unitary, we have
$\rk \mathbb{W} \geq 2g - 2\rk \mathbb U$
by \autoref{theorem:rank-bound-subsystem}. 
Hence $\rk \mathbb{W} + 2\rk \mathbb{U} \geq 2g$.
As $\rk \mathbb{W}, \rk\mathbb{U}$ are positive integers, we find 
$\rk \mathbb W \cdot \rk \mathbb U$ is minimized when $\rk \mathbb W = 1$
in which case $\rk \mathbb W \cdot \rk \mathbb U \geq \lceil (2g-1)/2 \rceil \geq g$. 
Note here we use that $g \geq 2$ in order for condition $(2)$ of the statement to be
satisfied as $\mathbb V \neq 0$.
In general, this implies
\begin{align}
	\label{equation:rank-product-bound}
\rk (\mathbb{U}\otimes (\pi^\circ)^*\mathbb{W})=\rk \mathbb W \cdot \rk \mathbb U \geq g,
\end{align}
as desired.
\end{proof}

\section{The asymptotic Putman-Wieland conjecture}
\label{section:Putman-Wieland}
\subsection{The statement of the Putman-Wieland Conjecture}
\label{subsection:geometric-setup}

We next discuss some applications of our methods to the Putman-Wieland
conjecture (\autoref{conjecture:putman-wieland}), a major open problem in geometric topology. Much of the interest in the Putman-Wieland conjecture arises from its close relationship to Ivanov's conjecture that mapping class groups $\on{Mod}_{g,n}$ do not virtually surject onto $\mathbb{Z}$ for $g\gg 0$ \cite[\S7]{Ivanov:problems}, as explained in
\cite[Theorem 1.3]{putmanW:abelian-quotients}. In particular, the Putman-Wieland conjecture for $g\gg 0$ is equivalent to Ivanov's conjecture for $g\gg 0$.
\begin{notation}
	\label{notation:prym}
Let $\Sigma_g$ be an oriented surface of genus $g \geq 0$, and let
$\Sigma_{g,b,p}$ be the complement of $b\geq 0$ disjoint open discs and $p\geq 0$ disjoint points in $\Sigma_g$, so that $\Sigma_{g,b,p}$ is an oriented genus $g$ surface with $b$ boundary components and $p$ punctures. We let $\on{PMod}_{g,b,p}$ be the subgroup of the mapping class group of $\Sigma_{g,b,p}$ fixing the punctures and boundary components pointwise.
Fix a basepoint $v_0 \in \Sigma_{g,b,p}$, which we count as an additional puncture to
obtain an action of $\on{PMod}_{g,b,p+1}$ on $\pi_1(\Sigma_{g,b,p},v_0)$.
Let $K \trianglelefteq \pi_1(\Sigma_{g,b,p}, v_0)$ be a finite index normal subgroup
corresponding to a finite Galois covering space $\Sigma_{g',b',p'} \to \Sigma_{g,b,p}$,
or equivalently a surjection
$\phi: \pi_1(\Sigma_{g,b,p}, v_0) \twoheadrightarrow H$, with $H=\pi_1(\Sigma_{g,b,p}, v_0)/K$ a finite group.
Let $\Gamma \subset \on{PMod}_{g,b,p+1}$ denote the finite index subgroup preserving
$\phi$ up to conjugacy.
Viewing
$H_1(\Sigma_{g',b',p'},\mathbb C)$ as an $H$-representation,
if $\rho$ is an irreducible $H$-representation we let 
\begin{align*}
	H_1(\Sigma_{g',b',p'},\mathbb C)^\rho := \rho \otimes \Hom_H(\rho,
H_1(\Sigma_{g',b',p'},\mathbb C))
\end{align*}
denote the $\rho$-isotypic component.
We obtain an action of $\Gamma$ on $K^{\text{ab}}\otimes \mathbb{C}=H_1(\Sigma_{g',b',p'},\mathbb C)$
and hence on the characteristic subrepresentation
$H^1(\Sigma_{g',b',p'},\mathbb C)^\rho$.
By filling in the punctures and deleted discs, we also obtain
an action of $\Gamma$ on $V_K := H_1(\Sigma_{g'}, \mathbb C)$, referred to in 
\cite[p. 80-81]{putmanW:abelian-quotients} as a {\em higher Prym
representation}.
\end{notation}

\begin{definition}
	\label{definition:pw}
	Fix a finite group $H$ and nonnegative integers $g, b,$ and $p$.
	Let $\PW_{g,b,p}^H$ be the statement that for any
	surjection $\phi: \pi_1(\Sigma_{g,b,p}, v_0) \twoheadrightarrow
	H$, taking $K := \ker \phi$ and $v \in V_K$ any nonzero vector,
	$v$ has infinite orbit under $\Gamma$.
\end{definition}

\begin{conjecture}[Putman-Wieland, ~\protect{\cite[Conjecture
	1.2]{putmanW:abelian-quotients}}]
	\label{conjecture:putman-wieland}
	$\PW_{g,b,p}^H$ holds for every group $H$ with $g \geq 2, b \geq 0, p \geq 0$.
\end{conjecture}

\begin{remark}
\label{remark:putman-wieland-differences}
We note that 
\cite[Conjecture 1.2]{putmanW:abelian-quotients}
is stated without a group $H$ and with $\mathbb Q$ coefficients instead of
$\mathbb C$ coefficients.
We use this equivalent statement to more easily state our results,
which imply that 
$\PW_{g,b,p}^H$ holds whenever $\# H$ is small compared to the genus $g$. 
\end{remark}
\begin{remark}
	\label{remark:}
	There is a counterexample to the Putman-Wieland conjecture when $g = 2$ \cite[Theorem 1.3]{markovic}.
\end{remark}

\subsection{Our results}

We are able to verify $\PW_{g,b,p}^H$ in many new cases.
Moreover, we prove the following more general statement, ruling out
not only invariant vectors in $H^1(\Sigma_{g'}, \mathbb C)$, but even ruling out low-dimensional invariant subspaces
of isotypic components.
The proof is given below in \autoref{subsection:proof-putman-wieland}.
\begin{theorem}
	\label{theorem:asymptotic-putman-wieland}
	With notation as in \autoref{notation:prym}, let $\rho$ be an irreducible
	complex $H$-representation and $\Gamma' \subset \Gamma$ be a finite-index subgroup. Then
	$H^1(\Sigma_{g',b',p'},\mathbb C)^\rho$
	has no nonzero $\Gamma'$-invariant subrepresentations 
	of dimension strictly less than $2g - 2\dim \rho$.
	The same holds for
	$H_1(\Sigma_{g'},\mathbb C)^\rho$ in place of 
	$H^1(\Sigma_{g',b',p'},\mathbb C)^\rho$.
\end{theorem}

\begin{remark}
In \autoref{theorem:asymptotic-putman-wieland}, it is important that we restrict our attention to subrepresentations, as opposed to arbitrary subquotients, as the $\Gamma$-representation 	$H^1(\Sigma_{g',b',p'},\mathbb C)$ is not in general semisimple for $b'+p'>0$.
\end{remark}

A perhaps more concrete corollary is the following.
\begin{corollary}
	\label{corollary:rep-theory-PW}
	$\PW_{g,b,p}^H$ holds for every group $H$ 
	such that every irreducible representation $\rho$ of $H$ has dimension
	$\dim \rho < g$.
\end{corollary}
\begin{proof}
	Fix an irreducible representation $\rho$ of $H$. By assumption, $\dim \rho < g$. Then
	$2g - 2\dim \rho > 1$, and so by
	\autoref{theorem:asymptotic-putman-wieland},
	$H_1(\Sigma_{g'},\mathbb C)^\rho$ has no nonzero $\Gamma'$-invariant $1$-dimensional subrepresentations for any $\rho$ and any finite index $\Gamma'\subset\Gamma$.
	In particular, $H_1(\Sigma_{g'},\mathbb C)$ has no nonzero invariant
	vectors under any finite index $\Gamma'$. But if there was a nonzero vector $v\in V_K$ with finite orbit, its stabilizer would yield a finite index $\Gamma'$ with a fixed vector.
\end{proof}
Even more concretely, we can give the following bound on $\#H$ independent of its
representation theory.
\begin{corollary}
	\label{corollary:asymptotic-putman-wieland}
	For fixed $g, b, p \geq 0$,
	$\PW_{g,b,p}^H$ holds for every group $H$ with $\# H < g^2$.
\end{corollary}
\begin{proof}
	If $H$ is a finite group of order $\# H < g^2$, then because $\#
	H$ is the sum of the squares of the dimensions of the irreducible representations of
	$H$ by \cite[(2.19)]{fultonH:representation-theory},
	no irreducible representations of $H$ have
	dimension $\geq g$. The result follows
	from \autoref{corollary:rep-theory-PW}.
\end{proof}

In order to prove \autoref{theorem:asymptotic-putman-wieland},
we first note that it suffices to consider the case of surfaces without boundary.
\begin{lemma}
	\label{lemma:ignore-boundary}
	With notation as in \autoref{notation:prym},
	for any $g \geq 0, b, p \geq 0$, 
	and $\rho$ an irreducible $H$-representation, let $\Gamma'\subset
	\Gamma$ be a finite index subgroup. Then there exists a finite index
	$\Gamma''\subset \on{PMod}_{g,0,b+p+1}$ and 
 an isomorphism	
$H_1(\Sigma_{g',b',p'}, \mathbb C)^\rho \to H_1(\Sigma_{g,0,b'+p'}, \mathbb
C)^\rho$
inducing a bijection from $\Gamma'$-invariant subspaces to
$\Gamma''$-invariant subspaces.
\end{lemma}
\begin{proof}
%
Define the natural inclusion $\iota: \Sigma_{g',b',p'} \hookrightarrow
\Sigma_{g,0,b'+p'}$,
shrinking boundary components to punctures.
There is a deformation retraction from $\Sigma_{g,0,b'+p'}$ onto the subspace
$\iota(\Sigma_{g',b',p'})$
and so $\iota$ induces an isomorphism
$H_1(\Sigma_{g',b',p'}, \mathbb C) \to H_1(\Sigma_{g,0,b'+p'}, \mathbb C).$
The kernel of the induced surjection $\on{PMod}_{g, b,p} \to \on{PMod}_{g,0,b+p}$
is generated by $\mathbb Z^b$ acting on $S_{g,b,p}$ by rotating the
boundary components,
\cite[Theorem 3.18]{farbM:a-primer}
and hence this $\mathbb Z^b$ acts trivially on 
$H_1(\Sigma_{g',b',p'}, \mathbb C)$. 
Let $\Gamma''$ be the image of $\Gamma'$ under this natural map $\on{PMod}_{g,
b,p+1} \to \on{PMod}_{g,0,b+p+1}$.
The action of $\Gamma'$ on
$H_1(\Sigma_{g',0,b'+p'}, \mathbb{C})$ factors through the natural map $\Gamma'\to \Gamma''$, inducing the claimed bijection.
\end{proof}
\subsection{}
	\label{subsection:proof-putman-wieland}

We now prove 
\autoref{theorem:asymptotic-putman-wieland}, which
is more or less an immediate consequence of \autoref{theorem:rank-bound-subsystem}
applied to the irreducible representations of $H$,
once one sets up the appropriate local systems.
For the proof, we will need the existence of a versal family of $\phi$-covers,
for $\phi: \pi_1(\Sigma_{g, 0, n}, v_0)\twoheadrightarrow H$,
which we now define.
\begin{definition}
	\label{definition:}
	As in \autoref{notation:prym}, specify a surjection
	$\phi: \pi_1(\Sigma_{g, 0, n}, v_0)\twoheadrightarrow H$.
	A {\em versal family of $\phi$-covers} is the data of 
	\begin{enumerate}
\item a dominant \'etale morphism $\mathscr{M}\to \mathscr{M}_{g,n}$, with $\pi^\circ: \mathscr{C}^\circ\to \mathscr{M}$ the associated family of punctured curves,
\item a point $c\in \mathscr{C}^\circ$, $m=\pi^\circ(c)$, $C^\circ = \mathscr
C^\circ_m$, and an identification $i: \pi_1(\Sigma_{g, 0, n}, v_0)\simeq
\pi_1(C^\circ, c)$, and
\item a finite \'etale Galois $H$-cover $f: \mathscr{X}^\circ\to
	\mathscr{C}^\circ$ inducing a map
	$$\pi_1(C^\circ, c) \to \pi_1(\mathscr{C}^\circ,
	c)/\pi_1(\mathscr{X}^\circ,x) \simeq H$$
	agreeing with the surjection $\phi$ under the
	identification of $(2)$.
\end{enumerate}
\end{definition}

\begin{proof}[Proof of \autoref{theorem:asymptotic-putman-wieland}]
	Using \autoref{lemma:ignore-boundary}, 
	it is enough to prove
	\autoref{theorem:asymptotic-putman-wieland}
	when $b =0$. 

	By \cite[Theorem 4]{wewers:thesis}, there exists a versal family of
	$\phi$-covers $\mathscr X^\circ \xrightarrow{f} \mathscr C^\circ
	\xrightarrow{\pi^\circ} \mathscr M$.
	Technically, 
	\cite[Theorem 4]{wewers:thesis} constructs $\mathscr M$ as a Deligne-Mumford stack,
	where points have isotropy groups isomorphic to the center of $H$,	
	but we may replace this stack by
	a scheme which has a dominant \'etale map to this stack. 
	Further, after replacing $\mathscr M$ with another scheme that has a dominant \'etale map
	to $\mathscr M$, corresponding generically to an \'etale multisection of $\pi^\circ$, we may assume $\pi^\circ$ admits a section, and hence that the map $\mathscr{M}\to \mathscr{M}_{g,n}$ lifts to a map $\mathscr{M}\to \mathscr{M}_{g,n+1}$.

	As $f$ is a Galois finite \'etale $H$-cover, for each irreducible
	representation $\rho$ of $H$, we obtain a local system $\mathbb{U}_\rho$
	on $\mathscr{C}^\circ$, with monodromy representation given by $\rho$,
	and with $\mathbb{U}_\rho|_{\mathscr{X}^\circ}$ trivial. After replacing $\mathscr{M}$ with a dominant \'etale cover, we may write
	$$f_*\mathbb{C}=\bigoplus_\rho \mathbb{U}_\rho\otimes
	(\pi^\circ)^*\mathbb{W}_\rho,$$ where
	$\mathbb{W}_\rho:=\pi^\circ_*\on{Hom}(\mathbb{U}_\rho, f_*\mathbb{C})$,
	by \autoref{lemma:unitary-direct-sum} (where we take $A=\mathbb{C}$). 
	As
	$\mathbb{U}_\rho$ has finite monodromy by
	\autoref{lemma:finite-rho-implies-finite-lift}(1), and $f_*\mathbb{C}$
	has finite monodromy by definition, the same is true for
	$\mathbb{W}_\rho$. Thus after replacing $\mathscr{M}$ with a finite \'etale cover we may assume each
	$\mathbb{W}_\rho$ is a trivial local system.
	
	We next verify the first part of
	\autoref{theorem:asymptotic-putman-wieland}, about $H^1(\Sigma_{g', 0,
	p'}, \mathbb{Q})^\rho$.
	Observe that for $m\in \mathscr{M}$, the action of $\pi_1(\mathscr{M},
	m)$ on $$R^1\pi^\circ_*( \mathbb{U}_\rho\otimes
	(\pi^\circ)^*\mathbb{W}_\rho)_m=(R^1\pi^\circ_* \mathbb{U}_\rho\otimes
	\mathbb{W}_\rho)_m$$ is identified with its action on
	$H^1(\mathscr{X}^\circ_m, \mathbb{C})^\rho=H^1(\Sigma_{g', 0, p'},
	\mathbb{C})^\rho$. 
	Note also that the action of
	$\pi_1(\mathscr{M})$ factors through $\Gamma$, 
	under our given map $\pi_1(\mathscr{M})\to \pi_1(\mathscr{M}_{g,n+1})=\on{PMod}_{g,n+1}$.
	To prove the first part of
	\autoref{theorem:asymptotic-putman-wieland}, about $H^1(\Sigma_{g', 0,
	p'}, \mathbb{Q})^\rho$, it thus suffices to show that, after replacing
	$\mathscr{M}$ with an arbitrary finite \'etale cover, $R^1(\pi^\circ_*
	\mathbb{U}_\rho)\otimes \mathbb{W}_\rho$ contains no non-trivial
	sub-local systems of rank less than $2g-2\dim\rho$. 
	As $\mathbb{W}_\rho$
	is a trivial local system, this holds by applying
	\autoref{theorem:rank-bound-subsystem} to $\mathbb{U}_\rho$.

We have  thus far proven a statement for the cohomology of $\Sigma_{g',n'}$, and to
conclude we deduce a corresponding statement for the homology of $\Sigma_{g'}$.
Using the inclusion 
$H^1(\Sigma_{g'}, \mathbb C)^\rho \subset H^1(\Sigma_{g', n'}, \mathbb{C})^\rho$
we deduce that the former can have no nonzero $\Gamma'$-invariant subrepresentations 
(for $\Gamma' \subset \Gamma$ finite index)
of dimension less than $2g - 2\dim \rho$, as the latter has no such 
subrepresentations.
Using Poincar\'e duality and the intersection pairing on $H_1(\Sigma_{g'}, \mathbb C)$ 
we obtain a $\on{Mod}_{g,n}$ equivariant isomorphism
$H^1(\Sigma_{g'},\mathbb C) \simeq H_1(\Sigma_{g'},\mathbb C)$.
Hence, $H_1(\Sigma_{g'}, \mathbb C)^\rho$ also has no subrepresentations of 
dimension less than $2g - 2\dim\rho$.
\end{proof}

\section{Proof of the main theorem on MCG-finite representations}
\label{section:main-proof}

We have finally developed the tools to verify our main result, \autoref{theorem:finite-image}: that  MCG-finite representations of rank $r<\sqrt{g+1}$ have finite image.
The proof is fairly involved, so the reader may find it useful to refer to the sketch in \autoref{subsection:outline-of-proof}. Briefly, we 
start by reviewing the notion of cohomological rigidity in
\autoref{subsection:background-rigidity}
and prove the necessary rigidity results in
\autoref{subsection:rigidity}, using our main vanishing result,
\autoref{theorem:unitary-rigid}.
We next deduce the integrality of MCG-finite representations of low rank in
\autoref{subsection:integrality}
using a result of 
Klevdal-Patrikis \cite{klevdalP:g-rigid-local-systems-are-integral}, which builds on work of Esnault-Groechenig \cite{esnault2018cohomologically} and ultimately relies on input from the Langlands program and known cases of the companion conjectures.
We then combine the above ingredients with input from non-abelian Hodge theory to deduce that low rank unitary MCG-finite representations
have finite image in \autoref{subsection:finiteness}.
We generalize this to the semisimple but non-unitary case in
\autoref{subsection:semisimple-case}, again using non-abelian Hodge theory.
Finally, in \autoref{subsection:non-semisimple} we bootstrap these results to the general, non-semisimple case,
using \autoref{theorem:asymptotic-putman-wieland} on the Putman-Wieland conjecture.

\subsection{Background on Cohomological rigidity}
\label{subsection:background-rigidity}

We now define the notion of cohomological rigidity, which detects whether a
representation has any first order deformations.

\begin{definition}[Cohomological rigidity]\label{definition:rigidity}
    Let $\overline{X}$ be a smooth projective variety, $Z\subset \overline{X}$ a
    strict normal crossings divisor, and $X=\overline{X}\setminus Z$. Let $G$ be
    a reductive group over $\mathbb{C}$ and $$\rho: \pi_1(X)\to G(\mathbb{C})$$
    a homomorphism such that the monodromy at infinity is quasi-unipotent.   
     The representation $\rho$ is said to be \emph{cohomologically rigid} if
    $$H^1(\overline{X}, j_{!*}\on{ad}\rho)=0,$$ where $j: X\hookrightarrow
    \overline{X}$ is the natural inclusion, $j_{!*}$ is the intermediate
    extension, and $\on{ad}\rho$ is defined as in \autoref{notation:adjoint}.
    
    If $\mathbb{V}$ is the local system on $X$ associated to $\rho$, we call $\mathbb{V}$ cohomologically rigid if $\rho$ is.
\end{definition}
\begin{remark}
	\label{remark:}
	In the case $\rho$ is irreducible, see \cite[Proposition 4.7, Remark
    4.8]{klevdalP:g-rigid-local-systems-are-integral} for a moduli-theoretic
    interpretation of cohomological rigidity. 
    This says that $\rho$ is a smooth isolated point of the appropriate character variety, parametrizing representations with {fixed} local monodromy at infinity. Note that if $\overline{X}$ is a curve, $j_{!*}=j_*.$
\end{remark}

We will typically prove cohomological rigidity by computing 
$H^1(X, \on{ad}\rho)$, via the next lemma.
\begin{lemma}
	\label{lemma:vanishing-h1-implies-rigidity}
		In the setting of \autoref{definition:rigidity}, assume $H^1(X, \on{ad}\rho)
	= 0$. Then $\rho$ is
	cohomologically rigid.
\end{lemma}
\begin{proof}
	This follows from the identification 
	$H^1(\overline{X}, j_{!*}\on{ad}\rho)\simeq H^1(U, a_*\on{ad}\rho)$ for
	$U \subset \overline{X}$ a certain dense open subscheme containing $X$
	with $a: X \to U$ the inclusion \cite[Remark
	2.4]{esnault2018cohomologically} (see also \cite[Remark
	4.8]{klevdalP:g-rigid-local-systems-are-integral}).
	The Leray spectral sequence gives an injection 
	$H^1(U, a_*\on{ad}\rho) \hookrightarrow H^1(X, \on{ad}\rho)$, implying
	the claim.
\end{proof}
\begin{definition}
If $H^1(X, \on{ad}\rho)=0$, we say that $\rho$ is \emph{strongly cohomologically rigid}.	
\end{definition}
\subsection{Cohomological rigidity and unitary MCG-finite representations}
\label{subsection:rigidity}

Let $\rho: \pi_1(\Sigma_{g,n})\to \on{GL}_r(\mathbb{C})$ be a unitary MCG-finite
representation of rank $<\sqrt{g+1}$. We will study $\rho$ by associating to $\rho$ a certain unitary local system on a well-chosen versal family of curves. Crucially, this unitary local system will be cohomologically rigid.
This will follow from our main vanishing result
\autoref{theorem:unitary-rigid}.

Throughout this section, we use notation as in 
\autoref{notation:versal-family}. In particular, $\pi:\mathscr{C}\to
\mathscr{M}$ is a versal family of $n$-pointed curves of genus $g$,
$\pi^\circ: \mathscr{C}^\circ \to \mathscr{M}$ is the associated family of
punctured curves, $m \in \mathscr M$ is a basepoint, and $C^\circ = \mathscr
C_m^\circ$.

\begin{proposition}
	\label{proposition:cohomologically-rigid}
	Let $\mathbb{V}$ be a $\on{GL}_r$ (respectively, ~$\on{PGL}_r$)-local system on
$\mathscr{C}^\circ$ with $r<\sqrt{g+1}$. Suppose that for $m\in \mathscr{M}$,
$\mathbb{V}|_{{C}^\circ}$ is (respectively, is the projectivization
of) an irreducible,
unitary local system. Then $\mathbb{V}$ is strongly cohomologically rigid. In
particular, $\mathbb V$ is cohomologically rigid.
\end{proposition}
\begin{proof}
	By \autoref{lemma:vanishing-h1-implies-rigidity}, it is enough to show
	$\mathbb V$ is strongly cohomologically rigid.
	We let $\on{ad}\mathbb V$ denote the adjoint local system as in
	\autoref{notation:adjoint}.
	We will check 
	$H^1(\mathscr C^\circ, \on{ad} \mathbb V) = 0$.
	Using the Leray spectral sequence associated to the map $\pi^\circ$, it suffices to show that  
	$$H^0(\mathscr M, R^1 \pi^\circ_* \on{ad} \mathbb V) = 
	H^1(\mathscr M, \pi^\circ_* \on{ad} \mathbb V) = 0.$$
	We have 
	$H^0(\mathscr M, R^1 \pi^\circ_* \on{ad} \mathbb V) = 0$
	by \autoref{theorem:unitary-rigid} (taking $A=\mathbb{C}$), as $\rk \on{ad}\mathbb{V}=r^2-1<g$ and $\on{ad}\mathbb{V}|_{{C}^\circ}$ is unitary by assumption.

	To conclude, it is enough to show $\pi^\circ_* \on{ad} \mathbb V = 0$.
	This is follows from Schur's lemma because
	$\mathbb{V}|_{{C}^\circ}$ is (the projectivization of) an
	irreducible local system, and so $\on{ad}\mathbb{V}|_{{C}^\circ}$ has no $\pi_1({C}^\circ)$-invariants.
\end{proof}
\subsection{Verifying integrality}
\label{subsection:integrality}

The next step of the proof is to use the cohomological rigidity provided by
\autoref{proposition:cohomologically-rigid} to deduce integrality of MCG-finite unitary representations.

\begin{definition}
	\label{definition:}
	Let $G$ be a group scheme over $\mathbb{Z}$. We say a representation $\pi_1(X, x) \to G(\mathbb C)$ is {\em integral}
if it is conjugate to a representation which factors through $G(\mathscr
O_K)$ for some number field $K$.
A local system on a connected space $X$ is integral if the corresponding monodromy representation is.
\end{definition}
We will primarily be concerned with integrality with respect to the group
schemes $G = \on{PGL}_r$ and $G = \on{GL}_r$.
The basic idea to show our local system is integral is to apply the main result of
\cite{klevdalP:g-rigid-local-systems-are-integral}, which builds on \cite{esnault2018cohomologically} and ultimately relies on Lafforgue's work on the Langlands program for function fields and consequent work on Deligne's companion conjectures.
To apply this result, we need to know the monodromy of our cohomologically rigid
local system $\mathbb{V}$ around boundary components of a strict normal crossings compactification of $\mathscr{C}^\circ$
is quasi-unipotent, which we will verify using the following result.

\begin{proposition}[\protect{~\cite[Proposition 2.4]{aramayona2016rigidity}}]
			\label{proposition:quasi-unipotent-dehn}
	Suppose that $g \geq 3$, $\gamma \subset \Sigma_{g,n}$ is a simple closed curve, 
	and $\Gamma\subset \on{Mod}_{g,n}$ is a 
finite index subgroup. Let $\delta_\gamma\in\on{Mod}_{g,n}$ be the Dehn twist about $\gamma$. Let $\rho: \Gamma\to \on{GL}_r(\mathbb{C})$ be a representation. Then for $m$ such that $\delta_\gamma^m\in \Gamma$, $\rho(\delta_\gamma^m)$ is quasi-unipotent.
\end{proposition}

We now verify that the projectivization of a unitary MCG-finite representation is integral.
\begin{lemma}
	\label{lemma:defined-over-number-field}
	Let $g\geq 3$ be an integer, and let $\rho: \pi_1(\Sigma_{g,n})\to \on{GL}_r(\mathbb{C})$ be an irreducible, unitary, MCG-finite representation, with $r<\sqrt{g+1}$. Then the composition $$\mathbb{P}\rho: \pi_1(\Sigma_{g,n})\overset{\rho}{\to} \on{GL}_r(\mathbb{C})\to \on{PGL}_r(\mathbb{C})$$ is integral.
\end{lemma}
\begin{proof}
By \autoref{lemma:geometric-mcg-rep-construction}, there exists a scheme $\mathscr{M}$ with a finite \'etale map $\mathscr{M}\to \mathscr{M}_{g,n}$, with associated family of punctured curves $\pi^\circ: \mathscr{C}^\circ\to \mathscr{M}$, and a representation $$\widetilde{\rho}: \pi_1(\mathscr{C}^\circ)\to \on{PGL}_r(\mathbb{C})$$ agreeing with $\mathbb{P}\rho$ on restriction to a fiber of $\pi^\circ$. It suffices to show that $\widetilde{\rho}$ is integral.  
	By \autoref{proposition:cohomologically-rigid},
	$\widetilde{\rho}$ is cohomologically rigid.
	Because the abelianization of $\on{PGL}_r(\mathbb C)$ is
	trivial, it follows from
\cite[Theorem 1.2]{klevdalP:g-rigid-local-systems-are-integral}	that $\widetilde\rho$
is integral once we verify that $\widetilde\rho$ has quasi-unipotent local monodromy around the boundary components of
	a good (i.e.~strict normal crossing) compactification of $\mathscr{C}^\circ$.

One may construct a strict normal crossing compactification as follows.
$\mathscr{C}^\circ$ is a finite \'etale cover of $\mathscr{M}_{g,n+1}$ by
construction; let $\overline{\mathscr{M}_{g, n+1}}'$ be a strict normal crossing
compactification of $\mathscr{M}_{g,n+1}$ obtained by blowing up boundary strata
of the Deligne-Mumford compactification $\overline{\mathscr{M}_{g,n+1}}$ (which
is only a normal crossing compactification, and is not in general strict). Let
$\overline{\mathscr{C}}$ be the normalization of $\overline{\mathscr{M}_{g,
n+1}}'$ in the function field of $\mathscr{C}^\circ$. By
\autoref{proposition:quasi-unipotent-dehn}, it suffices to check that the local
monodromy about the boundary components of $\overline{\mathscr{C}}$ correspond
to products of commuting Dehn twists about simple closed curves, under the
identification of $\pi_1(\mathscr{C}^\circ)$ with a subgroup of
$\on{PMod}_{g,n+1}=\pi_1(\mathscr{M}_{g,n+1})$. 
This is because a product of
commuting quasi-unipotent matrices is quasi-unipotent. To check the claim about
local monodromy about boundary components, it suffices to check the
corresponding claim for $\overline{\mathscr{M}_{g, n+1}}'$. But this is
\cite[Lemma 2.1.1]{li2020surface}.
\end{proof}

The next result shows we can lift integrality from the projectivization of a unitary irreducible MCG-finite representation to the original representation. 

\begin{lemma}
	\label{lemma:gl-defined-over-number-field}
	Let $g\geq 3$ be an integer, and let $\rho: \pi_1(\Sigma_{g,n})\to \on{GL}_r(\mathbb{C})$ be an irreducible, unitary, MCG-finite representation, with $r<\sqrt{g+1}$. Then $\rho$ is integral.
	\end{lemma}
\begin{proof}
By \autoref{proposition:1-dim-finite-image}, $\det\rho$ has finite image, and
hence $\rho$ factors through $G(\mathbb{C})$, where $G\subset \on{GL}_r$ is a
flat affine $\mathbb{Z}$-group scheme containing $\on{SL}_r$ with finite index (for
example, we can take $G=\det^{-1}(\mu_m)$, where $\det\rho$ factors through
$\mu_m$). Note that the natural morphism $G\to \on{PGL}_r$ is finite. After
conjugating $\rho$ by some matrix we may assume $\mathbb{P}\rho$ factors through
$\on{PGL}_r(\mathscr{O}_K)$ for some number field $K$, by \autoref{lemma:defined-over-number-field}.

As $\pi_1(\Sigma_{g,n})$ is finitely generated, it suffices to show that for a generator $\alpha\in \pi_1(\Sigma_{g,n})$, there exists some finite extension $K'/K$ such that the matrix $\rho(\alpha)$ lies in $G(\mathscr{O}_{K'})$. 
The obstruction to lifting any given $\mathscr O_K$ point of $\on{PGL}_r$ to
$G$ is a $\ker(G\to \on{PGL}_r)$ torsor, which becomes trivial on some finite
flat extension, and hence over $\mathscr O_{K'}$ for some finite $K'/K$.
\end{proof}

\subsection{Verifying finiteness}
\label{subsection:finiteness}
We now show that a unitary MCG-finite representation $\rho$ of low rank has finite image. We have already shown that, in the irreducible setting, such representations are defined over $\mathscr{O}_K$. We will use the techniques of \cite{LL:geometric-local-systems} to deduce that they have finite monodromy. To do so, we will argue that for each embedding $\tau: \mathscr{O}_K\hookrightarrow\mathbb{C}$, $\rho\otimes_{\mathscr{O}_K, \tau}\mathbb{C}$ has unitary monodromy. We know this for one such $\tau$ but not for the rest.  
To prove unitarity for all such $\tau$, we will use \cite[Theorem
1.2.12]{LL:geometric-local-systems}, and to verify its hypotheses we will need
to use non-abelian Hodge theory to construct certain complex PVHS.

\begin{proposition}\label{proposition:unitary-implies-finite}
Let  $$\rho:\pi_1(\Sigma_{g,n})\to
\on{GL}_r(\mathbb{C})$$ be a unitary MCG-finite representation with $r<\sqrt{g+1}$. Then $\rho$ has finite image.
\end{proposition}
\begin{proof}
As $\rho$ is unitary, it is semisimple. As an irreducible sub-representation of an MCG-finite representation is MCG-finite by \autoref{proposition:basic-properties}, we may without loss of generality assume $\rho$ is irreducible.

If $g< 3$, $r\leq 1$, and so we may conclude by \autoref{proposition:1-dim-finite-image}. Thus we may assume $g\geq 3$. By \autoref{lemma:gl-defined-over-number-field}, we may assume $\rho$ factors through $\on{GL}_r(\mathscr{O}_K)$ for some number field $K$ and some embedding $\iota: \mathscr{O}_K\hookrightarrow \mathbb{C}$.
We now adjust notation so that $\rho$ denotes this $\mathscr O_K$-representation and $\rho_\iota=\rho\otimes_{\mathscr{O}_K, \iota}\mathbb{C}$
denotes our original MCG-finite representation. By \autoref{corollary:converse-to-MCG-rep}, there exists a punctured versal family $\mathscr C^\circ \to \mathscr M$
and a representation $\rho'_\iota: \pi_1(\mathscr C^\circ) \to
\on{GL}_r(\mathbb{C})$ with finite determinant, restricting to $\rho_\iota$ on
the fundamental group of a fiber of $\pi^\circ$, i.e.~ for $C^\circ$ a fiber of
$\pi^\circ$, $\rho'_\iota|_{\pi_1(
C^\circ)}=\rho_\iota$. Note that $\rho'_\iota$ is irreducible (as $\rho_\iota$
is irreducible), and strongly cohomologically rigid, by \autoref{proposition:cohomologically-rigid}.

Choose any embedding $\tau: \mathscr O_K \hookrightarrow \mathbb C$ and let
$\rho_\tau := \rho\otimes_{\mathscr{O}_K, \tau}\mathbb{C}$. Note that $\rho_\tau$ is MCG-finite, as the same is true for $\rho_\iota$. 
We aim to show $\rho_\tau$ is unitary.

Choose $\sigma: \mathbb{C}\overset{\sim}{\to} \mathbb{C}$ so that $\sigma\circ\iota=\tau$, and set $\rho_\tau'=\rho_{\iota}'\otimes_{\mathbb{C}, \sigma} \mathbb{C}$. Note that $\rho_\tau'|_{\pi_1(
C^\circ)}=\rho_\tau$.
Now observe that $\rho_\tau'$ is strongly cohomologically rigid, as the same is true for $\rho_\iota'$, and strong cohomological rigidity (being a cohomological condition) is preserved by automorphisms of $\mathbb{C}$. Moreover $\rho_\tau'$ is irreducible with finite determinant, as the same is true for $\rho_\iota'$.

We next show $\rho_\tau$ is unitary.
By \cite[Theorem 1.2.12]{LL:geometric-local-systems}, 
in order to show $\rho_\tau$ is unitary, it is enough to show $\rho'_\tau$ underlies
a complex PVHS on $\mathscr{C}^\circ$.
Since $\rho'_\tau$ is strongly cohomologically rigid and irreducible with finite determinant,
this follows from \autoref{lemma:rigid-reps}.

We are thus in the following situation: $\rho: \pi_1(\Sigma_{g,n})\to \on{GL}_r(\mathscr{O}_K)$ is a representation such that for each $\tau: \mathscr{O}_K\hookrightarrow \mathbb{C}$, $\rho\otimes_{\mathscr{O}_K, \tau} \mathbb{C}$ is unitary. Such a representation has finite image by \cite[Lemma 7.2.1]{LL:geometric-local-systems}.
\end{proof}

\subsection{Reduction from the semisimple case to the unitary case}
\label{subsection:semisimple-case}
In this section we will prove \autoref{theorem:finite-image} for semisimple
representations.
The idea is to first deform our representation to a complex PVHS via
non-abelian Hodge theory. Then we will argue that this complex PVHS is in fact unitary, using the results of \cite{LL:geometric-local-systems}, so that we
can apply \autoref{proposition:unitary-implies-finite}. Finally we will use the rigidity properties of unitary MCG-finite representations to argue 
that our deformation was in fact trivial---the deformed representation was in fact the one we started with.

\begin{lemma}\label{lemma:infinitesmal-rigidity}
Let $\pi^\circ: \mathscr{C}^\circ\to \mathscr{M}$ be a punctured versal family
of curves of genus $g$. Let $c\in \mathscr{C}^\circ$ be a point, set
$m=\pi^\circ(c)$, and let $C^\circ=\mathscr{C}^\circ_m$. Let $$\rho_\infty: \pi_1(\mathscr{C}^\circ, c)\to GL_r(\mathbb{C}[[t]])$$ be a representation with $r<\sqrt{g+1}$ and constant determinant, and let $\rho_0=\rho_\infty\otimes\mathbb{C}$. Suppose $\rho_0|_{\pi_1(C^\circ, c)}$ is unitary. Then $\rho_\infty|_{\pi_1(C^\circ, c)}$ is conjugate to a constant representation.
\end{lemma}
\begin{proof}
	We apply \autoref{prop:stacky-split-constancy} to the short exact sequence $$1\to \pi_1(C^\circ, c)\to \pi_1(\mathscr{C}^\circ, c)\to \pi_1(\mathscr{M},m)\to 1.$$ We must verify that for any $n$ and any deformation $\gamma: \pi_1(\mathscr{C}^\circ, c)\to GL_r(\mathbb{C}[t]/t^{n+1})$ of $\rho_0$ with $\gamma|_{\pi_1(C^\circ, c)}$ constant, we have $$H^0(\mathscr{M}, R^1\pi_*\on{ad}(\gamma))=0.$$ But this follows from \autoref{theorem:unitary-rigid}. Indeed, $$\rk_{\mathbb{C}[t]/t^{n+1}} \on{ad}(\gamma)=r^2-1<g$$ as $r<\sqrt{g+1}$ by assumption.
\end{proof}

\begin{lemma}\label{lemma:global-rigidity}
Let $\pi^\circ: \mathscr{C}^\circ\to \mathscr{M}$ be a punctured versal family
of curves of genus $g$. Let $c\in \mathscr{C}^\circ$ be a point, set
$m=\pi^\circ(c)$, and let $C^\circ=\mathscr{C}^\circ_m$. Let $X$ be a
finite-type connected $\mathbb{C}$-scheme, and let $$\rho:
\pi_1(\mathscr{C}^\circ, c)\to \on{GL}_r(\mathscr{O}_X(X))$$ be a representation
with $r<\sqrt{g+1}$ and constant determinant. For a closed point $x\in X$ with residue field
$\kappa(x)=\mathbb{C}$, let $\rho_x:=\rho\otimes_{\mathscr{O}_X(X)}\kappa(x)$.
Suppose there exists a closed point $y\in X$ such that $\rho_y|_{\pi_1(C^\circ,
c)}$ is unitary. Then for each
closed point $x\in X$, $\rho_{x}|_{\pi_1(C^\circ, c)}$ is conjugate to $\rho_y|_{\pi_1(C^\circ, c)}$.
\end{lemma}
\begin{proof}
Without loss of generality $X$ is a reduced (not necessarily irreducible) connected curve, as any closed point $x\in X$ is connected to $y$ by such a curve. We claim it suffices to prove the statement if the parameter space $X$ is smooth and irreducible. 

To reduce to this case, suppose $X$ is neither smooth nor irreducible, and let $\tilde Y$ be the normalization of the component $Y$ of $X$ containing $y$, and $\iota:
\tilde Y\to X$ the natural map. By the case of smooth irreducible curves, $\iota^*\rho|_{\pi_1(C^\circ, c)}\otimes \kappa(y')$ is independent of $y'\in \tilde Y(\mathbb{C})$, and thus $\rho|_{\pi_1(C^\circ, c)}\otimes \kappa(y'')$ is independent of $y''\in Y$ as well, and in particular unitary for all $y''$. Let $Z$ be a connected component of the closure of $X-\on{im}(\iota)$. As $X$ is connected, $Z$ intersects $Y$ non-trivially, say at some point $z\in Y$. But as $z$ is in $Y$, $\rho_z|_{\pi_1(C^\circ, c)}$ is unitary, and hence we are done by induction on the number of components of $X$.

So we now assume $X$ is a smooth connected curve. Let $\on{Hom}(\pi_1(C^\circ, c), \on{GL}_r(\mathbb{C}))$ be the representation variety parametrizing $r$-dimensional complex representations of $\pi_1(C^\circ, c)$. There is a natural $\on{GL}_r(\mathbb{C})$ action on $\on{Hom}(\pi_1(C^\circ, c), \on{GL}_r(\mathbb{C}))$ induced by conjugation. Let $$f: X\to \on{Hom}(\pi_1(C^\circ, c), \on{GL}_r(\mathbb{C}))$$ be the map sending a point $x$ to $\rho_x|_{\pi_1(C^\circ, c)}$. We wish to show that the image of $f$ lies in a single $\on{GL}_r(\mathbb{C})$-orbit, namely that of $\rho_y|_{\pi_1(C^\circ, c)}$.

Choose a local parameter $t$ of $X$ at $y$, so that $\widehat{\mathscr{O}_{X,y}}\simeq \mathbb{C}[[t]]$. Let $$\rho_\infty: \pi_1(\mathscr{C}^\circ, c)\to GL_r(\mathbb{C}[[t]])$$ be the corresponding representation. By \autoref{lemma:infinitesmal-rigidity}, $\rho_\infty|_{\pi_1(C^\circ, c)}$ is conjugate to a constant representation. It follows that there exists a dense open
subset $U\subset X$ containing $y$ such that $f(U)$ is contained in the
$\on{GL}_r(\mathbb{C})$-orbit of $\rho_y|_{\pi_1(C^\circ, c)}$.

As $\rho_y|_{\pi_1(C^\circ, c)}$ is unitary, hence semisimple, its $\on{GL}_r(\mathbb{C})$-orbit in $\on{Hom}(\pi_1(C^\circ, c), \on{GL}_r(\mathbb{C}))$ is closed \cite[Theorem 30]{sikora2012character}. Hence the closure $\overline{U}$ of $U$ maps into the $\on{GL}_r(\mathbb{C})$-orbit of $\rho_y|_{\pi_1(C^\circ, c)}$. But as $X$ is irreducible, $\overline{U}=X$, completing the proof.
\end{proof}
\begin{theorem}
	\label{theorem:finite-image-semisimple}
	Let $$\rho: \pi_1(\Sigma_{g,n})\to \on{GL}_r(\mathbb{C})$$ be a semisimple MCG-finite representation, and suppose $r<\sqrt{g+1}$. Then $\rho$ has finite image.
\end{theorem}
\begin{proof}
	Observe that if $\rho$ is a sum of irreducible representations, each of
	which have finite monodromy, then $\rho$ has finite monodromy as well.
	As a summand of an MCG-finite representation is MCG-finite by
	\autoref{proposition:basic-properties}, it suffices to treat the case that $\rho$ is irreducible.

	By \autoref{corollary:converse-to-MCG-rep}, there exists a dominant
	\'etale map $\mathscr{M}\to \mathscr{M}_{g,n}$, and a local system
	$\mathbb{V}$ with finite determinant on the total space $\mathscr{C}^\circ$ of the associated
	family of $n$-times punctured curves of genus $g$, such that for
	$C^\circ$ a fiber of $\pi^\circ$, $\mathbb{V}|_{{C}^\circ}$ has monodromy representation given by $\rho$. 
		By \autoref{theorem:deformation-to-pvhs}, $\mathbb V$ admits a deformation with constant determinant to a
	local system $\mathbb V_0$ underlying a complex PVHS.
	Since $r < 2\sqrt{g+1}$, it follows from \cite[Theorem
	1.2.12]{LL:geometric-local-systems} that
	the local system $\mathbb V_0|_{{C}^\circ}$ is unitary. Moreover, $\mathbb V_0|_{{C}^\circ}$ is MCG-finite by \autoref{proposition:MCG-finite-family-of-curves}.
	
	It follows from \autoref{proposition:unitary-implies-finite}
	that $\mathbb V_0|_{{C}^\circ}$ has finite monodromy.
	To conclude, it's enough to show that $\mathbb V|_{{C}^\circ} =
	\mathbb V_0|_{{C}^\circ}$.
	But this is immediate from \autoref{lemma:global-rigidity}.
\end{proof}

\subsection{Completion of the non-semisimple case}
\label{subsection:non-semisimple}
We come to the final step, where we verify the non-semisimple case of \autoref{theorem:finite-image}.
The idea is to show that if we do have a non-semisimple MCG-finite representation of $\pi_1(\Sigma_{g,n})$ of low rank, we
can produce a certain
 finite cover $\Sigma_{g',n'}$ of $\Sigma_{g,n}$, which violates 
our results toward the Putman-Wieland conjecture.

If $\rho$ is a $G$-representation and $\rho'\subset \rho$ is a subrepresentation, we say that $\rho'$ is \emph{characteristic} if it is stable under all $G$-automorphisms of $\rho$ (for example, the socle of $\rho$, or an isotypic component thereof).
\begin{lemma}
	\label{lemma:semisimplicity}
Let $\rho$ be an MCG-finite representation of $\pi_1(\Sigma_{g,n})$, and let $\rho_1\subset\rho$ be a semisimple, characteristic subrepresentation of $\rho$. Let $\rho_2=\rho/\rho_1$, and suppose $\rho_2$ is semisimple as well. If $\rho_1, \rho_2$ have finite image and $\rk\rho<\sqrt{g+1}$, then $\rho$ has finite image as well, i.e.~the extension of $\rho_2$ by $\rho_1$ splits.
\end{lemma}
\begin{proof}
If $g=0$, $\rho$ is the zero representation, so without loss of generality we may assume $g$ is positive.

	Decompose $\rho_2^\vee \otimes \rho_1$ as $\rho_2^\vee \otimes \rho_1 \simeq \oplus_i \sigma_i^{\oplus n_i}$, with the $\sigma_i$ irreducible and pairwise non-isomorphic, and the $n_i$ positive integers.
	Since $\dim \rho_1 + \dim \rho_2 = \dim \rho < \sqrt{g+1}$, 
	we obtain from the AM-GM inequality that $\dim \rho_1 \cdot \dim \rho_2
	< (\sqrt{g+1}/2)^2 = (g+1)/4$. Therefore, for every $i$, $$n_i \cdot \dim
	\sigma_i < (g+1)/4$$
	and hence $n_i, \dim\sigma_i < (g+1)/4$. 
	To put ourselves in the setting of 
	\autoref{theorem:asymptotic-putman-wieland}, we choose an additional
	basepoint $v \in \Sigma_{g,n}$ and let $\Sigma_{g,n+1} := \Sigma_{g,n} \setminus\{v\}$. Let $\Gamma\subset \on{Mod}_{g, n+1}$ be a finite index subgroup stabilizing the conjugacy class of each $\sigma_i$.

	Let $\Sigma_{g',n'} \to \Sigma_{g,n}$ be a finite Galois cover, with Galois group $H$, upon which the local
	systems corresponding to both
	$\rho_1$ and $\rho_2$ trivialize (for example, the cover defined by $\ker(\rho_1\oplus\rho_2)$).
	It suffices to show that $\rho|_{\pi_1(\Sigma_{g',n'})}$ is trivial. As
	$\rho|_{\pi_1(\Sigma_{g',n'})}$ has abelian image, it factors through an
	$H$-equivariant map $$H_1(\Sigma_{g',n'})\to \rho_2^\vee\otimes
	\rho_1.$$ Equivalently, this is the extension class corresponding to
	$\rho$ in $$\on{Ext}^1_{\pi_1(\Sigma_{g', n'})}(\rho_2,
	\rho_1)=H^1(\pi_1(\Sigma_{g', n'}), \Hom(\rho_2,
	\rho_1))=\Hom(\pi_1(\Sigma_{g', n'}), \Hom(\rho_2, \rho_1)).$$
	Projecting onto the $\sigma_i$-isotypic piece of $\rho_2^\vee\otimes
	\rho_1$ yields a $\Gamma$-stable quotient of
	$H_1(\Sigma_{g',n'})^{\sigma_i}$, or equivalently a $\Gamma$-stable
	subspace of $H^1(\Sigma_{g',n'})^{\sigma_i}$, of rank at most $(g+1)/4$.
	Call this subspace $Q_i$.
	But by
	\autoref{theorem:asymptotic-putman-wieland}, we have that any nonzero
	$\Gamma$-stable subspace of $H^1(\Sigma_{g',n'})^{\sigma_i}$ has rank at
	least $2g-2\dim \sigma_i$, which satisfies the inequality
	\begin{align*}
		2g-2\dim \sigma_i >2g-(g+1)/2=\frac{3g}{2}-\frac{1}{2}>(g+1)/4
		> n_i \dim \sigma_i.
	\end{align*}
	Hence $Q_i$ equals zero. As this holds for all $i$, $\rho|_{\pi_1(\Sigma_{g',n'})}$ is trivial as desired.
\end{proof}

\subsection{}
\label{subsection:main-proof}
We finally complete the proof of our main theorem, that 
MCG-finite representations $$\rho: \pi_1(\Sigma_{g,n})\to
\on{GL}_r(\mathbb{C})$$ with $r<\sqrt{g+1}$ have finite image.
\begin{proof}[Proof of \autoref{theorem:finite-image}]
The proof is by induction on the length of the socle filtration. The base case is \autoref{theorem:finite-image-semisimple}. Now let $\rho_1$ be the socle of $\rho$.	Because the socle of a representation is characteristic both $\rho_1$
	and $\rho/\rho_1$ are
	 MCG-finite. 
	Therefore, both $\rho_1$ and $\rho/\rho_1$ have finite image by induction.
	We conclude by \autoref{lemma:semisimplicity}.
\end{proof}

\section{Consequences for arithmetic representations}\label{section:arithmetic-applications}
The main arithmetic consequence of \autoref{theorem:finite-image} is \autoref{theorem:arithmetic-consequence} below, which verifies a prediction of the Fontaine-Mazur conjecture, as we now explain. 
Throughout this section, we will no longer be working over $\mathbb C$.

\subsection{Application to relative Fontaine-Mazur}
\label{subsection:application-fontaine-mazur}
\begin{definition}
	\label{definition:arithmetic}
	Let $X$ be a variety over a finitely-generated field $K$ with algebraic closure $\overline{K}$, and 
$$\rho: \pi_1^{\text{\'et}}(X_{\overline{K}})\to
\on{GL}_r(\overline{\mathbb{Q}}_\ell)$$ a
continuous representation. We
say that $\rho$ is \emph{arithmetic} if its conjugacy class has finite orbit
under the action of $\on{Gal}(\overline{K}/K)$ on the set of isomorphism classes of
$\pi_1^{\text{\'et}}(X_{\overline{K}})$-representations. Here the action is
induced by the natural outer action of $\on{Gal}(\overline{K}/K)$ on
$\pi_1^{\text{\'et}}(X_{\overline{K}})$. 
\end{definition}

The relative form of the Fontaine-Mazur conjecture predicts
that all semisimple arithmetic representations are of geometric origin
\cite[Conjecture 1 bis]{petrov2020geometrically}. Our main arithmetic result, a
straightforward application of \autoref{theorem:finite-image}, is that this is true for representations of low rank on the \emph{generic} $n$-pointed curve of genus $g$. 
In fact, we show such representations are not only of geometric origin:
they even have finite image.
\begin{theorem}\label{theorem:arithmetic-consequence}
Let $K$ be a finitely-generated field of characteristic zero and $(C, x_1,
\cdots, x_n)$ a smooth $n$-pointed geometrically connected curve of genus $g$ 
over $K$,
such that the corresponding map $\on{Spec}(K)\to \mathscr{M}_{g,n, \mathbb{Q}}$ factors through the generic point. 
If $r<\sqrt{g+1}$, any continuous arithmetic representation 
\begin{align*}
\rho: \pi_1^{\text{\'et}}(C_{\overline{K}}\setminus\{x_1, \cdots, x_n\})\to \on{GL}_r(\overline{\mathbb{Q}}_\ell) 
\end{align*}
has finite image.
\end{theorem}
We prove this at the end of \autoref{subsection:application-fontaine-mazur}
after spelling out
some consequences.
\begin{remark}\label{remark:esnault-kerz}
	Let $K, (C, x_1, \cdots, x_n)$ be as in the theorem, and choose any
	embedding $K\hookrightarrow \mathbb{C}$. Let $C^{\on{an}}\setminus\{x_1,
	\cdots, x_n\}$ be the associated Riemann surface. By \cite[Lemma 7.5.1]{LL:geometric-local-systems}, 
	\autoref{theorem:arithmetic-consequence} implies that arithmetic
	representations are not Zariski-dense in the
	$\overline{\mathbb{Q}}_\ell$-points of the character variety
	parametrizing semisimple representations of
	$\pi_1(C^{\on{an}}\setminus\{x_1, \cdots, x_n\})$ of rank $r$, for
	$1<r<\sqrt{g+1}$. This answers negatively a question of Esnault and Kerz
	\cite[Question 9.1(1)]{esnault2020arithmetic}.
\end{remark}

Recall that a local system $\mathbb V$ on a smooth complex variety $X$ is of geometric origin if there is
some Zariski open $U \subset X$ and a smooth proper morphism $f: Y \to U$ so that
$\mathbb V|_U$ is a subquotient of $R^i f_* \mathbb C$ for some $i$.
\begin{corollary}
	\label{corollary:fontaine-mazur}
	Let $K$ and $(C, x_1, \cdots, x_n)$ be as in
	\autoref{theorem:arithmetic-consequence}, 
	and let
	$\rho$ be a continuous representation 
\begin{align*}
\rho: \pi_1^{\text{\'et}}(C_{\overline{K}}\setminus\{x_1, \cdots, x_n\})\to \on{GL}_r(\overline{\mathbb{Q}}_\ell) 
\end{align*}
with $\dim \rho < \sqrt{g+1}$. The following are equivalent.
	\begin{enumerate}
		\item $\rho$ has finite image.
		\item $\rho$ is arithmetic.
		\item For any embedding $K\hookrightarrow \mathbb{C}$ and any isomorphism $\mathbb{C}\simeq \overline{\mathbb{Q}}_\ell$, the local system corresponding to $\rho_\mathbb{C}|_{C^{\on{an}}\setminus\{x_1,
	\cdots, x_n\}}$ on $C^{\on{an}}\setminus\{x_1,
	\cdots, x_n\}$ is of geometric
			origin. 
		\item For any embedding $K\hookrightarrow \mathbb{C}$ and any isomorphism $\mathbb{C}\simeq \overline{\mathbb{Q}}_\ell$, the local system corresponding to $\rho_\mathbb{C}|_{C^{\on{an}}\setminus\{x_1,
	\cdots, x_n\}}$ on $C^{\on{an}}\setminus\{x_1,
	\cdots, x_n\}$ underlies an
			integral PVHS.
	\end{enumerate}
	Moreover, $(1), (3)$, and $(4)$ are equivalent whenever $\dim \rho <
	2\sqrt{g+1}$.
\end{corollary}
\begin{proof}
The equivalence of $(1)$ and $(2)$ when $\dim \rho <\sqrt{g+1}$ holds by
	\autoref{theorem:arithmetic-consequence}, and $(1)$
	certainly implies $(3)$ and $(4)$ for any value of $\dim \rho$.
	Next, $(3)$ implies $(4)$ for any value of $\dim \rho$ since any local system of geometric origin
	underlies an integral PVHS. This is well known, but see, for example, the proof of
	\cite[Corollary 1.2.7]{LL:geometric-local-systems}.

	To conclude, it remains to show that $(4)$ implies $(1)$ when $\dim \rho
	< 2\sqrt{g+1}$.
	This follows from \cite[Corollary 1.2.7]{LL:geometric-local-systems},
	by choosing an embedding $K\hookrightarrow \mathbb{C}$ corresponding to an analytically very general complex point of $\mathscr{M}_{g,n}$.
	By \cite[Corollary 1.2.7]{LL:geometric-local-systems},
	 a local system on an analytically very general curve underlying an integral PVHS has finite monodromy whenever its dimension is
	$< 2\sqrt{g+1}$.
	Therefore, $\rho$ also has finite monodromy.
\end{proof}
\begin{remark}\label{remark:fontaine-mazur}
	The relative Fontaine-Mazur conjecture predicts that all semisimple arithmetic  local systems
	are of geometric origin. Hence on a generic curve, all semisimple arithmetic local systems $\rho$ with $\dim \rho< 2 \sqrt{g+1}$ should  have finite monodromy by
	\autoref{corollary:fontaine-mazur}.
	In particular, this suggests one should be able to improve the bound of $\sqrt{g+1}$ in
	\autoref{theorem:arithmetic-consequence} to $2 \sqrt{g+1}$ if one
	restricts to semisimple representations.

Note that arithmeticity of $\rho$ on $C$ over $K$ is an a stronger condition than having finite orbit under
the action of the mapping class group after base change to $\mathbb C$, and so
this does not mean the bound of \autoref{theorem:finite-image-semisimple} can be improved
to $2 \sqrt{g+1}$. Indeed, it cannot be improved, even in the semisimple case, when $g = 1$, see
\autoref{remark:sharp-g-1}.
\end{remark}

\begin{proof}[Proof of \autoref{theorem:arithmetic-consequence}]
	Let $C^\circ_{\overline K} := 
	C_{\overline{K}}\setminus\{x_1, \cdots, x_n\}$. For $R$ a ring, let $\mathscr
	M_{g,n,R}$ denote the moduli stack of genus $g$, $n$-pointed curves over
	$\on{Spec} R$
	and let $\pi^\circ_R: \mathscr C_R^\circ \to \mathscr M_{g,n,R}$ denote the universal $n$-punctured genus
	$g$ curve over $R$.
	We use $C^\circ_R$ to denote a fiber of $\pi^\circ_R$.

	We claim (suppressing basepoints) that there is a map of exact sequences of \'etale fundamental groups
	\begin{equation}
	\label{equation:map-of-birman-exact-sequences}
	\begin{tikzcd}
	0 \ar {r} & \pi^\text{\'et}_1(C^\circ_{\overline K}) \ar {r} \ar {d}{\simeq} &
	\pi^\text{\'et}_1(C^\circ_K) \ar {r} \ar {d} & \pi^\text{\'et}_1(\on{Spec} K) \ar {r}
	\ar {d} & 0 \\
	0 \ar {r} & \pi^\text{\'et}_1(C^\circ_{\overline {K(M_{g,n,\mathbb Q})}}) \ar {r} &
	\pi^\text{\'et}_1(\mathscr C^\circ_{\mathbb Q}) \ar {r} & \pi^\text{\'et}_1(\mathscr
	M_{g,n,\mathbb Q}) \ar {r} & 0.
	\end{tikzcd}\end{equation}
	Indeed, the first sequence is the standard homotopy exact sequence, see
	\cite[\href{https://stacks.math.columbia.edu/tag/0BTX}{Tag
	0BTX}]{stacks-project}.
	The second one is less standard but, using 
	\cite[\href{https://stacks.math.columbia.edu/tag/0BTX}{Tag
	0BTX}]{stacks-project} again, its exactness
	can be reduced to verifying exactness of
	\begin{equation}
	\label{equation:q-curve-fibration}
	\begin{tikzcd}
	0 \ar {r} &  \pi^\text{\'et}_1(C^\circ_{\overline {K(M_{g,n,\mathbb Q})}}) \ar {r} &
	\pi^\text{\'et}_1(\mathscr C^\circ_{\overline {\mathbb Q}}) \ar {r} & \pi^\text{\'et}_1(\mathscr
	M_{g,n,\overline{ \mathbb Q}}) \ar {r} & 0.
	\end{tikzcd}\end{equation}
	By comparison of fundamental groups of algebraically closed fields of
	characteristic $0$, it is enough to verify exactness of the sequence
	with all fields in subscripts replaced by the complex numbers.
	The version for topological fundamental groups is given in
	\autoref{lemma:mgn-fundamental-group}. Since the \'etale fundamental
	group is the profinite completion of the topological fundamental group
	\cite[Expos\'e XII, Corollaire 5.2]{sga1},
	and profinite completions of exact sequences where the left term has
	trivial center remain exact \cite[Proposition
	3]{anderson:exactness-properties}, we obtain exactness of
	\eqref{equation:q-curve-fibration}.

We next use \eqref{equation:map-of-birman-exact-sequences} to deduce that 
$\rho$ has finite orbit under the natural action of $\pi_1(\mathscr
M_{g,n,\mathbb \overline{\mathbb Q}})$, to be constructed below.
	The two exact sequences in
	\eqref{equation:map-of-birman-exact-sequences} induce compatible outer actions of
	$\pi_1(\on{Spec} K)$ and $\pi_1(\mathscr
	M_{g,n,\mathbb Q})$ on 
	$\pi_1(C^\circ_{\overline {K(M_{g,n,\mathbb Q})}})$.
	Hence, the outer action of
	$\pi_1(\on{Spec} K)$ factors through that of $\pi_1(\mathscr
	M_{g,n,\mathbb Q})$.
	Because $\on{Spec} K \to \mathscr
	M_{g,n,\mathbb Q}$ is dominant, the induced map
$\pi_1(\on{Spec} K) \to \pi_1(\mathscr M_{g,n,\mathbb Q})$ has image of finite
index.
In particular, the orbit of the conjugacy class of $\rho$ under
$\pi_1^{\text{\'et}}(\mathscr{M}_{g,n, \mathbb{Q}}, \overline{\eta})$,
and hence under $\pi_1^{\text{\'et}}(\mathscr{M}_{g,n,
\overline{\mathbb{Q}}}, \overline{\eta})$, is finite. 

	Choose an embedding $\overline{K}\hookrightarrow\mathbb{C}$, and let
	$(\mathscr{C}^\circ)^{\text{an}}$ be the corresponding punctured Riemann
	surface, so that,
	by residual finiteness of surface groups,
	there is a natural inclusion
	$\pi_1((\mathscr{C}^\circ)^{\text{an}})\hookrightarrow
	\pi_1^{\text{\'et}}(\mathscr{C}^\circ_{\overline{K}})$. As the conjugacy
	class of $\rho$ has finite orbit under
	$\pi_1^{\text{\'et}}(\mathscr{M}_{g,n, \overline{\mathbb{Q}}},
	\overline{\eta})
	\simeq \pi_1^{\text{\'et}}(\mathscr{M}_{g,n, \mathbb{C}},
\overline{\eta}_{\mathbb C})$, the standard comparison theorems between \'etale and
	topological $\pi_1$
	\cite[Expos\'e XII, Corollaire 5.2]{sga1}
	imply that
	$\rho|_{\pi_1((\mathscr{C}^\circ)^{\text{an}})}$ has finite orbit under
	the topological fundamental group
	$\pi_1(\mathscr{M}_{g,n, \mathbb{C}}^{\text{an}},
	\overline{\eta}_{\mathbb{C}})$. Hence
	$\rho|_{\pi_1((\mathscr{C}^\circ)^{\text{an}})}$ is MCG-finite, and hence has
	finite image by \autoref{theorem:finite-image},
	which may be applied after choosing an isomorphism between $\overline{\mathbb{Q}}_\ell$ and $\mathbb{C}$. 
	But
	$\pi_1( (\mathscr{C}^\circ)^{\text{an}})$ is dense in
	$\pi_1^{\text{\'et}}(\mathscr{C}^\circ_{\overline{K}})$, so $\rho$ also
	has finite image.
\end{proof}

\subsection{Application to lifting residual representations}
Let $K, (C, x_1, \cdots, x_n)$ be as in
\autoref{theorem:arithmetic-consequence}. We next explain why there are residual
representations of $\pi_1^{\text{\'et}}(C_{\overline{K}}\setminus\{x_1, \cdots,
x_n\})$ which do not lift to arithmetic representations in characteristic zero.
In particular, in \autoref{example:non-liftable-representations}
we construct residual representations of geometric
fundamental groups which are not ``of geometric origin," as we now define. To our knowledge these are the first such examples.
\begin{definition}
Let $K$ be an algebraically closed field, and let $X/K$ be a variety. Let $\mathbb{F}$ be a finite field of characteristic different from that of $K$. Let $L$ be the fraction field of the Witt vectors $W(\mathbb{F})$. We say that a continuous representation $$\rho: \pi_1^{\text{\'et}}(X_{\overline{K}})\to \on{GL}_r(\mathbb{F})$$ is \emph{of geometric origin} if there exists a continuous representation $$\xi: \pi_1^{\text{\'et}}(X_{\overline{K}})\to \on{GL}_r(\overline{L})$$ such that:
\begin{enumerate}
\item There exists a $\xi$-stable
	$\mathscr{O}_{\overline{L}}$-lattice $W\subset \overline{L}^r$ with
	$W\otimes \overline{\mathbb{F}}\simeq \rho \otimes
	\overline{\mathbb{F}}$.
\item There exists a dense open subscheme $U\subset X$ and a smooth proper morphism $\pi: Y\to U$, such that $\xi|_{\pi_1^{\text{\'et}}(U_{\overline{K}})}$ arises as a subquotient of the monodromy representation of $R^i\pi_*\overline{L}$ for some $i\geq 0.$
\end{enumerate}
\end{definition}
In other words, we say a residual representation is of geometric origin if it
arises as the reduction of a characteristic zero representation of geometric
origin.
\begin{corollary}\label{corollary:lifting-and-geometric-origin}
	Let $K, (C, x_1, \cdots, x_n)$ be as in \autoref{theorem:arithmetic-consequence}. Then for $1<r<\sqrt{g+1}$ and  $p\gg_r 0$, no surjective representation 
	$$\rho: \pi_1^{\text{\'et}}(C_{\overline{K}}\setminus\{x_1, \cdots, x_n\})\to \on{GL}_r(\mathbb{F}_p)$$ admits an arithmetic lift to characteristic zero. In particular no such representation is of geometric origin.
\end{corollary}
\begin{proof}
	By definition, a mod $p$ representation of geometric origin lifts to a
representation of geometric origin over $\overline{\mathbb{Q}}_p$. As representations of geometric origin are arithmetic, the second statement follows from the first.
	
	Thus it suffices to show that for $p\gg 0$ and $\rho$ as in the theorem, $\rho$ does not admit an arithmetic lift. Suppose to the contrary that it did admit an arithmetic lift $\xi$; by \autoref{theorem:arithmetic-consequence}, $\xi$ would have finite image. Thus it suffices to show that for $Q$ a finite, totally ramified extension of $\mathbb{Q}_p$, there do not exist any finite subgroups of $\on{GL}_r(\mathscr{O}_Q)$ surjecting onto $\on{GL}_r(\mathbb{F}_p)$ if $p\gg_r 0$. 
	
	This follows from Jordan's theorem on finite subgroups of $\on{GL}_r(Q)$,
	where $Q$ is any field of characteristic zero
	(\cite[p.~91]{jordan1878memoire} or \cite[Theorem
	36.13]{curtis1966representation}). Recall that Jordan's theorem says
	that there exists some constant $n(r)$ such that if $G\subset \on{GL}_r(Q)$ is a finite subgroup, $G$ contains an abelian normal subgroup of index dividing $n(r)$. In particular, for any $g_1, g_2\in G$, we have $$[g_1^{n(r)}, g_2^{n(r)}]=I_r,$$ where $I_r$ is the $r\times r$  identity matrix.
	
	Now define $\overline{g_1}, \overline{g_2}\in \on{GL}_r(\mathbb{F}_p)$ by
	$$\overline{g_1}=\begin{pmatrix} 1 & 1 \\ 0 & 1 \end{pmatrix} \oplus
	I_{r-2}$$ $$\overline{g_2}=\begin{pmatrix} 1 & 0 \\ 1 & 1 \end{pmatrix}
	\oplus I_{r-2}.$$ Direct computation shows that $[\overline{g_1}^{n(r)},
	\overline{g_2}^{n(r)}]\neq I_r$ as long as $p$ does not divide $n(r)$.
	But if $\on{GL}_r(\mathbb{F}_p)$ was the image of a finite subgroup of
	$\on{GL}_r(\mathscr{O}_Q)$, this equality would hold by Jordan's theorem, giving the claim. Indeed, we have shown the result for all $p$ not dividing the constant $n(r)$ from Jordan's theorem.
\end{proof}
Using \autoref{corollary:lifting-and-geometric-origin}, we now construct
examples of residual representations with no arithmetic lifts.
\begin{example}
	\label{example:non-liftable-representations}
	Fix non-negative integers $g,n$ so that $n \geq 1$ if $g =1$ and $n\geq
	3$ if $g = 0$.
	Fix $r$ with $1<r<\sqrt{g+1}$. We claim
	that for any $n$-pointed genus $g$ curve $(C, x_1, \cdots, x_n)$, and any $p$,
	there exist surjective representations $$\rho:
	\pi_1^{\text{\'et}}(C_{\overline{K}}\setminus\{x_1, \cdots, x_n\})\to
	\on{GL}_r(\mathbb{F}_p)$$ as in \autoref{corollary:lifting-and-geometric-origin}. 
	In particular these representations do not admit arithmetic lifts to
	characteristic $0$.

	If $n > 0$, maintaining the assumption $n \geq 3$ when $g = 0$, we have
	that
$\pi_1^{\text{\'et}}(C_{\overline{K}}\setminus\{x_1, \cdots, x_n\})$ is
profinite free on at least $2$ generators.
The claim in this case follows from the fact that $\on{GL}_r(\mathbb{F}_p)$ is generated by two elements
\cite{waterhouse:two-generators}. 

We now deal with the case $n =0$ and $g \geq 2$.
Indeed, this follows from the fact that
$\pi_1^{\text{\'et}}(C_{\overline{K}})$ surjects onto a free profinite group on
$g$ generators by writing it as the profinite completion of $\langle a_1, \ldots,
a_g, b_1, \ldots, b_g \rangle/ \prod_{i=1}^g [a_i, b_i]$ and considering the map
to the free profinite group generated by $c_1, \ldots, c_g$ sending $a_i \mapsto
c_i, b_i \mapsto \on{id}$.
\end{example}

We conclude the section with several remarks describing consequences of the
above non-liftable residual representations.
\begin{remark}\label{remark:dejong}
\autoref{corollary:lifting-and-geometric-origin} provides a counterexample to a particularly optimistic extension of de Jong's conjecture \cite[Conjecture 2.3, Theorem 3.5]{de2001conjecture}.	De Jong's conjecture, proven by Gaitsgory \cite{gaitsgory2007jong}, implies that an absolutely irreducible residual $\mathbb{F}_q$-representations of geometric fundamental groups of smooth curves over finite fields of characteristic not dividing $q$ always lift to arithmetic representations over a field of characteristic zero. One might naturally ask if the same is true for curves over arbitrary finitely-generated fields; \autoref{corollary:lifting-and-geometric-origin} shows that this is not the case.
\end{remark}
\begin{remark}
	\label{remark:complete-intersection}
	The solution to de Jong's conjecture, as described in
	\autoref{remark:dejong}, implies that deformation rings of
absolutely irreducible $\mathbb{F}_p$-representations of arithmetic fundamental
groups of curves over finite fields (of characteristic different from $p$) are
always complete intersections over $\mathbb{Z}_p$. Our results show that the
analogous statement is not true for the arithmetic fundamental group
$\pi_1^{\text{\'et}}(C_{K}\setminus\{x_1, \cdots, x_n\})$ as in
\autoref{theorem:arithmetic-consequence} because they are not flat over $\mathbb
Z_p$ by \autoref{example:non-liftable-representations}. 
\end{remark}
\begin{remark}
	\label{remark:flach}
	Flach asks  \cite[p.~7]{cornelissen2005problems} if deformation rings of
	absolutely irreducible residual representations of profinite groups are
	always complete intersections. By now it is well-known that the answer
	to this question is in general ``no" (see e.g.~\cite{eardley2016inverse}
	for a more or less complete answer to this question, and the references
therein). Our result \autoref{corollary:lifting-and-geometric-origin} shows that
this question has a negative answer even for arithmetic fundamental groups of
generic smooth  curves, as explained in \autoref{remark:complete-intersection}.
\end{remark}

\section{Questions and examples}\label{section:questions}
\subsection{Rigidity questions for mapping class groups}
One of the key ingredients of our arguments is the statement that certain representations of (finite index subgroups of) $\on{Mod}_{g,n}$ are rigid.
\begin{question}\label{question:rigidity}
Let $g\gg 0$. Is every irreducible complex representation of every finite index subgroup of $\on{Mod}_{g,n}$ rigid?
\end{question}
The answer is, perhaps, plausibly ``yes" if one accepts the well-known analogy
between $\on{Mod}_{g,n}$ and lattices in simple Lie groups of rank $>1$, as
all complex representations of such lattices are rigid.  
For $g$ small, there
are examples of irreducible non-rigid representations of $\on{Mod}_{g,n}$ (for
example, $\pi_1(\mathscr{M}_{0,4})$ is free on two generators). 

\begin{remark}
	\label{remark:}
	A positive answer to \autoref{question:rigidity} would imply Ivanov's
conjecture that $\on{Mod}_{g,n}$ does not virtually surject onto $\mathbb{Z}$
as stated in \cite[\S 7]{Ivanov:problems} and \cite[Problem
2.11.A]{Kirby:problems}. Indeed, suppose some finite index subgroup
$\Gamma\subset \on{Mod}_{g,n}$ admitted a surjection $f: \Gamma
\twoheadrightarrow\mathbb{Z}$. For $t\in \mathbb{C}^\times$, the representation
\begin{align*}
\rho_t: \mathbb{Z} &\to \mathbb{C}^\times \\
n & \mapsto t^n
\end{align*}
is a non-trivial
family of one-dimensional representations of $\mathbb{Z}$. Then the
representation $\rho_t\circ f$ is a non-constant family of irreducible
representations of $\Gamma$.
\end{remark}

A positive answer to \autoref{question:rigidity} would also imply a positive answer to the following question:
\begin{question}\label{question:finite-fixed-points}
Let $g\gg0$. Let $\Gamma\subset \on{Mod}_{g,n}$ be a finite index subgroup, and
let $X_r(\pi_1(\Sigma_{g,n}))$ be the character variety parametrizing
$r$-dimensional semisimple representations of $\pi_1(\Sigma_{g,n})$. Is the fixed locus  $X_r(\pi_1(\Sigma_{g,n}))^\Gamma$ finite?
\end{question}

\begin{remark}
One reason to believe \autoref{question:finite-fixed-points} has a positive answer is that the profinite analogue does. Namely, let $X$ be any normal connected algebraic variety over $\mathbb{C}$, and let $\pi_1^{\text{\'et}}(X)$ be its profinite fundamental group. Let $\Gamma$ be a finite index subgroup of $\on{Out}(\pi_1^{\text{\'et}}(X))$. Then $\Gamma$ acts on the set of conjugacy classes of continuous representations $$\pi_1^{\text{\'et}}(X)\to \on{GL}_r(\overline{\mathbb{Q}}_\ell).$$ It follows from \cite[Remark 1.1.4]{litt2021arithmetic} that the $\Gamma$-fixed points are \emph{discrete} in the $\ell$-adic topology. (The result there is stated for curves, but one may reduce to this case by the Lefschetz hyperplane theorem.)
\end{remark}

If one accepts in addition Simpson's motivicity conjecture for rigid local systems \cite[Conjecture on p.~9]{simpson1992higgs}, a positive answer to \autoref{question:rigidity} would imply that every semisimple local system on $\mathscr{M}_{g,n}$ is of geometric origin for $g\gg 0$. This would be of particular interest for those local systems not obviously of geometric origin, e.g.~the local systems arising from TQFT constructions. 
Thus one might ask:

\begin{question}
For $g\gg 0$, is every semisimple local system on $\mathscr{M}_{g,n}$ of geometric origin? Let $C$ be a Riemann surface of genus $g\gg 0$. Is every semisimple MCG-finite representation of $\pi_1(C)$ of geometric origin?
\end{question}

\begin{example}\label{non-rigid-examples}
The analogy between the representation theory of mapping class groups and the representation theory of lattices in simple Lie groups of rank greater than one only goes so far. Indeed, mapping class groups typically admit non-rigid reducible representations, as we now explain. 
\begin{enumerate}
\item Let $H=H_1(\Sigma_g, \mathbb{Q})$. Morita \cite{morita1993extension} produces a nonzero class $$\sigma\in H^1(\on{Mod}_g, (\wedge^3 H)/H),$$ closely related to the Johnson homomorphism, yielding a non-split extension of $\on{Mod}_g$-representations $$1\to \wedge^3H/H\to W\to \mathbb{Q}\to 1.$$ The representation $W$ is evidently a non-trivial deformation of $\mathbb{Q}\oplus \wedge^3H/H$, and hence this latter representation is non-rigid.

\item An essentially identical but perhaps slightly simpler argument gives examples for punctured surfaces. For $n>0$ the group $\on{PMod}_{g,n}$ acts on $\pi_1(\Sigma_{g,n-1}, x)$, and hence on the group algebra $\mathbb{Q}[\pi_1(\Sigma_{g,n-1}, x)]$. Let $\mathscr{I}\subset \mathbb{Q}[\pi_1(\Sigma_{g,n-1}, x)]$ be the augmentation ideal. Then direct computation shows that the short exact sequence of $\on{PMod}_{g,n}$-representations $$0\to \mathscr{I}^2/\mathscr{I}^3\to \mathscr{I}/\mathscr{I}^3\to \mathscr{I}/\mathscr{I}^2\to 0$$ does not split, and hence the representation $\mathscr{I}/\mathscr{I}^2\oplus \mathscr{I}^2/\mathscr{I}^3$ is not rigid.

\item \label{h1-example} For another example, one may consider the non-torsion class in $$H^1(\on{Mod}_{g,1}, H^1(\Sigma_g, \mathbb{Z}))$$ constructed in \cite[Proposition 6.4]{morita1989families}, which yields a non-split extension of $\on{Mod}_{g,1}$-representations $$1\to H^1(\Sigma_g, \mathbb{Q})\to W\to \mathbb{Q}\to 1.$$ One may make the restriction to the point-pushing subgroup $W|_{\pi_1(\Sigma_g)}$ explicit as follows. Let $$a: \pi_1(\Sigma_g)\to \pi_1(\Sigma_g)^{\on{ab}}\otimes\mathbb{Q}\simeq H_1(\Sigma_g, \mathbb{Q})$$  be the map induced by the Hurewicz isomorphism $H_1(\Sigma_g, \mathbb{Z})=\pi_1(\Sigma_g)^{\on{ab}}$. Choose a splitting $W=H_1(\Sigma_g, \mathbb{Q})\oplus \mathbb{Q}e$, where we view $H_1(\Sigma_g, \mathbb{Q})$ as isomorphic to $H^1(\Sigma_g, \mathbb{Q})$ via the universal coefficient theorem and Poincar\'e duality. Then $\pi_1(\Sigma_g)$ acts trivially on $H_1(\Sigma_g, \mathbb{Q})$, and an element $\gamma\in\pi_1(\Sigma_g)$ acts on $e$ via the map $$e\mapsto e+a(\gamma).$$  
\end{enumerate}
\end{example}

\begin{remark}\label{remark:FLM-rank-1-result}
	For a related discussion of ways mapping class groups \emph{do not} behave like lattices in simple Lie groups of rank greater than one, see \cite{farbLM:rank-1-phenomena}. Interestingly, \cite[Theorem 1.6]{farbLM:rank-1-phenomena} shows there are no faithful linear
representations of finite index subgroups of $\on{Mod}_{g,0}$ of dimension $< 2 \sqrt{g-1}$. A related result follows from our work here, as we now explain.

Let $\Gamma\subset \on{PMod}_{g,n+1}$ be a finite-index subgroup containing the point-pushing subgroup. 
It follows from \autoref{theorem:finite-image},
\autoref{lemma:finite-rho-implies-finite-lift}(1), and
\autoref{proposition:MCG-finite-index-rep}
that if $\rho$ is a representation of $\Gamma$ of rank $<\sqrt{g+1}$ 
\begin{enumerate}
\item 	whose restriction to the point-pushing subgroup is irreducible, and
\item with finite determinant
\end{enumerate}
then $\rho$ has finite image.
Although the two results are not comparable,
it is interesting that the bound in
\cite{farbLM:rank-1-phenomena} is also asymptotic to $\sqrt{g}$. 
\end{remark}

\subsection{Bounds and examples}\label{question:bounds}
It is natural to ask how sharp the bound of $\sqrt{g+1}$ in
\autoref{theorem:finite-image} is. We have no reason to believe it is sharp.
That said, one cannot expect a bound that is much stronger, as the extension
$W$ constructed in \autoref{non-rigid-examples}(\ref{h1-example}) yields a
non-trivial unipotent representation of $\pi_1(\Sigma_g)$ of rank $2g+1$, fixed
by the action of $\on{Mod}_g$, namely $W|_{\pi_1(\Sigma_g)}$.  

\begin{remark}
	\label{remark:sharp-g-1}
	The bounds of many of our main results are sharp when $g=1$. 
In $g = 1$, there are semisimple $2$-dimensional ``special dihedral" MCG-finite representations
of $\pi_1(\Sigma_{1,n})$ for $n > 0$
\cite[Theorem B]{biswas2017surface},
with image contained in the infinite dihedral group.
This shows sharpness of \autoref{theorem:finite-image} and also the semisimple
case, \autoref{theorem:finite-image-semisimple}.
Moreover, there exist $2$-dimensional non-semisimple MCG-finite representations of $\pi_1(\Sigma_{1,n})$ for $n>0$, see \cite[Theorem B]{CH:isomonodromic}.
This implies sharpness of the bound in \autoref{lemma:semisimplicity}.

Note that the ``special dihedral" MCG-finite representations of \cite[Theorem B]{biswas2017surface} are not in general arithmetic in the sense of \autoref{definition:arithmetic}, as in general their local monodromy at infinity is not quasi-unipotent.
\end{remark}
\begin{remark}
	\label{remark:not-sharp-g-2-3}
	On the other hand, the bound $r < \sqrt{g+1}$ in \autoref{theorem:finite-image} is not sharp when $g =
	2$ and $g =3$.
	Indeed, to verify this, we only need show there are no $2$-dimensional
	MCG-finite representations with infinite image. There are no non-semisimple such
	representations by \cite[Theorem B]{CH:isomonodromic}, and so it only
	remains to show there are no irreducible $2$-dimensional
	representations.
	Any such representation $\rho: \pi_1(\Sigma_{g,n}) \to \on{GL}_2(\mathbb
	C)$ has finite determinant by \autoref{lemma:finite-central-part}.
	Replacing $\rho$ by $\rho \otimes \det \rho^{-1/2}$, we would obtain another
	MCG-finite $2$-dimensional representation with infinite image factoring
	through $\on{SL}_2(\mathbb C)$.
	No such representations	exist by \cite[Theorem A]{biswas2017surface}.
\end{remark}

\begin{question}
Let $g, n\geq 0$. What is the minimal rank of an MCG-finite representation of $\pi_1(\Sigma_{g,n})$ with infinite image?
\end{question}

	It is plausible that much stronger bounds 
than $\sqrt{g+1}$ hold in \autoref{theorem:finite-image}
if one assumes the representation is semisimple (see \autoref{figure:MCG-geography}).

\begin{example}
	\label{example:semisimple-bound}
We know of examples of MCG-finite
semisimple representations of $\pi_1(\Sigma_g)$ with infinite image, of rank exponential in $g$. For example, one may via TQFT techniques construct representations of $\on{Mod}_{g,n}$
of rank exponential in $g$
\cite[Corollary 4.3]{koberda2016quotients}, \cite[Theorem 5.1]{biswas2018representations}; 
these representations are non-trivial when restricted to the point-pushing subgroup, and hence by \autoref{proposition:MCG-finite-index-rep} yield MCG-finite representations.
Similarly, representations of $\on{Mod}_{g,n}$ constructed via variants of
the Kodaira-Parshin trick 
\cite[Example 3.3.1]{lawrence2019representations}
(see also \cite[Proposition 5.1.1 and Remark 5.1.2]{LL:geometric-local-systems})
are semisimple of rank exponential in $g$ and restrict to 
MCG-finite, semisimple representations of the point-pushing subgroup, with infinite image. 
\end{example}

\begin{example}
	In genus zero, there is a huge collection of MCG-finite representations, namely the \emph{rigid local systems} studied by Katz \cite{katz2016rigid}. The use of rigidity in our proof of \autoref{theorem:finite-image}, and \autoref{question:rigidity}, suggest that the study of MCG-finite representations is a natural generalization  of the study of rigid local systems to the higher genus setting. 
\end{example}

\begin{question}
Let $g, n\geq 0$. What is the minimal rank of a semisimple MCG-finite representation of $\pi_1(\Sigma_{g,n})$ with infinite image?
\end{question}

In our view it would be extremely interesting to give an improved bound for semisimple MCG-finite representations, and to produce fundamentally new examples of semisimple MCG-finite representations.

\begin{remark}[Bounds for free groups]
	\label{example:free-group-bounds}
	A variant of the construction in \autoref{non-rigid-examples}(\ref{h1-example})
gives an example of a representation of the free group $F_N$ on $N$ generators
of rank $N+1$, whose conjugacy class is fixed by the action of $\on{Out}(F_N)$,
see \cite[p.~1444]{potapchik2000low}. For example, when $N=2$, this representation is given by the matrices $$\begin{pmatrix} 1 & 0 & 1 \\ 0 & 1 & 0 \\ 0 & 0 & 1 \end{pmatrix}, \begin{pmatrix} 1 & 0 & 0 \\ 0 & 1 & 1 \\ 0 & 0 & 1 \end{pmatrix}.$$
Thus again the bound of $\sqrt{g+1}$ in
\autoref{corollary:free-groups} cannot be improved too much further if one
allows non-semisimple representations. However it is in principle possible that
there is \emph{no} semisimple representation of $F_N$ with infinite image whose
conjugacy class has finite orbit under $\on{Out}(F_N)$ when $N\geq 3$. Indeed, this would
follow from a conjecture of Grunewald and Lubotzky, see \cite[\S9.2]{grunewald2009linear}, using the main result of \cite{farb2017moving}.

There are interesting semisimple representations of $F_2$ whose conjugacy class
has finite orbit under $\on{Out}(F_2)$. This follows, for example, from the linearity of $\on{Aut}(F_2)$ \cite{krammer2000braid}, e.g.~by choosing a faithful representation of $\on{Aut}(F_2)$ and restricting to the inner automorphisms, and then semisimplifying. Alternately, one may construct examples by applying the Kodaira-Parshin trick to construct representations of $\on{Mod}_{1,2}$, which is commensurable with $\on{Aut}(F_2)$.
\end{remark}

As remarked in the introduction, we conjecture the opposite of Grunewald and Lubotzky:
\begin{conjecture}
	\label{conjecture:infinite-image-semisimple}
For all $N>0$, there exist semisimple representations of $F_N$ with infinite image, whose conjugacy class has finite orbit under $\on{Out}(F_N)$. 
\end{conjecture}
Our main impetus for this conjecture is the existence of many interesting
MCG-finite representations of $\pi_1(\Sigma_{g,n})$, see \autoref{example:semisimple-bound}. In our view this conjecture is of great importance, both because of the intrinsic interest of ``canonical" representations of $F_N$ (i.e.~those with finite orbit under $\on{Out}(F_N)$), and because of its relationship to the representation theory of $\on{Aut}(F_N)$. 

\begin{remark}
We briefly spell out the the relationship between representations of $F_N$ whose conjugacy class has finite orbit under $\on{Out}(F_N)$, and the representation theory of
$\on{Aut}(F_N)$.
If $\Gamma\subset \on{Aut}(F_N)$ is a finite index subgroup containing the inner automorphisms, and
$$\rho:\Gamma\to \on{GL}_r(\mathbb{C})$$ is a representation, then the conjugacy
class of $\rho|_{F_N}$ has finite orbit under $\on{Out}(F_N)$, via a proof
analogous to that of
\autoref{proposition:MCG-finite-index-rep}. Conversely, as in to the proof of \autoref{lemma:mcg-rep-construction}, the projectivization of any irreducible $F_N$-representation whose conjugacy class has finite orbit under $\on{Out}(F_N)$ arises from a projective representation of a finite index subgroup of $\on{Aut}(F_N)$. See \cite[Theorem 7.12]{baumeister-kielak-pierro} for a related result about low-rank projective representations of certain finite index subgroups of $\on{Aut}(F_n)$.
\end{remark}
\subsection{Classification}
The following is evidently quite difficult, but in our view is of great interest:
\begin{question}
Can one classify MCG-finite representations of $\pi_1(\Sigma_{g,n})$	 of rank $r$? Equivalently, can one classify algebraic solutions to the rank $r$ isomonodromy differential equations over $\mathscr{M}_{g,n}$? Can one classify representations of the free group $F_N$ on $N$ generators whose conjugacy class has finite orbit under $\on{Out}(F_N)$?
\end{question}
\begin{remark}
	\label{remark:}
It would also be extremely interesting to learn of any new sources of
\emph{examples} of MCG-finite representations. For example, are there MCG-finite representations $\pi_1(\Sigma_{g,n})\to \on{SL}_r(\mathbb{C})$ with (Zariski-)dense image for $r$ arbitrarily large?
\end{remark}

\appendix
\section{Proof of \autoref{theorem:period-map-computation}}\label{appendix}
We now explain the proof of \autoref{theorem:period-map-computation}, which loosely follows \cite[\S9 and \S10]{Voisin:hodgeTheory}. We retain notation as in \autoref{section:unitary-period-map}.
\subsection{Preparatory Lemmas}
Recall that the derivative of the period map $dP$ was a map 	$$dP : T_{\mathscr B} \to P^* T_{\on{Gr}(\rk F^1 \mathscr H, \rk \mathscr H)} \simeq (F^1
	\mathscr H)^\vee \otimes (\mathscr H/F^1 \mathscr H).$$
By adjointness, we obtain from $dP$ a map
\begin{align*}
	dP' : F^1 \mathscr H \to 
	(\mathscr H/F^1 \mathscr H) \otimes T_{\mathscr B}^\vee.
\end{align*}
Recall that we defined in \autoref{subsection:period-map-derivative} a map $\overline{\nabla}$ with the same source and target.

\begin{lemma}
	\label{lemma:connection-and-period-map}
	We have an equality $\overline \nabla = dP'$.
\end{lemma}
\begin{proof}
	Let $W$ be a vector space and $K\subset W$ a subspace of rank $r$. It is shown in the course of the proof of \cite[Lemma 10.7]{Voisin:hodgeTheory}
	that for $[K] \in \on{Gr}(r,W)$, the identification $T_{\on{Gr}(r,W),[K]} \simeq
	\Hom(K,W/K)$ is given as follows.
	Let $\mathscr S \subset \mathscr{O}_{\on{Gr}(r,W)} \otimes W$ denote the universal
	subbundle.
	Given $\sigma \in K$, let $\widetilde{\sigma}$
	denote a choice of holomorphic section of $\mathscr S$
	defined on a neighborhood of $[K] \in \on{Gr}(r,W)$ with $\widetilde{\sigma}([K]) =
	\sigma$.
	Then, the map 
	\begin{align*}
		T_{\on{Gr}(r, W), [K]} &\to \Hom(K, W/K) \\
		u &\mapsto \left( \sigma \mapsto \frac{\partial }{\partial
		u} (\widetilde{\sigma}) \bmod K \right)
	\end{align*}
	is well-defined and independent of the choice of lift $\widetilde{\sigma}$.

	Choosing $b \in \mathscr B$, the derivative of the period map $dP_b'$ is then given as follows:
	send 
	$\sigma \in F^1 \mathscr H_b$ to the function which sends
	$v \in T_{\mathscr B,b}$ to $\frac{\partial }{\partial v} \widetilde{\sigma} \bmod F^1
	\mathscr H_b$,
	for $\widetilde{\sigma}$ a local holomorphic lift of $\sigma$ to $F^1
	\mathscr H$. Here $\frac{\partial }{\partial v}$ makes sense as we have
	chosen a flat trivialization of $\mathscr{H}$, but we may equivalently
	write this map as $$\sigma \mapsto (v \mapsto \nabla_{GM}(\widetilde{\sigma})(v)
	\bmod F^1\mathscr{H}_b),$$
	or equivalently 
	\begin{equation*}
	\sigma \mapsto (v \mapsto
	\overline{\nabla}(\widetilde{\sigma})(v)).\qedhere
\end{equation*}
\end{proof}

After introducing some notation, we next give an explicit computation of the Gauss-Manin connection $\nabla_{GM}$ on $\mathscr{H}$. 

\begin{notation}
	\label{notation:}
	Recall from \autoref{notation:period-map}, that $(\mathscr E, \nabla)$ denotes the Deligne canonical extension of
$(\mathbb V \otimes \mathscr O_{\mathscr C^\circ}, \on{id} \otimes d)$ to
$\mathscr C$.
	Let $\mathcal{A}^{i,j}_{\log \mathscr{D}}(\mathscr{E})$ be the sheaf of $C^\infty$ logarithmic $(i,j)$-forms valued in $\mathscr{E}$, and let $\mathcal{A}^n_{\log \mathscr{D}}(\mathscr{E})=\oplus_{i+j=n} \mathcal{A}^{i,j}_{\log \mathscr{D}}(\mathscr{E})$. 
	Then, $(\mathcal{A}^{i,\bullet}_{\log\mathscr{D}}(\mathscr{E}),
	\overline{\partial})$ is the Dolbeault complex of $\mathscr{E}\otimes
	\Omega^i_{\mathscr{C}}(\log \mathscr D)$. The complex
	$(\mathcal{A}^\bullet_{\log \mathscr{D}}(\mathscr{E}),
	\nabla+\overline{\partial})$ is the de Rham resolution of the unitary
	local system $\mathbb{V}$. We similarly define the relative variants  $\mathcal{A}^{i,j}_{\log \mathscr{D}, \mathscr{C}/\mathscr{B}}(\mathscr{E})$, $\mathcal{A}^{n}_{\log \mathscr{D}, \mathscr{C}/\mathscr{B}}(\mathscr{E})$.
\end{notation}
\begin{notation}
	\label{notation:}
	For $v$ a $C^\infty$ vector field on an open $U \subset \mathscr C$, tangent to $\mathscr{D}$, and
$w$ a $C^\infty$ section of $\mathcal{A}^i_{\log \mathscr{D}}(\mathscr{E})$ on $U$, the interior product
$\on{int}(v)(w) \in H^0(U, \mathcal{A}^{i-1}_{\log \mathscr D}(\mathscr E))$ is
defined by
$$\on{int}(v)(w)(X_1, \ldots, X_{i+j-1}) = w(v, X_1, \ldots, X_{i+j-1})$$ for any
$C^\infty$ vector fields $X_1, \ldots, X_{i+j-1}$ on $U$ tangent to $\mathscr{D}$.
\end{notation}

\begin{lemma}
	\label{lemma:connection-to-interior}
	Let $\sigma\in \mathcal A^1_{\log \mathscr{D}}(\mathscr{E})$ be a
	differential form such that for each $b'\in \mathscr{B}$,
	$\sigma|_{\mathscr{C}_{b'}}$ is closed, so that $\sigma$ represents a
	section $[\sigma]$ of $\mathscr{H}$.
Fix $b\in \mathscr{B}$ and a tangent vector $u\in T_{\mathscr{B}, b}$, and let
$v$ be a $C^\infty$ section of $T_\mathscr{C}|_{\mathscr{C}_b}$, tangent to
$\mathscr{D}$ with $u = \pi_*(v).$
	We have
	\begin{align*}
		\nabla_{GM}([\sigma])|_b(u)= [(\on{int}(v)((\nabla+\overline{\partial})
	(\sigma)|_{\mathscr C_b})]\in \mathscr{H}_b.
	\end{align*}
\end{lemma}
\begin{proof}
	This proof essentially follows \cite[Proposition
	9.14]{Voisin:hodgeTheory}.	
	After possibly shrinking $\mathscr B$, which will not alter the statement of this
	lemma, we may write $\mathscr C \simeq
	\mathscr C_b \times \mathscr B$ (as $C^\infty$ manifolds), by Ehresmann's theorem. 
	Let $t_1, \ldots, t_m$ be $C^\infty$  functions on $C_b \times
	\mathscr B$ pulled back from a system of local coordinates for $\mathscr B$.
	After possibly shrinking $\mathscr B$, we can write
	$\sigma$ as
	\begin{align*}
		\sigma = \Phi + \sum_{i=1}^m dt_i \otimes \psi_i, 
	\end{align*}
	where $\Phi$ a section of $\mathscr{A}^1_{\log \mathscr{D}}(\mathscr{E})$ independent of
	$dt_i$, 
	each
	$\psi_i$ is a $C^\infty$ section of $\mathscr{E}$,
	and $\Phi|_{\mathscr C_b} = \sigma|_{\mathscr C_b}$. 
	Since $\sigma$ is fiber-wise closed, applying $\nabla+\overline{\partial}$ gives
		\begin{align*}
		(\nabla+\overline{\partial}) \sigma = \sum_i dt_i
		\on{int}(\partial/\partial t_i)((\nabla+\overline{\partial})\Phi) - \sum_i dt_i \wedge (\nabla+\overline{\partial})
		\psi_i 
	\end{align*}
	and hence
	\begin{align*}
		\on{int}(\partial/\partial t_i) ((\nabla+\overline{\partial})\sigma)|_{\mathscr{C}_b} =
	\on{int}(\partial/\partial t_i)((\nabla+\overline{\partial})\Phi|_{\mathscr{C}_b}) - (\nabla+\overline{\partial})
		\psi_i|_{\mathscr{C}_b}. 
	\end{align*}
	
	Since $\Phi|_{\mathscr C_b} = \sigma|_{\mathscr C_b}$, 
	$$\on{int}(\partial/\partial
	t_i)((\nabla+\overline{\partial})\Phi|_{\mathscr{C}_b}) =
	\nabla_{GM}([\sigma])|_b(\pi_*(\partial/\partial t_i)),$$ by definition of the Gauss-Manin connection.	
	But as $(\nabla+\overline{\partial})
		\psi_i|_{\mathscr{C}_b}$ is exact, 
		$\on{int}(\partial/\partial t_i) ((\nabla+\overline{\partial})\sigma)|_{\mathscr{C}_b}$ and $\on{int}(\partial/\partial t_i)((\nabla+\overline{\partial})\Phi|_{\mathscr{C}_b})$ represent the same cohomology class. Thus 
\begin{align*}
	\mathscr{H}_b \ni [(\on{int}(\partial/\partial t_i)((\nabla+\overline{\partial})(\sigma)|_{\mathscr C_b})] =
		[(\on{int}(\partial/\partial t_i)((\nabla+\overline{\partial})(\Phi)|_{\mathscr C_b})] =
		\nabla_{GM}([\sigma])|_b(\pi_*(\partial/\partial t_i)).
\end{align*}

 Since
	$\sigma|_{\mathscr{C}_b}$ is closed, 
	the same result holds after
	replacing $\partial/\partial t_i$ with any vector field $v$ satisfying
	$\pi_*v=\pi_*(\partial/\partial t_i)$; 	
	as the $\pi_*(\partial/\partial t_i)$ form a
	basis for $T_{\mathscr{B},b}$, the proof is complete.
		\end{proof}

For the next result, we need to recall the Kodaira-Spencer map.
\begin{definition}
	\label{definition:kodaira-spencer}
	With notation as in \autoref{notation:period-map},
we have a short exact sequence
\begin{equation}
	\label{equation:tangent-exact-sequence}
	\begin{tikzcd}
		0 \ar {r} & T_{\mathscr C_b}(-\log \mathscr D) \ar {r} & T_{\mathscr C}(-
		\log \mathscr D)|_{\mathscr C_b} \ar {r}
	& \pi^* T_{\mathscr B}|_{\mathscr C_b} \ar {r} & 0 
\end{tikzcd}\end{equation}
where $T_{\mathscr{C}}(-\log\mathscr{D})$ is the dual of $\Omega^1_{\mathscr{C}}(\log \mathscr{D})$.
The {\em Kodaira-Spencer} map $$\kappa: T_{\mathscr{B}, b}\to H^1(T_{\mathscr{C}_b}(-\log\mathscr{D}))$$ sends 
a vector $v \in T_{\mathscr B,b} \simeq H^0(\mathscr C_b, \pi^* T_{\mathscr B}|_{\mathscr C_b})$ to its image in 
$H^1(T_{\mathscr{C}_b}(-\log\mathscr{D}))$ under the connecting homomorphism in the long exact sequence in cohomology arising from \eqref{equation:tangent-exact-sequence}. 
\end{definition}

\begin{remark}
	\label{remark:kodaira-spencer}
	Geometrically, the Kodaira-Spencer map sends a deformation of $(\mathscr C_b, \mathscr D_b)$ to the corresponding
cohomology class in $H^1(T_{\mathscr C_b}(-\mathscr D))$ (identifying $T_{\mathscr C_b}(-\mathscr D)$ with $T_{\mathscr{C}_b}(-\log\mathscr{D})$).
For an explanation of why the boundary map coincides with this description in
the case $\mathscr D = \emptyset$, see, \cite[\S9.1.2]{Voisin:hodgeTheory}.
In the case $\mathscr D \neq \emptyset$, an analogous argument goes through.
\end{remark}

\begin{proposition}
	\label{proposition:period-map}
	Keep notation as in \autoref{notation:period-map}.
	For $b \in \mathscr B$, set $C=\mathscr{C}_b$, $D=\mathscr{D}_b$, and $E=\mathscr{E}|_C$. We have a commuting diagram
	\begin{equation}
		\label{equation:period-to-cup}
		\begin{tikzcd} 
			H^0( C, E \otimes \Omega^1_{ C}(D)) \otimes
			T_{\mathscr B,b} \ar {r}{dP_b''} \ar {d}{1 \otimes \kappa} & H^1( C,  E) \\
			H^0( C,  E \otimes \Omega^1_{ C}( D)) \otimes
			H^1( C, T_{ C}(- D)) \ar {r}{\cup} & H^1( C,  E
				\otimes \Omega^1_{ C}( D) \otimes T_{ C}(- D)) \ar{u}{\alpha}
	\end{tikzcd}\end{equation}
	where $dP_b''$ is adjoint to the derivative of the period map $dP_b$, $\kappa$ is the Kodaira-Spencer map, $\alpha$ is the map induced by the pairing 
	$\Omega^1_{ C}( D) \otimes T_{ C}(- D) \to \mathscr O_{ C}$
	and $\cup$ is the cup product.
\end{proposition}
\begin{proof}
Let $\sigma\otimes u$ be an element of $H^0( C, E \otimes \Omega^1_{ C}(D)) \otimes
			T_{\mathscr B,b}$. By e.g.~\cite[Proposition
			9.22]{Voisin:hodgeTheory} (applied to the
			$\overline\partial$-Laplacian---here unitarity is used
		to define the adjoint of $\overline\partial$ and consequently
	the Laplacian), we may choose a lift $\widetilde\sigma$ of $\sigma$ to a $C^\infty$ section of $\mathscr{E}\otimes \Omega^1_{\mathscr{C}}(\log \mathscr{D})$, which is holomorphic on fibers.
	Combining \autoref{lemma:connection-and-period-map} and
	\autoref{lemma:connection-to-interior}, we find that $dP''(\sigma
	\otimes u)$
	can be expressed as $[((\on{int}(v)(\nabla+\overline{\partial})\widetilde{\sigma})|_{\mathscr C_b})^{0,1}]$ for
$v$ a $C^\infty$ vector field of type $(1,0)$ tangent to $\mathscr{D}$ and lifting $u$, as in \autoref{lemma:connection-to-interior}. 
As $\mathscr{C}$ is a relative curve, for degree reasons we have
$$[((\on{int}(v)(\nabla+\overline{\partial})\widetilde{\sigma}|_{\mathscr C_b}))^{0,1}]=[\on{int}(v)\overline{\partial}\widetilde{\sigma}|_{\mathscr C_b}].$$
Direct computation gives
\begin{align*}
	\overline{\partial} (\on{int}(v) (\widetilde{\sigma}|_{\mathscr{C}_b})) = 
	- \on{int}(v)\overline{\partial}\widetilde{\sigma}|_{\mathscr C_b}
	+ \on{int}(\overline{\partial} v)(\widetilde{\sigma}|_{\mathscr{C}_b}).
\end{align*}
Using the vanishing of 
the Dolbeault cohomology class $[\overline{\partial} (\on{int}(v)
(\widetilde{\sigma}))]$
we obtain 
\begin{align*}
	[\on{int}(v)(\overline{\partial}\widetilde{\sigma}|_{\mathscr C_b})]
	=
	[\on{int}(\overline{\partial}v)(\widetilde{\sigma}|_{\mathscr C_b})].
\end{align*}
Combining the above yields that $dP''_b(\sigma\otimes u) = 
[\on{int}(\overline{\partial}v)(\widetilde{\sigma}|_{\mathscr C_b})]$.
By definition of the Kodaira Spencer map 
$\overline{\partial}v$ is identified with $\kappa(u)$.
So
\begin{align*}
dP''_b(\sigma \otimes u) =
[\on{int}(\overline{\partial}v)(\widetilde{\sigma}|_{\mathscr C_b})]
=
[\alpha\left( \overline{\partial}v \cup \widetilde{\sigma}|_{\mathscr C_b}\right)]
=
[\alpha\left( \overline{\partial}v\cup \sigma \right)]
=
[\alpha\left( \kappa(u) \cup \sigma \right)]
\end{align*}
verifying the commutativity of
\eqref{equation:period-to-cup} as desired.
\end{proof}

\begin{lemma}
	\label{lemma:multiplication-map}
	Suppose $E$ is a vector bundle on a smooth curve $C$, and $D \subset C$
	is a divisor.
	Let $H^1(C, T_C(-D)) \otimes H^0(C, E \otimes \omega_C(D)) \to H^1(C, E)$
	denote the map corresponding to the Yoneda cup product composition $\alpha \circ \cup$ 
	from \autoref{proposition:period-map},
	induced by the natural pairing $E \otimes \omega_C(D) \otimes T_C(-D) \to E$.
This is adjoint to a map $\upsilon: H^0(C, E \otimes \omega_C(D)) \otimes
	H^1(C, E)^\vee \to H^1(C, T_C(-D))^\vee$.
	We have a commutative diagram
	\begin{equation}
		\label{equation:serre-dual-diagram}
		\begin{tikzcd} 
			H^0(C, E \otimes \omega_C(D)) \otimes H^0(C, E^\vee \otimes
			\omega_C)  \ar {r}{\mu}\ar{d}{\eta}  & H^0(C, \omega_C^{\otimes 2}(D))\ar{d}{\zeta}
\\
H^0(C, E \otimes \omega_C(D)) \otimes H^1(C, E)^\vee \ar
			{r}{\upsilon} & H^1(C, T_C(-D))^\vee ,
				\end{tikzcd}\end{equation}
	where the two vertical maps are induced by Serre duality, and the 
	map $\mu$ is the map on global sections induced by $$(E
	\otimes \omega_C(D)) \times (E^\vee \otimes \omega_C) \to (E \otimes E^\vee)
	\otimes (\omega_C(D) \otimes \omega_C) \to \omega_C^{\otimes 2}(D)$$
	where the first map is the tensor product and the second is obtained from the trace pairing $E\otimes E^\vee\to \mathscr{O}_C$.
\end{lemma}
\begin{proof}
	Recall that for $F$ a vector bundle on $C$, the Serre duality pairing
	between $H^0(C, F)$ and $H^1(C, F^\vee\otimes \omega_C)$ is obtained
	as the composition $H^0(C, F) \otimes H^1(C, F^\vee\otimes \omega_C) \to
	H^1(C, F^\vee\otimes F \otimes \omega_C) \to
	H^1(C, \omega_C) \to \mathbb C$, where
	the first map is induced by cup product, the second is induced by the
	pairing
	$F \otimes F^\vee \to \mathscr O$,
	and the third map is the trace
	map $\on{tr}: H^1(C, \omega_C) \simeq \mathbb C$
	\cite[30.3.15]{vakil:foundations-of-algebraic-geometry-2017}.

	Using the above description, we now check commutativity of
	\eqref{equation:serre-dual-diagram}.
	We use $\cup$ for the usual cup product, and $\overline{\cup}$ to denote
	the composition of the cup product and the pairing $E \otimes E^\vee \to
	\mathscr O$. 
		Choose $\alpha \in H^0(C, E \otimes \omega_C(D)), \beta \in
	H^0(C, E^\vee \otimes \omega_C)$; we will denote by $s$ a general element of $H^1(C, E)$ and by $t$ a general element of $H^1(C,
	T_C(-D))$.
	On the one hand $\mu(\alpha \otimes \beta) = \alpha \overline{\cup}\beta$, and so $\zeta(\mu(\alpha\otimes \beta))$ is given by $$t\mapsto \on{tr}((\alpha\overline{\cup}\beta)\cup t).$$
	On the other hand, using the above description of Serre duality,
	\begin{align*}
		\upsilon(\eta(\alpha \otimes \beta))
		&= \upsilon(\alpha \otimes (s \mapsto \on{tr}(\beta
		\overline{\cup}s)))
		\\
		&= \left(t \mapsto \on{tr}(\beta \overline{\cup} (\alpha \cup
			t))\right) \\
			&= \left( t \mapsto \on{tr}( (\beta \overline{\cup} \alpha) \cup
			t )\right)  
	\end{align*}
	and so we obtain the desired commutativity (using that $\alpha \overline{\cup} \beta = \beta \overline{\cup} \alpha$).
	\end{proof}

\subsection{Proof of \autoref{theorem:period-map-computation}}
\begin{proof}
\label{subsection:proof-period-map}
	The final statement claiming adjointness follows from
	\autoref{lemma:connection-and-period-map}.
	The map $\mu$ in \autoref{lemma:multiplication-map}, 
	is another name for the composition	
	$\on{tr} \circ \otimes$ in \eqref{equation:multiplication-to-period},
	and by \autoref{lemma:multiplication-map},
	$\mu$ is identified with the map labeled $\upsilon$ there, upon applying
	Serre duality.
	Moreover, $\upsilon$ is by definition adjoint to the composition $\alpha \circ \cup$
	in \autoref{proposition:period-map}.
	The result now follows by identifying the Kodaira-Spencer map $\kappa$
	of \autoref{proposition:period-map} as Serre dual to the pullback map
	$c_b^*$ in the statement of \autoref{theorem:period-map-computation};
	this identification follows from \autoref{remark:kodaira-spencer}.
\end{proof}

\bibliographystyle{alpha}
\bibliography{bibliography-mcg-hodge-theory}

\begin{thebibliography}{BGMW17}

\bibitem[AHL19]{arapura2019vanishing}
Donu Arapura, Feng Hao, and Hongshan Li.
\newblock Vanishing theorems for parabolic {H}iggs bundles.
\newblock {\em Mathematical Research Letters}, 26(5):1251--1279, 2019.

\bibitem[AK02]{andreev2002transformations}
FV~Andreev and AV~Kitaev.
\newblock Transformations ${RS}_4^2 (3)$ of the ranks $\leq 4$ and algebraic
  solutions of the sixth {P}ainlev{\'e} equation.
\newblock {\em Communications in Mathematical Physics}, 228(1):151--176, 2002.

\bibitem[And74]{anderson:exactness-properties}
Michael~P. Anderson.
\newblock Exactness properties of profinite completion functors.
\newblock {\em Topology}, 13:229--239, 1974.

\bibitem[AS16]{aramayona2016rigidity}
Javier Aramayona and Juan Souto.
\newblock Rigidity phenomena in the mapping class group.
\newblock {\em Handbook of Teichm{\"u}ller theory}, 6:131--165, 2016.

\bibitem[BGMW17]{biswas2017surface}
Indranil Biswas, Subhojoy Gupta, Mahan Mj, and Junho~Peter Whang.
\newblock Surface group representations in ${SL}_2(\mathbb{C})$ with finite
  mapping class orbits.
\newblock {\em arXiv preprint arXiv:1707.00071v4}, 2017.

\bibitem[BGS16]{bourgainGS:announcement}
Jean Bourgain, Alexander Gamburd, and Peter Sarnak.
\newblock Markoff triples and strong approximation.
\newblock {\em C. R. Math. Acad. Sci. Paris}, 354(2):131--135, 2016.

\bibitem[BKMS18]{biswas2018representations}
Indranil Biswas, Thomas Koberda, Mahan Mj, and Ramanujan Santharoubane.
\newblock Representations of surface groups with finite mapping class group
  orbits.
\newblock {\em New York J. Math}, 24:241--250, 2018.

\bibitem[BKP19]{baumeister-kielak-pierro}
Barbara Baumeister, Dawid Kielak, and Emilio Pierro.
\newblock On the smallest non-abelian quotient of {${\rm Aut}(F_n)$}.
\newblock {\em Proc. Lond. Math. Soc. (3)}, 118(6):1547--1591, 2019.

\bibitem[Boa05]{boalch2005klein}
Philip Boalch.
\newblock From {K}lein to {P}ainlev\'{e} via {F}ourier, {L}aplace and {J}imbo.
\newblock {\em Proc. London Math. Soc. (3)}, 90(1):167--208, 2005.

\bibitem[Boa06]{boalch2006fifty}
Philip Boalch.
\newblock The fifty-two icosahedral solutions to {P}ainlev\'{e} {VI}.
\newblock {\em J. Reine Angew. Math.}, 596:183--214, 2006.

\bibitem[Boa07a]{boalch2007higher}
Philip Boalch.
\newblock Higher genus icosahedral {P}ainlev\'{e} curves.
\newblock {\em Funkcial. Ekvac.}, 50(1):19--32, 2007.

\bibitem[Boa07b]{boalchsome}
Philip Boalch.
\newblock Some explicit solutions to the {R}iemann-{H}ilbert problem.
\newblock In {\em Differential equations and quantum groups}, volume~9 of {\em
  IRMA Lect. Math. Theor. Phys.}, pages 85--112. Eur. Math. Soc., Z\"{u}rich,
  2007.

\bibitem[Boa10]{boalch2007towards}
Philip Boalch.
\newblock Towards a non-linear {S}chwarz's list.
\newblock In {\em The many facets of geometry}, pages 210--236. Oxford Univ.
  Press, Oxford, 2010.

\bibitem[Bur05]{burnside1905criteria}
William Burnside.
\newblock On criteria for the finiteness of the order of a group of linear
  substitutions.
\newblock {\em Proceedings of the London Mathematical Society}, 2(1):435--440,
  1905.

\bibitem[BY96]{bodenY:moduli-spaces-of-parabolic-higgs-bundles}
Hans~U. Boden and K\^{o}ji Yokogawa.
\newblock Moduli spaces of parabolic {H}iggs bundles and parabolic {$K(D)$}
  pairs over smooth curves. {I}.
\newblock {\em Internat. J. Math.}, 7(5):573--598, 1996.

\bibitem[CH19]{CH:isomonodromic}
Ga{\"e}l Cousin and Viktoria Heu.
\newblock Algebraic isomonodromic deformations and the mapping class group.
\newblock {\em Journal of the Institute of Mathematics of Jussieu}, pages
  1--49, 2019.

\bibitem[CL09]{cantat2009holomorphic}
Serge Cantat and Frank Loray.
\newblock Dynamics on character varieties and {M}algrange irreducibility of
  {P}ainlev\'{e} {VI} equation.
\newblock {\em Ann. Inst. Fourier (Grenoble)}, 59(7):2927--2978, 2009.

\bibitem[CM18]{calligaris2018finite}
Pierpaolo Calligaris and Marta Mazzocco.
\newblock Finite orbits of the pure braid group on the monodromy of the
  2-variable {G}arnier system.
\newblock {\em Journal of Integrable Systems}, 3(1):xyy005, 2018.

\bibitem[CO05]{cornelissen2005problems}
Gunther Cornelissen and Frans Oort.
\newblock Problems from the workshop on automorphisms of curves ({L}eiden,
  {A}ugust, 2004).
\newblock {\em Rendiconti del Seminario Matematico della Universit{\`a} di
  Padova}, 113:129--177, 2005.

\bibitem[CR66]{curtis1966representation}
Charles~W Curtis and Irving Reiner.
\newblock {\em Representation theory of finite groups and associative
  algebras}, volume 356.
\newblock American Mathematical Soc., 1966.

\bibitem[Deb01]{debarre:higher-dimensional}
Olivier Debarre.
\newblock {\em Higher-dimensional algebraic geometry}.
\newblock Universitext. Springer-Verlag, New York, 2001.

\bibitem[Del70]{deligne:regular-singular}
Pierre Deligne.
\newblock {\em \'{E}quations diff\'{e}rentielles \`a points singuliers
  r\'{e}guliers}.
\newblock Lecture Notes in Mathematics, Vol. 163. Springer-Verlag, Berlin-New
  York, 1970.

\bibitem[Dia13]{diarra2013construction}
Karamoko Diarra.
\newblock Construction et classification de certaines solutions alg{\'e}briques
  des syst{\`e}mes de {G}arnier.
\newblock {\em Bulletin of the Brazilian Mathematical Society, New Series},
  44(1):129--154, 2013.

\bibitem[dJ01]{de2001conjecture}
Aise~Johan de~Jong.
\newblock A conjecture on arithmetic fundamental groups.
\newblock {\em Israel Journal of Mathematics}, 121(1):61--84, 2001.

\bibitem[DL15]{diarra2015ramified}
Karamoko Diarra and Frank Loray.
\newblock Ramified covers and tame isomonodromic solutions on curves.
\newblock {\em Tranactions of the Moscow Mathematical Society}, 76(2):219--236,
  2015.

\bibitem[DM00]{dubrovin2000monodromy}
B.~Dubrovin and M.~Mazzocco.
\newblock Monodromy of certain {P}ainlev\'{e}-{VI} transcendents and reflection
  groups.
\newblock {\em Invent. Math.}, 141(1):55--147, 2000.

\bibitem[Dor01]{doran:isomonodromic}
Charles~F. Doran.
\newblock Algebraic and geometric isomonodromic deformations.
\newblock {\em J. Differential Geom.}, 59(1):33--85, 2001.

\bibitem[Dub96]{dubrovin1996geometry}
Boris Dubrovin.
\newblock Geometry of 2d topological field theories.
\newblock {\em Integrable systems and quantum groups}, pages 120--348, 1996.

\bibitem[EG18]{esnault2018cohomologically}
H{\'e}l{\`e}ne Esnault and Michael Groechenig.
\newblock Cohomologically rigid local systems and integrality.
\newblock {\em Selecta Mathematica}, 24(5):4279--4292, 2018.

\bibitem[EH16]{eisenbudH:3264-&-all-that}
David Eisenbud and Joe Harris.
\newblock {\em 3264 and all that---a second course in algebraic geometry}.
\newblock Cambridge University Press, Cambridge, 2016.

\bibitem[EK20]{esnault2020arithmetic}
H{\'e}l{\`e}ne Esnault and Moritz Kerz.
\newblock Arithmetic subspaces of moduli spaces of rank one local systems.
\newblock {\em Cambridge Journal of Mathematics}, 8(3):453--478, 2020.

\bibitem[EM16]{eardley2016inverse}
Timothy Eardley and Jayanta Manoharmayum.
\newblock The inverse deformation problem.
\newblock {\em Compositio Mathematica}, 152(8):1725--1739, 2016.

\bibitem[FH91]{fultonH:representation-theory}
William Fulton and Joe Harris.
\newblock {\em Representation theory}, volume 129 of {\em Graduate Texts in
  Mathematics}.
\newblock Springer-Verlag, New York, 1991.
\newblock A first course, Readings in Mathematics.

\bibitem[FH13]{franksH:triviality}
John Franks and Michael Handel.
\newblock Triviality of some representations of {${\rm MCG}(S_g)$} in
  {$GL(n,\Bbb C)$}, {${\rm Diff}(S^2)$} and {${\rm Homeo}(\Bbb T^2)$}.
\newblock {\em Proc. Amer. Math. Soc.}, 141(9):2951--2962, 2013.

\bibitem[FH17]{farb2017moving}
Benson Farb and Sebastian Hensel.
\newblock Moving homology classes in finite covers of graphs.
\newblock {\em Israel Journal of Mathematics}, 220(2):605--615, 2017.

\bibitem[FLM01]{farbLM:rank-1-phenomena}
Benson Farb, Alexander Lubotzky, and Yair Minsky.
\newblock Rank-1 phenomena for mapping class groups.
\newblock {\em Duke Math. J.}, 106(3):581--597, 2001.

\bibitem[FM12]{farbM:a-primer}
Benson Farb and Dan Margalit.
\newblock {\em A primer on mapping class groups}, volume~49 of {\em Princeton
  Mathematical Series}.
\newblock Princeton University Press, Princeton, NJ, 2012.

\bibitem[Fun11]{funar:two-questions}
Louis Funar.
\newblock Two questions on mapping class groups.
\newblock {\em Proc. Amer. Math. Soc.}, 139(1):375--382, 2011.

\bibitem[Gai07]{gaitsgory2007jong}
Dennis Gaitsgory.
\newblock On de {J}ong's conjecture.
\newblock {\em Israel Journal of Mathematics}, 157(1):155--191, 2007.

\bibitem[Gam10]{gambier:on-the-differential-equations-of-second-order}
B.~Gambier.
\newblock Sur les \'{e}quations diff\'{e}rentielles du second ordre et du
  premier degr\'{e} dont l'int\'{e}grale g\'{e}n\'{e}rale est a points
  critiques fixes.
\newblock {\em Acta Math.}, 33(1):1--55, 1910.

\bibitem[Gir16a]{girand2016algebraic}
Arnaud Girand.
\newblock Algebraic isomonodromic deformations of the five punctured sphere
  arising from quintic plane curves.
\newblock {\em arXiv preprint arXiv:1612.01168v1}, 2016.

\bibitem[Gir16b]{girand:a-new-two-parameter-family}
Arnaud Girand.
\newblock A new two-parameter family of isomonodromic deformations over the
  five punctured sphere.
\newblock {\em Bull. Soc. Math. France}, 144(2):339--368, 2016.

\bibitem[GL09]{grunewald2009linear}
Fritz Grunewald and Alexander Lubotzky.
\newblock Linear representations of the automorphism group of a free group.
\newblock {\em Geometric and Functional Analysis}, 18(5):1564--1608, 2009.

\bibitem[GLLM15]{grunewald2015arithmetic}
Fritz Grunewald, Michael Larsen, Alexander Lubotzky, and Justin Malestein.
\newblock Arithmetic quotients of the mapping class group.
\newblock {\em Geometric and Functional Analysis}, 25(5):1493--1542, 2015.

\bibitem[Gol06]{goldmanmapping}
William~M. Goldman.
\newblock Mapping class group dynamics on surface group representations.
\newblock In {\em Problems on mapping class groups and related topics},
  volume~74 of {\em Proc. Sympos. Pure Math.}, pages 189--214. Amer. Math.
  Soc., Providence, RI, 2006.

\bibitem[GR71]{sga1}
A.~Grothendieck and M.~Raynaud.
\newblock {\em Rev\^etements \'etales et groupe fondamental}.
\newblock Springer-Verlag, Berlin-New York, 1971.
\newblock S{\'e}minaire de G{\'e}om{\'e}trie Alg{\'e}brique du Bois Marie
  1960--1961 (SGA 1).

\bibitem[Hit95]{hitchin1996poncelet}
N.~J. Hitchin.
\newblock Poncelet polygons and the {P}ainlev\'{e} equations.
\newblock In {\em Geometry and analysis ({B}ombay, 1992)}, pages 151--185. Tata
  Inst. Fund. Res., Bombay, 1995.

\bibitem[Hit03]{hitchin2003lecture}
Nigel Hitchin.
\newblock A lecture on the octahedron.
\newblock {\em Bull. London Math. Soc.}, 35(5):577--600, 2003.

\bibitem[Iva06]{Ivanov:problems}
Nikolai~V. Ivanov.
\newblock Fifteen problems about the mapping class groups.
\newblock In {\em Problems on mapping class groups and related topics},
  volume~74 of {\em Proc. Sympos. Pure Math.}, pages 71--80. Amer. Math. Soc.,
  Providence, RI, 2006.

\bibitem[Jor78]{jordan1878memoire}
M~Camille Jordan.
\newblock M{\'e}moire sur les {\'e}quations diff{\'e}rentielles lin{\'e}aires
  {\`a} int{\'e}grale alg{\'e}brique.
\newblock {\em Journal f{\"u}r die reine und angewandte Mathematik (Crelles
  Journal)}, 1878(84):89--215, 1878.

\bibitem[Kas15]{kasahara2015visualization}
Yasushi Kasahara.
\newblock On visualization of the linearity problem for mapping class groups of
  surfaces.
\newblock {\em Geometriae Dedicata}, 176(1):295--304, 2015.

\bibitem[Kat72]{katz:pcurvature-and-hodge}
Nicholas~M Katz.
\newblock Algebraic solutions of differential equations (p-curvature and the
  {H}odge filtration).
\newblock {\em Inventiones mathematicae}, 18(1-2):1--118, 1972.

\bibitem[Kat96]{katz2016rigid}
Nicholas~M. Katz.
\newblock {\em Rigid local systems}, volume 139 of {\em Annals of Mathematics
  Studies}.
\newblock Princeton University Press, Princeton, NJ, 1996.

\bibitem[Kir97]{Kirby:problems}
Problems in low-dimensional topology.
\newblock In Rob Kirby, editor, {\em Geometric topology ({A}thens, {GA},
  1993)}, volume~2 of {\em AMS/IP Stud. Adv. Math.}, pages 35--473. Amer. Math.
  Soc., Providence, RI, 1997.

\bibitem[Kit05]{kitaev2006grothendieck}
A.~V. Kitaev.
\newblock Grothendieck's dessins d'enfants, their deformations, and algebraic
  solutions of the sixth {P}ainlev\'{e} and {G}auss hypergeometric equations.
\newblock {\em Algebra i Analiz}, 17(1):224--275, 2005.

\bibitem[Kit06]{kitaev2005remarks}
Alexander~V. Kitaev.
\newblock Remarks towards a classification of {$RS_4^2(3)$}-transformations and
  algebraic solutions of the sixth {P}ainlev\'{e} equation.
\newblock 14:199--227, 2006.

\bibitem[Kor02]{korkmaz:low-dimensional-linear-representations}
Mustafa Korkmaz.
\newblock Low-dimensional homology groups of mapping class groups: a survey.
\newblock {\em Turkish J. Math.}, 26(1):101--114, 2002.

\bibitem[KP20a]{kielak-pierro}
Dawid Kielak and Emilio Pierro.
\newblock On the smallest non-trivial quotients of mapping class groups.
\newblock {\em Groups Geom. Dyn.}, 14(2):489--512, 2020.

\bibitem[KP20b]{klevdalP:g-rigid-local-systems-are-integral}
Christian Klevdal and Stefan Patrikis.
\newblock G-rigid local systems are integral.
\newblock {\em arXiv preprint arXiv:2009.07350v2}, 2020.

\bibitem[Kra00]{krammer2000braid}
Daan Krammer.
\newblock The braid group {$B_4$} is linear.
\newblock {\em Invent. Math.}, 142(3):451--486, 2000.

\bibitem[KS16]{koberda2016quotients}
Thomas Koberda and Ramanujan Santharoubane.
\newblock Quotients of surface groups and homology of finite covers via quantum
  representations.
\newblock {\em Inventiones mathematicae}, 206(2):269--292, 2016.

\bibitem[KS18a]{koberda2018irreducibility}
Thomas Koberda and Ramanujan Santharoubane.
\newblock Irreducibility of quantum representations of mapping class groups
  with boundary.
\newblock {\em Quantum Topology}, 9(4):633--641, 2018.

\bibitem[KS18b]{koberda2018representation}
Thomas Koberda and Ramanujan Santharoubane.
\newblock A representation theoretic characterization of simple closed curves
  on a surface.
\newblock {\em Mathematical Research Letters}, 25(5):1485--1496, 2018.

\bibitem[Kup11]{kuperberg2011denseness}
Greg Kuperberg.
\newblock Denseness and {Z}ariski denseness of {J}ones braid representations.
\newblock {\em Geom. Topol.}, 15(1):11--39, 2011.

\bibitem[Lit21]{litt2021arithmetic}
Daniel Litt.
\newblock Arithmetic representations of fundamental groups, {II}: {F}initeness.
\newblock {\em Duke Math. J.}, 170(8):1851--1897, 2021.

\bibitem[LL19]{lawrence2019representations}
Brian Lawrence and Daniel Litt.
\newblock Representations of surface groups with universally finite mapping
  class group orbit.
\newblock {\em arXiv preprint arXiv:1907.03941v3}, 2019.

\bibitem[LL22]{LL:geometric-local-systems}
Aaron Landesman and Daniel Litt.
\newblock Geometric local systems on very general curves and isomonodromy.
\newblock {\em arXiv preprint arXiv:2202.00039v2}, 2022.

\bibitem[LLSS20]{li2020surface}
Wanlin Li, Daniel Litt, Nick Salter, and Padmavathi Srinivasan.
\newblock Surface bundles and the section conjecture.
\newblock {\em arXiv preprint arXiv:2010.07331v1}, 2020.

\bibitem[Loo97]{looijenga:prym-representations}
Eduard Looijenga.
\newblock Prym representations of mapping class groups.
\newblock {\em Geom. Dedicata}, 64(1):69--83, 1997.

\bibitem[Loo15]{looiejenga:AG}
Eduard Looijenga.
\newblock Some algebraic geometry related to the mapping class group.
\newblock {\em Oberwolfach Reports}, 2015.

\bibitem[Loo21]{looijenga2021arithmetic}
Eduard Looijenga.
\newblock Arithmetic representations of mapping class groups.
\newblock {\em arXiv preprint arXiv:2108.12791v1}, 2021.

\bibitem[LT14]{lisovyy2014algebraic}
Oleg Lisovyy and Yuriy Tykhyy.
\newblock Algebraic solutions of the sixth {P}ainlev\'{e} equation.
\newblock {\em J. Geom. Phys.}, 85:124--163, 2014.

\bibitem[Mah99]{mahoux:itroduction-to-the-theory-of-isomonodromic-deformations}
Gilbert Mahoux.
\newblock Introduction to the theory of isomonodromic deformations of linear
  ordinary differential equations with rational coefficients.
\newblock In {\em The {P}ainlev\'{e} property}, CRM Ser. Math. Phys., pages
  35--76. Springer, New York, 1999.

\bibitem[Mar]{markovic}
Vladimir Markovi{\'c}.
\newblock Unramified correspondences and virtual properties of mapping class
  groups.
\newblock \url{http://people.maths.ox.ac.uk/~markovic/M-mod.pdf}.

\bibitem[Maz97]{mazur1997introduction}
Barry Mazur.
\newblock An introduction to the deformation theory of {G}alois
  representations.
\newblock {\em Modular forms and {F}ermat's last theorem}, pages 243--311,
  1997.

\bibitem[Mic06]{michel-hurwitz}
J.~Michel.
\newblock Hurwitz action on tuples of {E}uclidean reflections.
\newblock {\em J. Algebra}, 295(1):289--292, 2006.

\bibitem[MKS04]{magnusKS:combinatorial-group-theory}
Wilhelm Magnus, Abraham Karrass, and Donald Solitar.
\newblock {\em Combinatorial group theory}.
\newblock Dover Publications, Inc., Mineola, NY, second edition, 2004.
\newblock Presentations of groups in terms of generators and relations.

\bibitem[Moc06]{mochizuki2006kobayashi}
Takuro Mochizuki.
\newblock Kobayashi-{H}itchin correspondence for tame harmonic bundles and an
  application.
\newblock {\em Ast\'{e}risque}, (309):viii+117, 2006.

\bibitem[Mor89]{morita1989families}
Shigeyuki Morita.
\newblock Families of {J}acobian manifolds and characteristic classes of
  surface bundles. {I}.
\newblock {\em Ann. Inst. Fourier (Grenoble)}, 39(3):777--810, 1989.

\bibitem[Mor93]{morita1993extension}
Shigeyuki Morita.
\newblock The extension of {J}ohnson's homomorphism from the {T}orelli group to
  the mapping class group.
\newblock {\em Invent. Math.}, 111(1):197--224, 1993.

\bibitem[MS80]{mehta1980moduli}
Vikram~Bhagvandas Mehta and Conjeevaram~Srirangachari Seshadri.
\newblock Moduli of vector bundles on curves with parabolic structures.
\newblock {\em Mathematische Annalen}, 248(3):205--239, 1980.

\bibitem[MT]{markovic2}
Vladimir Markovi{\'c} and Ognjen To{\v s}i{\'c}.
\newblock The second variation of the {H}odge norm and higher {P}rym
  representations.
\newblock \url{https://people.maths.ox.ac.uk/~markovic/MT-hodge.pdf}.

\bibitem[Nie24]{nielsen:die-isomorphismengruppe}
Jakob Nielsen.
\newblock Die {I}somorphismengruppe der freien {G}ruppen.
\newblock {\em Math. Ann.}, 91(3-4):169--209, 1924.

\bibitem[Pai02]{painleve:on-the-differential-equations-of-sectond-order}
P.~Painlev\'{e}.
\newblock Sur les \'{e}quations diff\'{e}rentielles du second ordre et d'ordre
  sup\'{e}rieur dont l'int\'{e}grale g\'{e}n\'{e}rale est uniforme.
\newblock {\em Acta Math.}, 25(1):1--85, 1902.

\bibitem[PdJ95]{pikaart-dejong}
M.~Pikaart and A.~J. de~Jong.
\newblock Moduli of curves with non-abelian level structure.
\newblock In {\em The moduli space of curves ({T}exel {I}sland, 1994)}, volume
  129 of {\em Progr. Math.}, pages 483--509. Birkh\"{a}user Boston, Boston, MA,
  1995.

\bibitem[Pet20]{petrov2020geometrically}
Alexander Petrov.
\newblock Geometrically irreducible $ p $-adic local systems are de {R}ham up
  to a twist.
\newblock {\em arXiv preprint arXiv:2012.13372v3}, 2020.

\bibitem[PR00]{potapchik2000low}
A.~Potapchik and A.~Rapinchuk.
\newblock Low-dimensional linear representations of {${\rm Aut}\,F_n,\ n\geq
  3$}.
\newblock {\em Trans. Amer. Math. Soc.}, 352(3):1437--1451, 2000.

\bibitem[PS08]{peters2008mixed}
Chris A.~M. Peters and Joseph H.~M. Steenbrink.
\newblock {\em Mixed {H}odge structures}, volume~52 of {\em Ergebnisse der
  Mathematik und ihrer Grenzgebiete. 3. Folge. A Series of Modern Surveys in
  Mathematics [Results in Mathematics and Related Areas. 3rd Series. A Series
  of Modern Surveys in Mathematics]}.
\newblock Springer-Verlag, Berlin, 2008.

\bibitem[PW13]{putmanW:abelian-quotients}
Andrew Putman and Ben Wieland.
\newblock Abelian quotients of subgroups of the mappings class group and higher
  {P}rym representations.
\newblock {\em J. Lond. Math. Soc. (2)}, 88(1):79--96, 2013.

\bibitem[PX00]{previte2000topological}
Joseph~P. Previte and Eugene~Z. Xia.
\newblock Topological dynamics on moduli spaces. {I}.
\newblock {\em Pacific J. Math.}, 193(2):397--417, 2000.

\bibitem[PX02]{previte2002topological}
Joseph~P. Previte and Eugene~Z. Xia.
\newblock Topological dynamics on moduli spaces. {II}.
\newblock {\em Trans. Amer. Math. Soc.}, 354(6):2475--2494, 2002.

\bibitem[Rie57]{riemann1857beitrage}
Bernhard Riemann.
\newblock {\em Beitr{\"a}ge zur Theorie der durch die Gauss' sche Reihe F
  ([alpha],[beta],[gamma],[chi]) darstellbaren Functionen}.
\newblock Verlag der Dieterichschen Buchhandlung, 1857.

\bibitem[Sai88]{saito1988modules}
Morihiko Saito.
\newblock Modules de {H}odge polarisables.
\newblock {\em Publ. Res. Inst. Math. Sci.}, 24(6):849--995 (1989), 1988.

\bibitem[Sai90]{saito1990mixed}
Morihiko Saito.
\newblock Mixed {H}odge modules.
\newblock {\em Publ. Res. Inst. Math. Sci.}, 26(2):221--333, 1990.

\bibitem[Sch73]{schwarz1873ueber}
H.~A. Schwarz.
\newblock Ueber diejenigen {F}\"{a}lle, in welchen die {G}aussische
  hypergeometrische {R}eihe eine algebraische {F}unction ihres vierten
  {E}lementes darstellt.
\newblock {\em J. Reine Angew. Math.}, 75:292--335, 1873.

\bibitem[Sch19]{schnell2014overview}
Christian Schnell.
\newblock An overview of {M}orihiko {S}aito's theory of mixed {H}odge modules.
\newblock In {\em Representation theory, automorphic forms \& complex
  geometry}, pages 27--80. Int. Press, Somerville, MA, [2019] \copyright 2019.

\bibitem[SGA73]{SGA4}
{\em Th\'{e}orie des topos et cohomologie \'{e}tale des sch\'{e}mas. {T}ome 3}.
\newblock Lecture Notes in Mathematics, Vol. 305. Springer-Verlag, Berlin-New
  York, 1973.
\newblock S\'{e}minaire de G\'{e}om\'{e}trie Alg\'{e}brique du Bois-Marie
  1963--1964 (SGA 4), Dirig\'{e} par M. Artin, A. Grothendieck et J. L.
  Verdier. Avec la collaboration de P. Deligne et B. Saint-Donat.

\bibitem[Sik12]{sikora2012character}
Adam Sikora.
\newblock Character varieties.
\newblock {\em Transactions of the American Mathematical Society},
  364(10):5173--5208, 2012.

\bibitem[Sim90]{simpson1990harmonic}
Carlos~T Simpson.
\newblock Harmonic bundles on noncompact curves.
\newblock {\em Journal of the American Mathematical Society}, 3(3):713--770,
  1990.

\bibitem[Sim92]{simpson1992higgs}
Carlos~T Simpson.
\newblock Higgs bundles and local systems.
\newblock {\em Publications Math{\'e}matiques de l'IH{\'E}S}, 75:5--95, 1992.

\bibitem[Sin10]{sinz:thesis}
Justin~Conrad Sinz.
\newblock {\em Some results in representation dynamics}.
\newblock ProQuest LLC, Ann Arbor, MI, 2010.
\newblock Thesis (Ph.D.)--The University of Chicago.

\bibitem[{Sta}]{stacks-project}
The {Stacks Project Authors}.
\newblock {\itshape Stacks Project}.
\newblock \url{http://stacks.math.columbia.edu}.

\bibitem[SZ85]{steenbrink1985variation}
Joseph Steenbrink and Steven Zucker.
\newblock Variation of mixed {H}odge structure. {I}.
\newblock {\em Invent. Math.}, 80(3):489--542, 1985.

\bibitem[Tim87]{timmerscheidt:mixed-hodge-structure-for-unitary}
Klaus Timmerscheidt.
\newblock Mixed {H}odge theory for unitary local systems.
\newblock {\em J. Reine Angew. Math.}, 379:152--171, 1987.

\bibitem[Vak]{vakil:foundations-of-algebraic-geometry-2017}
Ravi Vakil.
\newblock {\itshape MATH 216: Foundations of Algebraic Geometry}.
\newblock Available at \url{http://math.stanford.edu/~vakil/216blog/} November
  18, 2017 version.

\bibitem[Voi07]{Voisin:hodgeTheory}
Claire Voisin.
\newblock {\em Hodge theory and complex algebraic geometry. {I}}, volume~76 of
  {\em Cambridge Studies in Advanced Mathematics}.
\newblock Cambridge University Press, Cambridge, english edition, 2007.
\newblock Translated from the French by Leila Schneps.

\bibitem[Wat89]{waterhouse:two-generators}
William~C. Waterhouse.
\newblock Two generators for the general linear groups over finite fields.
\newblock {\em Linear and Multilinear Algebra}, 24(4):227--230, 1989.

\bibitem[Wew98]{wewers:thesis}
Stefan Wewers.
\newblock {\em Construction of Hurwitz spaces}.
\newblock IEM, 1998.

\end{thebibliography}

\end{document}